\documentclass[12pt,a4paper]{amsart}
\usepackage[thinlines]{easytable}
\usepackage{array}
\usepackage[shortlabels]{enumitem}
\usepackage{upgreek}
\usepackage{setspace}
\usepackage{amsmath,amsfonts,amssymb,amscd,amsthm,comment,cite,float}
\usepackage{hyperref}
\usepackage{lipsum}       
\usepackage{changepage}
\usepackage[fleqn]{mathtools}

\DeclareMathSymbol{\shortminus}{\mathbin}{AMSa}{"39}
\DeclareMathSymbol{\shortplus}{\mathbin}{AMSa}{"2B}

\usepackage{color, soul}
\definecolor{gr}{rgb}{0.7, 1, 0.7}
\definecolor{rr}{rgb}{1, 0.7, 0.7}

\usepackage{scalerel,stackengine}
\stackMath
\newcommand\reallywidehat[1]{%
\savestack{\tmpbox}{\stretchto{%
  \scaleto{%
    \scalerel*[\widthof{\ensuremath{#1}}]{\kern-.6pt\bigwedge\kern-.6pt}%
    {\rule[-\textheight/2]{1ex}{\textheight}}
  }{\textheight}%
}{0.5ex}}%
\stackon[1pt]{#1}{\tmpbox}%
}
\usepackage{latexsym}
\usepackage{graphicx}

\theoremstyle{plain} 
\newtheorem{theorem}{Theorem}[section]

\newtheorem*{thmA}{Theorem A}
\newtheorem*{thmB}{Theorem B}
\newtheorem*{apthm}{Theorem: Trivial a-priori bounds}
\newtheorem*{napthm}{Theorem: Nontrivial a-priori bounds}
\newtheorem*{napthmLt}{Theorem: Nontrivial a-priori bounds in $L^t$}

\newtheorem*{FKRW}{Frechet-Kolmogorov-Riesz-Weil Compactness Theorem}
\newtheorem*{thmYR}{Young's Convolution Inequality for Weighted $R^p$-spaces}

\newtheorem{lemma}[theorem]{Lemma}
\newtheorem{corollary}[theorem]{Corollary}

\newtheorem{proposition}[theorem]{Proposition}

\theoremstyle{definition} 
\newtheorem{definition}[theorem]{Definition}
\theoremstyle{remark} 
\newtheorem{remark}[theorem]{Remark}

\newcolumntype{C}[1]{>{\centering\arraybackslash$}p{#1}<{$}}

\renewcommand{\Im}{\,\mathrm{Im}\,}

\newcommand{\ignore}[1]{}

\renewcommand{\Im}{\operatorname{Im}}

\usepackage{bm}

\newcommand{\cC}{{\mathcal C}}

\newcommand{\cE}{{\mathcal E}}
\newcommand{\cF}{{\mathcal F}}

\newcommand{\cI}{{\mathcal I}}

\newcommand{\cM}{{\mathcal M}}
\newcommand{\cN}{{\mathcal N}}

\newcommand{\cP}{{\mathcal P}}

\newcommand{\cR}{{\mathcal R}}

\newcommand{\cT}{{\mathcal T}}

\newcommand{\cK}{{\mathcal K}}

\newcommand{\bL}{{\mathbf L}}
\newcommand{\fL}{{\mathfrak L}}
\newcommand{\bPL}{{\mathbf{PL}}}

\newcommand{\fRL}{{\mathfrak{RL}}}
\newcommand{\bR}{{\mathbf{R}}}
\newcommand{\fR}{{\mathfrak{R}}}
\newcommand{\bH}{{\mathbf{H}}}
\newcommand{\bW}{{\mathbf{W}}}

\newcommand{\RR}{{\mathbb R}}

\newcommand{\CC}{{\mathbb C}}

\newcommand{\ZZ}{{\mathbb Z}}
\newcommand{\NN}{{\mathbb N}}

\newcommand{\PP}{{\mathbb P}}
\newcommand{\AAA}{{\mathbb A}}

\newcommand{\mt}[1]{{\boldsymbol{#1}}}

\def\epsilon{\varepsilon}

\def\ii{\mathrm{i}}
\def\ee{\mathrm{e}}

\def\PV{\mathrm{PV}\int}

\def\dif{ {\mbox{\rm d}} }
\def\pd{ \partial }

\makeatletter
\def\l@section{\@tocline{1}{0pt}{1pc}{}{}}
\def\l@subsection{\@tocline{2}{0pt}{1pc}{4.6em}{}}
\def\l@subsubsection{\@tocline{3}{0pt}{1pc}{7.6em}{}}
\renewcommand{\tocsection}[3]{%
  \indentlabel{\@ifnotempty{#2}{\makebox[2.3em][l]{%
    \ignorespaces#1 #2.\hfill}}}#3}
\renewcommand{\tocsubsection}[3]{%
  \indentlabel{\@ifnotempty{#2}{\hspace*{1.3em}\makebox[2.3em][l]{%
    \ignorespaces#1 #2.\hfill}}}#3}
\renewcommand{\tocsubsubsection}[3]{%
  \indentlabel{\@ifnotempty{#2}{\hspace*{4.6em}\makebox[3em][l]{%
    \ignorespaces#1 #2.\hfill}}}#3}

\makeatother

\title[Renormalization and  Leray solution for generalized  Navier-Stokes]{Renormalization and {\it a-priori} bounds for Leray self-similar solutions to the generalized mild Navier-Stokes equations
}
\author{Denis Gaidashev}
\address{Uppsala University, Uppsala, Sweden}
\email{gaidash@math.uu.se}

\subjclass[2010]{}
\date{\today}

\begin{document}

\begin{abstract}
  We demonstrate that the problem of existence of Leray self-similar blow up solutions in a generalized  mild  Navier-Stokes system with the fractional Laplacian $(-\Delta)^{\gamma/2}$ can be stated as a fixed point problem for a ``renormalization'' operator. We proceed to construct {\it a-priori} bounds, that is a renormalization invariant precompact set in an appropriate weighted $L^p$-space.

As a consequence of {\it a-priori} bounds, we prove existence of  renormalization fixed points for $d \ge 2$ and $d<\gamma <2 d+2$, and existence of non-trivial Leray self-similar mild solutions in $C^\infty([0,T),(H^k)^d \cap (L^p)^d)$, $k>0, p \ge 2$, whose $(L^p)^d$-norm becomes unbounded in finite time $T$.     
\end{abstract}
\maketitle

\setcounter{tocdepth}{2}
\tableofcontents

\section{Introduction}

Bakhtin, Dinaburg and Sinai \cite{BDS} proposed a renormalization approach to the question of existence of the blow-up solutions in the generalized Navier-Stokes equations,
\begin{equation}\label{eq:NS}
\partial_t u + (u \cdot \nabla) u + \nabla p = - \Lambda^\gamma u\,, 
\qquad
\nabla \cdot u =0\,, 
\qquad
u(\cdot,0)=u_0,
\end{equation}
where $\Lambda^\gamma=\left(-\Delta\right)^{ \gamma \over 2}$ is the fractional Laplacian.  In their approach, a solution to an initial value problem with a carefully chosen, self-similar, initial conditions is ``shadowed'' by a solution of certain fixed point problem for an integral nonlinear ``renormalization'' operator in an appropriate functional space.

Later the renormalization approach was applied to several hydrodynamics models by Li and Sinai \cite{LS0,LS1,LS2,LS3}. The method of these publications is a development of that of  \cite{BDS}, unlike the authors of \cite{BDS}, Li and Sinai derive an exact, and not an approximate, renormalization fixed point equation: existence of a fixed point for such equation implies existence of finite time blow up solutions for the hydrodynamic models considered. The solutions obtain by Li and Sinai are, however, complex-valued.

In all these works the form of the solution sought for is exactly the one proposed by J. Leray in his classical work \cite{Leray}:
\begin{equation}
\label{Leraysol}u(x,t)=  (T-t)^{1 -\gamma \over \gamma} U \left( x (T-t)^{-{1 \over \gamma}} \right).
\end{equation}
The Bakhtin-Dinaburg-Sinai-Li renormalization operator can be viewed as a Picard-Lindel\"off integral operator for Leray self-similar solutions in the Fourier space.

In this paper we further study this renormalization for the $d$-dimensional generalized mild Navier-Stokes system with the goal of understanding its applicability to the problem of  existence of Leray solutions for some $d$ and $\gamma$.

Some strong regularity results are known for the system $(\ref{eq:NS})$. The case of $\gamma>2$ is commonly called the {\it hyperdissipative} Navier-Stokes system, that  of $\gamma<2$ - {\it hypodissipative}.  The strong solutions of $(\ref{eq:NS})$  are globally regular in the case $\gamma={(d+2)/2}$ (the critical exponent) and $\gamma>{(d+2)/2}$ (the subcritical hyperdissipative case): see \cite{KaPa}. 

Global regularity in ``barely'' supercritical regimes have been obtained in \cite{Roy}, \cite{Tao2007} and \cite{Tao2009}. 

Estimate on the dimensions of the set of singular points, i.e. points where the solution is not locally essentially bounded, are known for the Navier-Stokes itself, see \cite{CKN} for $\gamma=2$ and \cite{KaPa} for $\gamma>2$.

In this paper we investigate possible mild Leray blow ups in the hyperdissipative regime.

We derive an equation for the Fourier transform of certain self-similar solutions of the mild generalized Navier-Stokes system, and demonstrate that existence of  blow-up solutions within that self-similar class reduces to a fixed point problem for a renormalization operator on an appropriate functional space for Fourier images of solutions.

We would like to emphasize that our approach to  the renormalization problem is quite different from that of Li and Sinai: rather than looking for solutions in the form of a power series and deriving a fixed point equation for the coefficients of this series as in \cite{LS0}, we use  general considerations of compactness and the notion of {\it a-priori} bounds borrowed from the renormalization theory in complex dynamics within a class of {\it real} solutions different from \cite{LS0}.

The main result of the paper is  existence of {\it a-priori} bounds for this renormalization operator, that is, existence of a renormalization invariant, precompact set. The following is a rather non-rigorous statement of this result, the detailed version can be found in Section $\ref{results}$. We emphasize again that the renormalization operator acts in a space of Fourier images of functions from the physical space.

\hspace{2mm}

\begin{adjustwidth}{6mm}{}
  \noindent \textit{\textbf{A-priori bounds.}  For $d \ge 2$ and $d<\gamma<2 d+2$, the renormalization operator  admits an invariant precompact convex set $\cN \not \ni 0$ in an appropriate weighted $\bL^p_u:=(L^p_u(\RR^d,\RR))^d$ space.}

    \textit{As a consequence, the renormalization operator has a non-trivial fixed point in $\bigcap_{1 \le m  \le 2} \bL^m$.}
\end{adjustwidth}

\hspace{2mm}

An immediate observation is that the {\it a-priori} bounds do not cover the case of $d=3$ and $\gamma=2$. This is not surprising as the Leray solutions have been shown to be largely trivial in this case. Specifically, \cite{NRS} shows that  for $d=3$ and $\gamma=2$, any weak Leray self-similarity solution $(\ref{Leraysol})$ with $U \in L^3(\RR^3)$ has to be trivial.  The condition $U \in L^3(\RR^3)$ holds, in particular, if the solution satisfies a global energy estimate of the form
\begin{equation}
\nonumber \| u( \cdot, t ) \|^2_2+2\int_{t_1}^t \| \nabla u ( \cdot ,s) \|_2^2  \ d s \le \| u( \cdot, t_1 )\|^2_2.
\end{equation}

Furthermore, \cite{Tsai} demonstrate that for $d=3$ and $\gamma=2$, any weak Leray self-similarity solution $(\ref{Leraysol})$ with $U \in L^m$, $m \in (3,\infty)$,  has to be trivial, and it is constant if $m=\infty$.  Additionally, \cite{Tsai} shows that any weak Leray self-similar solution  satisfying a local energy estimate
\begin{equation}
  \nonumber {\rm ess sup}_{t \in (\tau,T)}  \int_{B_r(0)} |u(x, t |^2 \ d x+2 \int_{\tau}^T  \int_{B_r(0)} | \nabla u ( x,t)|^2 \ d x \ d t < \infty
\end{equation}
for some ball $B_r(0)$ and some $\tau < T$,  is trivial.

In contrast to these results, non-trivial Leray self-similar solutions in the hyperdisipative subcritical case, $d \ge 2$ and $\gamma>(d+2)/2$, have no chance of satisfying even the local energy estimate since the function
$$\int_{B_r(0)} |u(x, t) |^2 \ d x=(T-t)^{{2+d \over \gamma}-2} \int_{B_{\rho(t)}(0)} |U(y)|^2 d y,$$
where $\rho(t)=r (T-t)^{-{1 \over \gamma}}$, is essentially unbounded on $(0,T)$.

We, therefore, ask a similar question about yet a weaker class of solutions: Are there any non-trivial Leray self-similar {\it mild} solutions of the generalized Navier-Stokes system in the hyperdissipative subcritial case? This is a class of solutions which has nice local itegrability properties with repsect to time, but not necessarily global (that is, over the maximal time-domain of existence of such solutions), see \cite{FuKa}, \cite{Kato}, \cite{Giga}.   Existence of non-trivial {\it a-priori} bounds allows us to deduce existence of such similarity solutions for the mild generalized  Navier-Stokes system
\begin{equation}
  \label{mildNS}  u(x,t)=e^{-t \PP \Lambda^\gamma} u(x,0)-\int_{0}^t e^{-(t-\tau) \PP \Lambda^\gamma} \ \PP(u(x,\tau) \cdot \nabla u(x,\tau)  ) d \tau,
\end{equation}
(here $\PP: \bL^2 \mapsto  \bPL^2$, is the Leray projector, $\bPL^p$ being the subspace of divergence-free $\bL^p$-functions) in the hyperdissipative case  $d \ge 2$ and $d<\gamma <2 d+2$. 

\hspace{2mm}
  
\begin{adjustwidth}{6mm}{}
  \noindent \textit{\textbf{Existence of finite-time hyperdissipative subcritical mild blow ups.}}

  \textit{For every $d \ge 2$ and $d<\gamma <2 d+2$ there exists a $T>0$ and a function
    $$\Psi \in \bigcap_{l \ge 2} \bPL^l$$
    with the following property. Set $$u(x,t)=  (T-t)^{1 -\gamma \over \gamma} \Psi \left( x (T-t)^{-{1 \over \gamma}} \right).$$  Then
    $$u \in C^\infty([0,T),  X), \quad X=\left(  \hspace{1.0mm} \bigcap_{k >0} \bH^k \right) \bigcap \left( \hspace{1.0mm} \bigcap_{m \ge 2}  \bPL^m   \right)\bigcap C^\infty(\RR^d),$$
      with derivatives vanishing at infinity, and  satisfies $(\ref{mildNS})$ for all  $0 <t <T$.}

   \textit{ Additionally, for any $m \ge 2$, the norm $\| u(\cdot, t) \|_m$ is finite for  $t \in [0,T)$ and becomes unbounded as $t \rightarrow T$.}
\end{adjustwidth}

\hspace{2mm}

We remark that the last result does not contradict strong regularity for the subcritical Navier-Stokes system: this solution clearly fails to be strong (see Remark $(\ref{notstrong})$.


Our existence proof uses the results of  Y. Giga \cite{Giga}  and T. Kato \cite{Kato}. In \cite{Giga} the author has considered an initial value problem with the initial data in  class $\bPL^m$ for the general parabolic equation satisfying a certain set of bounds (in particular, the generalized Navier-Stokes system with the fractional Laplacian satisfies these bounds). \cite{Giga} proves local existence and uniqueness of solutions in  the class $BC([0,T), \bPL^m)$ - a result which is essential to our proof.

  The paper is organized as follows.  In Section $\ref{renormNS}$ we derive a renormalization operator acting in a space of Fourier images of the solutions. The domain of this operator is the weighted space $\bL^p_u$: in Section $\ref{invset}$ we demonstrate that this space is renormalization invariant. Section $\ref{rensym}$ describes some straightforward symmetries of renormalization, while Section $\ref{trivapb}$ demonstrates that the trivial point $0$ is the unique renormalization fixed point in a certain neighborhood of itself.   The key proofs of the paper: those of existence of a non-trivial renormalization fixed point, and its integrability properties are contained in Sections $\ref{nontrivapb}$, $\ref{properties}$ and $\ref{blowups}$.

  The statement of results, more precise than those announced in this Introduction, can be found in Section $\ref{results}$. 

\section{Renormalization problem for $d$-dimensional Navier-Stokes equations} \label{renormNS}

\subsection{The generalized Navier-Stokes equations}

Recall the initial value problem for a system of the generalized Navier-Stokes equations
\begin{equation}\label{eq:NS}
\partial_t u + (u \cdot \nabla) u + (-\Delta)^{\gamma \over 2}  u + \nabla p = g\,, 
\qquad
\nabla \cdot u =0\,, 
\qquad
u(\cdot,0)=u_0
\end{equation}
where 
$g : \RR^d \times \RR \rightarrow \RR^d$ 
is an external force, and $u_0: \RR^d \rightarrow \RR^d$ is the initial datum. The unknowns are
$u : \RR^d \times \RR \rightarrow \RR^d$ and $p : \RR^d \times \RR \rightarrow \RR$. The  fractional Laplacian of order $s$ in $\RR^d$, $d \ge 1$ is defined as
\begin{equation}\label{eq:FracDelta}
(-\Delta)^\alpha u(x)=c_{d,\alpha} \PV_{\RR^d}
\frac{u(x)-u(y)}{|x-y|^{d+2\alpha}} \dif y\,,\qquad
c_{d,\alpha} = \frac{4^\alpha \Gamma(d/2+\alpha)}{\pi^{d/2} |\Gamma(-\alpha)|}\,,
\end{equation}
where $\Gamma$ is the Gamma function.

We will be interested in $d \geq 2$ and $g \equiv 0$. We formally apply the  Fourier transform  $v(y)=\cF[u](y)$, given by
\[
v(y,t)=\int_{\RR^d} u(x,t) \ee^{-\ii y \cdot x}  d x\,,
\qquad
u(x,t)={1 \over (2 \pi)^d }\int_{\RR^d} v(y,t) \ee^{\ii y \cdot x} d y\,,
\]
so that the Navier-Stokes system  becomes the following system of integro-differential equations:
\begin{equation}\label{eq:sist:FT}
\frac{\pd v(y,t)}{\pd t} = -|y|^\gamma v(y,t) - i \int_{\RR^d} (y \cdot v(y-z,t)) P_y v(z,t) d z \,,
\qquad
\end{equation}
where $P_y$ is the {\it Leray} orthogonal projection to the subspace orthogonal to $y$, i.e.
\begin{equation}\label{eq:leray}
P_y v = v - \frac{v \cdot y}{|y|^2} y\,.
\end{equation}

\begin{remark}
The incompressibility condition $\nabla \cdot u=0$ reads $v(y,t) \cdot y =0$ in
Fourier space. We observe that $P_y v(y,t)=v(y,t)$ for a function that satisfies the incompressibility condition.
\end{remark}

\begin{remark}
If $u : \RR^d \times \RR \rightarrow \RR$ then $v:\RR^d\times \RR \rightarrow \CC$
satisfies the symmetry property
\[
v(y,t)=\overline{v(-y,t)}\,,\qquad \forall t \in \RR\,.
\]
\end{remark}

Integrating \eqref{eq:sist:FT} with respect to $t$, we obtain the integral equations:
\begin{equation}\label{eq:integral:NSgamma}
v(y,t) = \ee^{-|y|^\gamma t} v_0(y) - i \int_0^t \int_{\RR^d} (y \cdot v(y-z,s)) P_y v(z,s)\ee^{-|y|^\gamma (t-s)} d z  \ d s\,,
\end{equation}
where $v$ is transversal, i.e. $v(y,t) \cdot y=0$, and  $v_0=v(y,0)$ stands for the initial datum at $t=0$.

At this point we will recall the definition of a {\it mild solution} (see, e.g. \cite{LiMa})

Let
\begin{equation}
  \nonumber   \bL^p  =  (L^p(\RR^d,\RR))^d, \quad   \bW^{l,p} =  (W^{l,p}(\RR^d,\RR))^d,
\end{equation}
equipped with a product norm.

\begin{definition}
  Let $u \in L^\infty([0,T),\bL^d)$. Then $u$ is a mild solution for the generalized Navier-Stokes system with the initial condition $u_0 \in \bL^d$, if $u$ satisfies $(\ref{mildNS})$, where the domain of the generalized Stokes operator $\AAA=\PP (-\Delta)^{\gamma \over 2}$ is
$$D(\AAA)=\overline{\{ u \in C_0^\infty(\RR^d,\RR^d): {\rm div (u) }=0\}} \cap \bW_0^{1,p} \cap \bW^{\gamma,p},$$ 
and where $p>1$ and the closure is taken in $\bL^p$.
\end{definition}

We will, however, modify this definition somewhat and consider solutions in  $u \in  C^\infty([0,T),\bL^d) \cap \left(\cap_{\epsilon>0} L^\infty([0,T-\epsilon],\bL^d) \right)$. 

P.-L. Lions and N. Masmoudi have proved in \cite{LiMa} uniqueness of mild solutions in $C([0,T),\bL^d)$ for $d \ge 3$ (in $L^\infty([0,T),\bL^d)$ for $d \ge 4$).

Equations $(\ref{eq:integral:NSgamma})$ are a formal Fourier transform of $(\ref{mildNS})$, however, we will construct solutions $v$ of the system $(\ref{eq:integral:NSgamma})$ in  $\cap_{1 \le l \le 2} \bL^l$: an inverse Fourier transform on this intersection is a well-defined operator into  $\cap_{ \ge 2} \bL^l$, additionally, we will show that such inverse Fourier transforms are  in $C^\infty([0,T),\bL^d) \cap  \left(\cap_{\epsilon>0} L^\infty([0,T-\epsilon],\bL^d) \right)$, and hence are mild solutions.

\subsection{Derivation of the renormalization operator.} We will make an assumption that
\begin{equation}\label{vform}
  v(y,t)= {\theta(y \tau(t))  \over \tau(t)^c} + i {\psi(y \tau(t))  \over \tau(t)^k} ,
\end{equation}
where $\theta, \psi: \RR^d \mapsto \CC^d$, and
$$\tau(t)=(T-t)^{1 \over \gamma}.$$
We further assume that the $\theta$ is even, and $\psi$ is odd: this ensures that the inverse Fourier transform of function $v$ is real.

We can now write the integral equation   $(\ref{eq:integral:NSgamma})$ for the initial value problem with the initial condition
$$v_0(y)= {\theta(y \tau_0)  \over \tau_0^c} +  i {\psi(y \tau_0)  \over \tau_0^k}, \quad \tau_0 \equiv \tau(0)=T^{1 \over \gamma}$$
{\it under the assumption that  $v(y,t)$ remains of the form $(\ref{vform})$ for $t \in (0,T)$}. Setting
\begin{equation} \label{etanotation}
\eta=y \tau_0, \quad \xi= y \tau(t), \quad \zeta=y \tau(s), \quad x=z \tau(s), \quad \beta={ \tau(t)  \over \tau_0 },
\end{equation}
we get the following two equations, after multiplication by $\tau(t)$; one for the real part of $v(y,t)$:
\begin{eqnarray}
  \nonumber    {\theta(\xi) \over \tau(t)^{c-1}} \hspace{-1.5mm}& \hspace{-0.5mm}= \hspace{-0.5mm}& \hspace{-1.5mm}  e^{|\xi|^\gamma-|\eta|^\gamma} {\theta(\eta) \tau(t) \over \tau_0^c}- \gamma   \hspace{-1.0mm}\int_{|\xi|}^{|\eta|}   \hspace{-1.5mm}  {|\zeta|^{\gamma-1}   e^{|\xi|^\gamma-|\zeta|^\gamma}     \over \tau(s)^{d+c+k} |y|^\gamma}   \int_{\RR^d}  \hspace{-1.0mm}{\xi \cdot \theta(\zeta-x) P_\zeta \psi(x)}   \ d x   \ d |\zeta| \\
   \nonumber    & \hspace{-0.5mm} \phantom{=} \hspace{-0.5mm}& \hspace{-1.5mm}  \phantom{e^{|\xi|^\gamma-|\eta|^\gamma} {\theta(\eta) \tau(t) \over \tau_0^c}}- \gamma   \hspace{-1.0mm}\int_{|\xi|}^{|\eta|}   \hspace{-1.5mm}  {|\zeta|^{\gamma-1}   e^{|\xi|^\gamma-|\zeta|^\gamma}     \over \tau(s)^{d+c+k} |y|^\gamma}   \int_{\RR^d}  \hspace{-1.0mm}{\xi \cdot \psi(\zeta-x) P_\zeta \theta(x)}   \ d x   \ d |\zeta|,
\end{eqnarray}
where $\cdot$ denotes the scalar product, and an equation for the imaginary part of $v$:
\begin{eqnarray}
  \nonumber    {\psi(\xi) \over \tau(t)^{k-1}} \hspace{-1.5mm}& \hspace{-0.5mm}= \hspace{-0.5mm}& \hspace{-1.5mm}  e^{|\xi|^\gamma-|\eta|^\gamma} {\psi(\eta) \tau(t) \over \tau_0^k}+ \gamma   \hspace{-1.0mm}\int_{|\xi|}^{|\eta|}   \hspace{-1.5mm}  {|\zeta|^{\gamma-1}   e^{|\xi|^\gamma-|\zeta|^\gamma}     \over \tau(s)^{d+2c} |y|^\gamma}   \int_{\RR^d}  \hspace{-1.0mm}{\xi \cdot \theta(\zeta-x) P_\zeta \theta(x)}   \ d x   \ d |\zeta| \\
   \label{impart1}    & \hspace{-0.5mm} \phantom{=} \hspace{-0.5mm}& \hspace{-1.5mm}  \phantom{e^{|\xi|^\gamma-|\eta|^\gamma} {\theta(\eta) \tau(t) \over \tau_0^k}}- \gamma   \hspace{-1.0mm}\int_{|\xi|}^{|\eta|}   \hspace{-1.5mm}  {|\zeta|^{\gamma-1}   e^{|\xi|^\gamma-|\zeta|^\gamma}     \over \tau(s)^{d+2k} |y|^\gamma}   \int_{\RR^d}  \hspace{-1.0mm}{\xi \cdot \psi(\zeta-x) P_\zeta \psi(x)}   \ d x   \ d |\zeta|.
\end{eqnarray}
We emphasize that $\xi$, $\eta$ and $\zeta$ here are colinear, and that the equations hold for $0 \le t<T$.

If we now set
$$k=\gamma-1-d$$
in the first equation, it becomes after multiplication by $|y|^{1-c}$:
\begin{eqnarray}
  \nonumber    |\xi|^{-c+1} {\theta(\xi)} \hspace{-1.5mm}& \hspace{-0.5mm}= \hspace{-0.5mm}& \hspace{-1.5mm}  e^{|\xi|^\gamma-|\eta|^\gamma}  |\eta|^{-c}|\xi| {\theta(\eta)}- \gamma   \hspace{-1.0mm}\int_{|\xi|}^{|\eta|}   \hspace{-1.5mm}  {|\zeta|^{-c}   e^{|\xi|^\gamma-|\zeta|^\gamma}}   \int_{\RR^d}  \hspace{-1.0mm}{\xi \cdot \theta(\zeta-x) P_\zeta \psi(x)}   \ d x   \ d |\zeta| \\
   \label{repart2}    & \hspace{-0.5mm} \phantom{=} \hspace{-0.5mm}& \hspace{-1.5mm}  \phantom{  e^{|\xi|^\gamma-|\eta|^\gamma}  |\eta|^{-c}|\xi| {\theta(\eta)}}- \gamma   \hspace{-1.0mm}\int_{|\xi|}^{|\eta|}   \hspace{-1.5mm}  {|\zeta|^{-c}   e^{|\xi|^\gamma-|\zeta|^\gamma}}   \int_{\RR^d}  \hspace{-1.0mm}{\xi \cdot \psi(\zeta-x) P_\zeta \theta(x)}   \ d x   \ d |\zeta|.
\end{eqnarray}
For $(\ref{impart1})$ to be an equation for the functions that depend only on $\xi$ and $\eta$, as $(\ref{repart2})$, we must necessarily have
$$c=k=\gamma-1-d,$$
and then $(\ref{impart1})$ becomes, after multiplication by $|y|^{1-c}$:
\begin{eqnarray}
 \nonumber    |\xi|^{-c+1} {\psi(\xi)} \hspace{-1.5mm}& \hspace{-0.5mm}= \hspace{-0.5mm}& \hspace{-1.5mm}  e^{|\xi|^\gamma-|\eta|^\gamma}  |\eta|^{-c}|\xi| {\psi(\eta)}+ \gamma   \hspace{-1.0mm}\int_{|\xi|}^{|\eta|}   \hspace{-1.5mm}  {|\zeta|^{-c}   e^{|\xi|^\gamma-|\zeta|^\gamma}}   \int_{\RR^d}  \hspace{-1.0mm}{\xi \cdot \theta(\zeta-x) P_\zeta \theta(x)}   \ d x   \ d |\zeta| \\
   \label{impart3}    & \hspace{-0.5mm} \phantom{=} \hspace{-0.5mm}& \hspace{-1.5mm}  \phantom{ |\eta|^{-c}|\xi| e^{|\xi|^\gamma-|\eta|^\gamma} \psi(\eta)}- \gamma   \hspace{-1.0mm}\int_{|\xi|}^{|\eta|}   \hspace{-1.5mm}  {|\zeta|^{-c}   e^{|\xi|^\gamma-|\zeta|^\gamma}}   \int_{\RR^d}  \hspace{-1.0mm}{\xi \cdot \psi(\zeta-x) P_\zeta \psi(x)}   \ d x   \ d |\zeta|.
\end{eqnarray}
We eventually obtain a pair of coupled equations
\begin{eqnarray}
  \nonumber    {\theta(\xi)} \hspace{-1.5mm}& \hspace{-0.5mm}= \hspace{-0.5mm}& \hspace{-1.5mm}  e^{|\xi|^\gamma-|\eta|^\gamma}  {|\xi|^c \over |\eta|^{c}} {\theta(\eta)}- \gamma |\xi|^{c-1}   \hspace{-1.0mm}\int_{|\xi|}^{|\eta|}   \hspace{-1.5mm}  {r^{-c}   e^{|\xi|^\gamma-r^\gamma}}   \int_{\RR^d}  \hspace{-1.0mm}{\xi \cdot \theta\left({\eta \over |\eta|} r-x \right) P_\zeta \psi(x)}   \ d x   \ d r \\
   \label{repart4}    & \hspace{-0.5mm} \phantom{=} \hspace{-0.5mm}& \hspace{-1.5mm}  \phantom{  e^{|\xi|^\gamma-|\eta|^\gamma}  {|\xi|^c \over |\eta|^{c}} {\theta(\eta)}}- \gamma |\xi|^{c-1}  \hspace{-1.0mm}\int_{|\xi|}^{|\eta|}   \hspace{-1.5mm}  {r^{-c}   e^{|\xi|^\gamma-r^\gamma}}   \int_{\RR^d}  \hspace{-1.0mm}{\xi \cdot \psi\left({\eta \over |\eta|} r-x \right) P_\zeta \theta(x)}   \ d x   \ d r.
\end{eqnarray}
\begin{eqnarray}
 \nonumber    {\psi(\xi)} \hspace{-1.5mm}& \hspace{-0.5mm}= \hspace{-0.5mm}& \hspace{-1.5mm}  e^{|\xi|^\gamma-|\eta|^\gamma}   {|\xi|^c \over |\eta|^{c}} {\psi(\eta)}+ \gamma   |\xi|^{c-1}   \hspace{-1.0mm}\int_{|\xi|}^{|\eta|}   \hspace{-1.5mm}  {r^{-c}   e^{|\xi|^\gamma-r^\gamma}}   \int_{\RR^d}  \hspace{-1.0mm}{\xi \cdot \theta\left({\eta \over |\eta|} r-x \right) P_\zeta \theta(x)}   \ d x   \ d r \\
   \label{impart4}    & \hspace{-0.5mm} \phantom{=} \hspace{-0.5mm}& \hspace{-1.5mm}  \phantom{  {|\xi|^c \over |\eta|^{c}} e^{|\xi|^\gamma-|\eta|^\gamma} \psi(\eta)}- \gamma   |\xi|^{c-1}   \hspace{-1.0mm}\int_{|\xi|}^{|\eta|}   \hspace{-1.5mm}  {r^{-c}   e^{|\xi|^\gamma-r^\gamma}}   \int_{\RR^d}  \hspace{-1.0mm}{\xi \cdot \psi\left({\eta \over |\eta|} r-x \right) P_\zeta \psi(x)}   \ d x   \ d r.
\end{eqnarray}
These equations are equivalent to equations $(\ref{eq:integral:NSgamma})$ for the special form $(\ref{vform})$ for times $0 \le t<T$. A function $\vartheta=\theta+i \psi$ solving this equation for all $\eta \in \RR^d$ and all $\xi \in \RR^d$ such that $|\xi| \le |\eta|$, provides a  solution
$$v(y,t)=  (T-t)^{{d+1-\gamma \over \gamma}} \left( \theta\left(y (T-t)^{1 \over \gamma} \right) + i  \psi\left(y (T-t)^{1 \over \gamma} \right)  \right),$$
to the initial value problem $(\ref{eq:integral:NSgamma})$  with the initial data
\begin{equation}
\label{IC} v_0(y)=  T^{{d+1-\gamma \over \gamma}} \left( \theta\left(y T^{1 \over \gamma} \right) + i  \psi\left(y T^{1 \over \gamma} \right)  \right).
\end{equation}

The problem of existence of solution of $(\ref{repart4})-(\ref{impart4})$ can be stated as a fixed point problem for a family of operators

\begin{eqnarray}
 \label{operatorR1} \hspace{10mm} \cR_\beta[\vartheta](\eta) \hspace{-1.5mm} &\hspace{-1.5mm}=\hspace{-1.5mm}& \hspace{-1.5mm} \beta^c \cP_\beta[\vartheta]\left({\eta \over \beta} \right),   \\
 \label{operatorP1}  \hspace{10mm}  \cP_\beta[\vartheta](\eta)\hspace{-1.5mm}  & \hspace{-1.5mm} = \hspace{-1.5mm} & \hspace{-1.5mm} e^{|\beta \eta|^\gamma-|\eta|^\gamma} \vartheta(\eta)+ i |\eta|^{c-1}\gamma  \hspace{-1.0mm} \int \displaylimits_{\beta |\eta|}^{|\eta|}  \hspace{-1.0mm}{ e^{|\beta \eta|^\gamma-r^\gamma} \over  r^{c} } \hspace{-1.5mm}  \int \displaylimits_{\RR^d} \hspace{-1.0mm} \eta \cdot \vartheta  \hspace{-0.5mm}  \left({\eta \over |\eta|} r-x \right) P_\eta \vartheta(x)   \ d x   \ d r,
\end{eqnarray}
or
\begin{equation}
 \label{operatorR2}  \cR_\beta[\vartheta](\eta)   =  \beta^c e^{|\eta|^\gamma \left(1- {1 \over \beta^\gamma}  \right)} \vartheta \left(\eta \over \beta \right)+ i \gamma  \int \displaylimits_{1}^{1 \over \beta}  { e^{|\eta|^\gamma \left(1-t^\gamma\right)} \over  t^{c} }   \int \displaylimits_{\RR^d}  \eta \cdot \vartheta  \hspace{-0.5mm}  \left(\eta t - x \right) P_\eta \vartheta(x)   \ d x   \ d t,
\end{equation}
for a vector field $\vartheta: \RR^d  \mapsto \CC^d$.

\vspace{2mm}

\begin{remark}
  The fact that the exponent $c=\gamma-d-1$ is nothing but the consequence of the well-known scaling property of the solutions of the Navier-Stokes system: a function $u_{\beta}(x,t)=\beta^{\gamma-1} u(\beta x , \beta^{\gamma} t)$ is a solution for the initial data $u_{0,\beta}(x)= \beta^{\gamma-1} u_0(\beta x)$ whenever $u$ solves the initial value problem for $u_0$.  
\end{remark}

\begin{remark}
  For $d=3$ and $\gamma=2$, the scaling exponent
  $$c=-2.$$
\end{remark}

\vspace{2mm}

We will refer to the operator $\cR_\beta$ as renormalization, or the renormalization operator, since it exhibits all the requisite properties found in renormalization in dynamics. The fixed point equation $\psi_*=\cR_\beta[\psi_*]$ says informally that the time evolution of the initial data in the Fourier space at a {\it later time} looked at {\it at a larger scale} is equivalent to the initial data itself.  In the physical space the fixed point equation exactly captures the key idea behind renormalization in dynamics: equivalence of ``later'' dynamics at smaller scale with the original dynamics. This comes as no surprise as this has been already written in Leray self-similar solution $(\ref{Leraysol})$.

\section{Statement of results} \label{results}
We will now state the main conclusions of this paper. Theorem A will be stated in a rather informal way. For the exact statement see subsection $\ref{apriori_subsection}$. 
\begin{thmA} \label{theoremA}(\underline{A-priori bounds.})
  For $\beta_0 \in (0,1)$, $d \ge 2$ and $d<\gamma<2 d+2$,  there exist $p>2$, $b>0$,  $\nu \in \RR$, such that  for any $\beta \in (\beta_0,1)$, the operator $\cR_\beta$ admits an invariant bounded convex precompact subset $\cN \not\ni 0$ in  $L^p_{v_\nu}(\RR^d,\RR)$  of {\it azimuthal} vector fields
  $$\vartheta=i \psi e_{d-1},$$
where $\psi: \RR^d \mapsto \RR$ and  $e_{d-1}$ is the azimthal coordinate unit vector,   equipped with the norm
  $$ \|f \|_{v_\nu}^p= \int \displaylimits_{\RR_d} |f(x)|^p v_\nu^p(x) d x, \quad v_\nu(x)=e^{{b \over 2} |\eta|} |\eta|^\nu.$$
\end{thmA}

Our second result deals with the existence of Leray blow up solutions for an $L^m$, $m \ge 2$, initial data and describes their  regularity for $d \ge 2$ and certain $\gamma$.

\begin{thmB} \label{theoremB}
  For any  $d \ge 2$ and $d<\gamma<2 d+2$ there exists $0<\beta_0<1$, and  a function
  $$\psi_* \in \bigcap_{1 \le l \le2} L^l(\RR^d,\RR),$$
  for which the following holds.

  \vspace{2mm}
  
  \noindent $\mathbf{(1)}$ Let $\tau(t)=(T-t)^{1 \over \gamma}$. For any $T>0$  the function
  $$v_*(y,t)=i \tau(t)^{d+1-\gamma} \psi_* (y \tau(t)) \ e_{d-1}(y \tau(t))$$
  satisfies equation $(\ref{eq:integral:NSgamma})$ for $t \in (0,(1-\beta_0^\gamma) T)$; 

  \vspace{2mm}
  
  \noindent  $\mathbf{(2)}$ the inverse Fourier transform
  \begin{equation}
\label{ubeta}   u_*(x,t)=\cF^{-1}[v_*](x,t)
   \end{equation}
  is the unique solution in $BC\left([0,T),\bPL^m \right)$, $m>2$,  of the mild Navier-Stokes system $(\ref{mildNS})$ with the initial data $u_*(x,0)$.

    Additionally, this solution belongs to  $C^\infty\left([0,T),X \right)$ where
      $$X \in \left( \hspace{1mm} \bigcap_{k>0} \bH^k \right) \bigcap \left( \hspace{1mm} \bigcap_{m \ge 2} \bPL^m\right) \bigcap C^\infty(\RR^d),$$
      specifically:
      
  \vspace{2mm}
    
\noindent  $\mathbf{(3)}$  the functions  $u_*(x,t)$ is a $C^\infty(\RR^d \times[0,T),\RR^d)$-function whose derivatives vanish at infinity, in particular
  \begin{equation}
    \label{decay} |u_*(x,t)| < C \ \tau(t)^{1+\theta -\gamma}  |x|^{-\theta}
  \end{equation}
  for some $0<\theta \le 1$;

  \vspace{2mm}
    
    \noindent  $\mathbf{(4)}$   $u_*(x,t)$ is a holomorphic function in a tube $\{z:=x \tau(t)^{-1} \in \CC^d: |\Im(z)|<b/2\}$;

  \vspace{2mm}
    

\vspace{2mm}



\noindent  $\mathbf{(5)}$ for any $m \ge 2$, the $\bL^m$-norm of $u_*( \cdot,t)$ is bounded for any fixed  $t \in [0,T)$, and diverges as $t \rightarrow T$:
\begin{equation}
  \label{energy2}  \|u_*\|_m  \ge {C  \over \tau(t)^{{\gamma-1 }-{d \over m}}}.
\end{equation}
\end{thmB}

We will say that a solution of $(\ref{eq:NS})$ is {\it strong}, if  $u \in C^1\left([0,T),\bH^\gamma \right)$, and, additionally,
  $$\partial_t u, \ (u \cdot \nabla) u, \  (-\Delta)^{\gamma \over 2} u \in L^2((0,T); \bL^2),$$
  (in particular, these conditions are sufficient for an energy iequality to hold). As Remark $\ref{notstrong}$ demonstrates, the Leray self-similar solution from Theorem B {\it is not strong}.

\section{Outline of the proof}\label{outline}

Our proof consists of constructing a precompact renormalization invariant set in an appropriate functional space, that does not contain a trivial function, and demonstrating that the renormalization operator has one and the same fixed point in that set for all $\beta \in (0,1)$. From that we deduce existence of Leray self-similar solutions for all $d \ge 2$ and $d<\gamma <2 d+2$.

We start by considering the action of the renormalization operator on a weighted space $\fRL^p_{u_\nu}$, $u_\nu(\eta)=e^{|\eta|} |\eta|^\nu$, of complex valued vector fields from $\RR^2$ to $\CC^d$ which can be bounded by a radial function, and use a version of Young's convolution inequality to demonstrate that there exists a choice of the weight and a choice of $p>1$ such that this space is renormalization invariant. This is done in Section $\ref{invset}$. Incidentally, in Section  $\ref{trivapb}$, we show that the family of renormalization operators $\cR_\beta$  does not contain any common fixed points in a ball around zero of certain size this space. 

Next, we consider the azimuthal vector fields, that is the vector fields whose only non-zero component is the azimuthal one, in Section $\ref{nontrivapb}$. We construct a renormalization invariant convex set $\cM^\alpha_\mu$ in $L^p_{v_\nu}$, $v_\nu(\eta)=|\eta|^\nu e^{b|\eta|/2}$, as the intersection of two convex sets: the first one being a set of exponentially decaying function bounded in the absolute value by $k |\eta|^\sigma e^{-b |\eta|}$, with a mild singularity $|\eta|^\sigma$ at zero: the second one - a set of functions $\psi \in L^p_{v_\nu}$ for which the integral $\cI[\psi]=\int_{\RR^d} \psi(\eta) e^{-|\eta|^\gamma} |\eta|^\alpha   d \eta$ is larger than some positive constant $\mu$. As part of the construction we show that there exist choices of constants such that both sets are renormalization invariant. We also show that the set $\cM^\alpha_\mu$ is {\it equitight} - that is the $L^p_{v_\nu}$-norms of the functions of sets $(R,\infty)$ vanish as $R \rightarrow \infty$,uniformly in $\cM^\alpha_\mu$.

The set $\cM^\alpha_\mu$ is further intersected with a convex set $\cN$ of functions whose translates satisfy a bound of the form
$$|\psi(\eta)-\psi(\eta-\delta)| \le \omega_\psi (\eta) |\delta|^\theta,$$
with a function $\omega_\psi$ having a bounded $L^p_{v_\nu}$-norm, and a fast decay for large $\eta$.  We demonstrate that this set is also renormalization invariant. This set is also clearly {\it equicontinuity}, that is the $L^p_{v_\nu}$-norms of $\psi(\eta)-\psi(\eta-\delta)$ vanish as $\delta \rightarrow 0$, uniformly in $\cN$.

The key ingredient of the proof is the Frechet-Kolmogorov-Riesz-Weil compactness theorem, which says that a bounded, equitight and equicontinuity set in $L^p_{v_\nu}$ is precompact. This is precompactness of the set $\cN^\alpha_\mu=\cN \cap \cM^\alpha_\mu$ is formulated at the end of Section $\ref{nontrivapb}$ in the form of {\it a-priori bounds}.

After demonstrating the existence of a compact renormalization invariant set $\cN^\alpha_\mu$, we proceed to first embed it into $L^p$ with $1 \le  p \le 2$.
Having done this, we pass back to the physical space $L^m$, $m \ge 2$, and invoke an existence and regularity result by Giga \cite{Giga}, to claim that there is a choice of $T>0$ and a $\beta_1$, sufficiently close to $1$, such that the unique mild solution of the generalized mild Navier-Stokes system for times $t \in (0,(1-\beta_1^\gamma)T)$  coincides with a function
$$(T-t)^{1 -\gamma \over \gamma} \Psi_* \left( x (T-t)^{-{1 \over \gamma}} \right),$$
where $\Psi_*$ is the inverse Fourier transform of the fixed point for $\cR_{\beta_1}$ (see Section $\ref{blowups}$).  Eventually, we extend the same solution to all times $t \in (0,(1-\beta_1^{n \gamma})T)$, $n \in \NN$, using the following key renormalization property: a fixed point for $\cR_\beta$ is also a fixed point for $\cR_{\beta^n}$ for all $n \in \NN$.

Finally, the regularity properties of the solution listed in Theorem B, follow from the exponential boundedness and equicontinuity of the renormalization fixed points.

\section{Renormalization symmetries.}\label{rensym}  In this section we will collect some straightforward properties of $\cR_\beta$, mostly related to renormalization symmetries.

Renormalization clearly preserves parity.

\begin{proposition}
The set of vector fields  with the real and imaginary parts of $\vartheta$  even and odd, respectively, is invariant under $ \cR_\beta$ for all $0<\beta<1$.
\end{proposition}

 We will eventually restrict the search of fixed points of the operator $\cR_\beta$ to the subset of symmetric transversal imaginary fields $\vartheta= i \psi$. To that end we quote some simple properties of the renormalization operator for such fields.

To be specific, a vector field $\psi$ will be called {\it transversal}, if $e_\eta(\eta) \cdot \psi(\eta)=0$ where $e_\eta(\eta)$ is the radial  unit vector field.

Let $\eta=(r,\hat \eta)$ be the representation of $\eta$ in the spherical coordinates, with  $\hat \eta = (\hat \eta_1, \hat \eta_2, \ldots, \hat \eta_{d-1})  \in S_{d-1}$ being the polar  ($\hat \eta_k$, $1\le k \le d-2$)  and  azimuthal ($\hat \eta_{d-1}$)  angles. Clearly any transversal vector field has a unique representation
\begin{equation}
\psi(\eta)= \sum_{i=1}^{d-1} \psi_{i}(r,\hat \eta) e_{_i }(\hat \eta),  
\end{equation}
where $e_i(\hat \eta)$ are the coordinate unit vector fields on the sphere.

We continue with the following collection of some straightforward renormalization properties.

\begin{proposition}\label{symmetries}(\underline{Symmetries of renormalization}) 
\vspace{1mm}
  
  \noindent $i)$ The operator $\cR_\beta$ preserves the set of imaginary vector fields $\vartheta=i \psi$, where $\psi: \RR^d \mapsto \RR^d$:
\begin{equation}
 \label{operatorR3}  \cR_\beta[\psi](\eta)   =  \beta^c e^{|\eta|^\gamma \left(1- {1 \over \beta^\gamma}  \right)} \psi \left(\eta \over \beta \right)- \gamma  \int \displaylimits_{1}^{1 \over \beta}  { e^{|\eta|^\gamma \left(1-t^\gamma\right)} \over  t^{c} }   \int \displaylimits_{\RR^d}  \eta \cdot \psi  \hspace{-0.5mm}  \left(\eta t - x \right) P_\eta \psi(x)   \ d x   \ d t.
\end{equation}

\vspace{2mm}

\noindent $ii)$  The operator $(\ref{operatorR3})$  preserves the set of odd vector fields $\psi(\eta)=-\psi(-\eta)$. 

\vspace{2mm}

\noindent $iii)$  The operator $(\ref{operatorR3})$  preserves the set of  rotationally invariant vector fields $\psi$, i. e. for all $d \ge 2$,  a fixed $n$, $2 \le  n \le d$ and all for all $R \in  SO(n)$,
$$R^{-1} \cR_\beta[\psi] \circ R = \psi$$ 
whenever
$$R^{-1} \psi \circ R = \psi.$$
\vspace{2mm}

\noindent $iv)$  The operator $(\ref{operatorR3})$  preserves the set of transversal vector fields $\psi$, i.e. 
$$e_\eta \cdot \cR_\beta[\psi] =0$$ 
whenever
$$e_\eta \cdot \psi =0$$ 
for all $\eta \in \RR^d$.

\vspace{2mm}

\noindent $v)$  The operator $(\ref{operatorR3})$  preserves the set of transversal $SO(2)$-invariant vector fields which project trivially on the polar directions, i. e. for any rotation $R$ in the subgroup  $SO(2)$ preserving the azimuthal plane,
$$e_{k} \cdot \cR_\beta[\psi] =0,$$ 
$1 \le k \le d-2$, whenever
$$e_{k} \cdot \psi =0$$ 
for all $\eta \in \RR^d$. Additionally,
\begin{equation}
  \label{Ronphi} e_{d-1}(\eta) \cdot \cR_\beta[\psi](\eta)=\beta^c e^{-{ |\eta|^\gamma \over \beta^\gamma}   } e_{d-1}(\eta) \cdot \psi \left({\eta \over \beta}  \right).  
\end{equation}
\end{proposition}
\begin{proof}
  The proof is relegated to Appendix A.
\end{proof}

Transversal vector fields that project trivially on the polar directions will be called {\it azimuthal}. The set of azimuthal $SO(2)$-invariant vector fields will be referred to as $\cT$:
\begin{eqnarray}
  \nonumber \cT:&=&\left\{ \psi: \RR^d \mapsto  \RR^d: e_\eta(\eta) \cdot \psi(\eta)=0, \ e_k(\eta) \cdot \psi(\eta)=0, 1 \le k \le d-2; \right. \\
  \nonumber &&  \left.  \phantom{\psi: \RR^d \mapsto  \RR^d:} \hspace{3mm}   \psi(\eta)=-\psi(-\eta);  \  \psi( R \eta)= R \psi(\eta) \ {\rm for} \ \forall R \in SO(2) \right\},
  \end{eqnarray}
here $SO(2)$ is the rotational group in the azimuthal plane. The following is the direct consequence of $(\ref{Ronphi})$.

\begin{corollary}(\underline{Renormalization fixed points for azimuthal fields}.)
   Any  vector field of the form
  \begin{equation}
    \label{triv_sols} \psi(\eta)=|\eta|^c \ e^{|\eta|^\gamma} f(\hat \eta_p) \  e_{d-1}(\eta),
  \end{equation}
  where $f$ is any function of the polar angles $\hat \eta_p=\hat \eta_1, \ldots, \hat \eta_{d-2}$ are the polar angles,  is a renormalization fixed point in the set $\cT$.
\end{corollary}

\begin{remark}
  Whenever $d$ is odd, the class of fully rotationally invariant transversal vector fields, i.e., invariant under all $R \in SO(d)$, coincides with the zero vector field. Indeed, according to the Hairy Ball Theorem, there is no non-trivial constant tangent vector field on a sphere. However, the class of transversal vector fields invariant under the subgroup $SO(d-1)$ is non-trivial, and contains smooth vector fields.

\end{remark}

\section{Renormalization as an operator on weighted $L^p$-spaces} \label{invset}

\subsection{Notation and preliminaries } Given a non-negative function $w: \RR^d \mapsto \RR_+ \cup \{0\}$,  we will consider the weighted $L^p_{w}(\RR^d, \RR)$ space, equipped with the norm
\begin{equation}
  \label{wnorm} \|f\|_{w,p}= \left( \int_{\RR^d} |f(x)|^p w(x)^p \ d x    \right)^{1 \over p},
\end{equation}
Additionally, the subspace of  functions in $L^p_{w}(\RR^d, \RR)$ which depend only on the radial coordinate in $\RR^d$ will be denoted $R^p_{w}(\RR^d, \RR)$; such functions will be referred to as radial.

We will consider the action of the renormalization operator on complex vector valued fields
$$\vartheta=\theta+i \psi,$$
where $\theta, \ \psi: \RR^d \mapsto \RR^d$.

As before, the product spaces of vector-valued functions will be denoted as follows
$$\bL^p_w=(L^p_{w}(\RR^d, \RR))^d, \quad \bR^p_w=(R^p_{w}(\RR^d, \RR))^d.$$

We will not require the complex structure for $\vartheta$, and will simply consider $\vartheta$ as and element of $\fL^p_w$, where larger product spaces are def
ined as
$$\fL^p_w=(L^p_{w}(\RR^d, \RR))^{2 d}, \quad \fR^p_w=(R^p_{w}(\RR^d, \RR))^{2 d}.$$

The subset of functions $f \in \fL^p_{w}$ such there exists a function $\widehat{f} \in  \fR^p_w$ satisfying $|f| \le |\widehat{f} |$ will be denoted as  $\fRL^p_w$. We will also denote
$$\tilde{f}(\rho):=\widehat{f}(x).$$
In this Section we will demonstrate existence of a renormalization invariant space $\fRL^p_{w}$ for a specific choice of $w$. We will first require two preliminary results.

\begin{lemma}\label{radial} (Lemma 4.1 \cite{NDD}.)
  Let $x \in S^{d-1}=\{x \in \RR^d: |x|=1\}$ and consider an integral of the form
  $$I(x)= \int_{S^{d-1}}   f(x \cdot y) d y$$
  where $f: [-1,1] \mapsto \RR$ is in $L^{1}_{(1-t^2)^{d-3 \over 2 }}([-1,1],\RR)$. Then, $I(x)$ is a constant independent of $x$, and, moreover,
  $$I(x)=\omega_{d-2} \int_{-1}^1 f(t) (1-t^2)^{d -3 \over 2} d t,$$
  where $\omega_d$ is the Lebesgue measure of the $d$-dimensional unit sphere.
\end{lemma}

A consequence of Lemma $\ref{radial}$ is the following
\begin{lemma}\label{radialconv}
  For any two radial functions $\widehat{f}, \widehat{g}:\RR^d  \mapsto \CC^d$,  any unit vector $e \in S^{d-1}$, and $r \in \RR_+$, the following convolution
\begin{equation}
 \nonumber  H(r e):= \int_{\RR^d} \widehat{f}(x) \widehat{g}(r e - x) d x,
\end{equation}
is also a radial funnction, whenever defined, and is independent of $e$:
$$ H(r e)= h(r),$$
where $h: \RR_+ \cup \{0\} \mapsto \RR_+ \cup \{0\}$.
\end{lemma}
\begin{proof}
Let $x=(\rho,\hat x)$, $\hat x \in S^{d-1}$, be the spherical coordinates in $\RR^d$. Then
  \begin{eqnarray}
    \nonumber \int_{\RR^d} \widehat{f}(x) \widehat{g}(r e - x) d x &=& \int_{\RR_+} d \rho \  \rho^{d-1} \tilde{f}(\rho) \ \int_{S^{d-1}} \tilde{g}( \sqrt{\rho^2 +r^2 -2 \rho r e \cdot \hat x    }) d \hat x \\
    \nonumber &=& \omega_{d-2} \int_{\RR_+} \hspace{-2.4mm}  d \rho   \rho^{d-1} \tilde{f}(\rho) \hspace{-1mm} \int_{-1}^1  \hspace{-1.5mm} \tilde{g}( \sqrt{\rho^2 +r^2 -2 \rho r t    }) (1-t^2)^{d -2 \over 2}  d t,
  \end{eqnarray}
  where we have used Lemma $\ref{radial}$. The last expression  clearly becomes a function of $r$ only after the integration.
  \end{proof}

\begin{corollary}\label{nicebound}
Let $h : \RR^d \mapsto \RR_+$, and let  $\widehat{f}, \widehat{g}:\RR^d  \mapsto \RR^d$ be radial. Then
  \begin{equation}
\nonumber    \int_{\RR^d} h(y) \int_{\RR^d} \widehat{f}(x) \widehat{g}(|y| e - x) d x  d y =  \int_{\RR^d} h(y)  \int_{\RR^d} \widehat{f}(x) \widehat{g}(y - x) d x  d y,
  \end{equation}
  whenever both integrals in both sides exist.
\end{corollary}

We will also require a version of Young's inequality for radial functions,  \cite{BiSwa} (see Proposition 2.2 and Example 3.2) and \cite{NaDr} (see Theorem 5):

\begin{thmYR} \label{young}
  Assume that
  \begin{itemize}
  \item[$1)$]  $1 < p,q,r < \infty$, ${1\over r} \le {1 \over p}+{1 \over q}$, ${1 \over r}={1 \over p}+{1 \over q}+{\nu+\mu+\kappa \over d}-1$;
  \item[$2)$] $\nu <{d \over p'}$,  $\mu <{d \over q'}$,  $\kappa <{d \over r}$;
  \item[$3)$] $\nu+\mu\ge (d-1)\left(1-{1 \over p} -{1 \over q}\right)$,  $\nu+\kappa \ge (d-1)\left({1 \over r} -{1 \over p}\right)$, $\mu+\kappa \ge (d-1)\left({1 \over r} -{1 \over q}\right)$;
  \item[$4)$] $\max\{\nu,\mu,\kappa\} >0$ or $\nu=\mu=\kappa=0$,
  \end{itemize}
  where $p'$ and $q'$ denote the H\"older conjugates of $p$ and $q$, respectively.

  Furthermore, suppose that $\phi_i: \RR_+ \mapsto \RR$, $i=1,2,3$ are such that
\begin{equation}
\label{phis}  \phi_3(|x|) \le K + \phi_1(|y|)+\phi_2(|x-y|)
\end{equation}
  for some $K \in \RR$  and any $x,y \in \RR^d$.

  Set
  \begin{equation}
\label{weights}    w_1(x)=e^{\phi_1(|x|)} |x|^\nu, \quad  w_2(x)=e^{\phi_2(|x|)} |x|^\mu, \quad  w_3(x)=e^{\phi_3(|x|)} |x|^{-\kappa}, \quad C=e^K.
  \end{equation}
  
  Then
  $$R^p_{w_1}(\RR^d, \RR) \star R^q_{w_2}(\RR^d, \RR) \subset R^p_{w_3}(\RR^d, \RR),$$
  and there is a constant $C>0$, such that
  \begin{equation}
    \| f \star g \|_{w_3,r} \le C \|f \|_{w_1,p} \| g\|_{w_2,q}.
  \end{equation}  
\end{thmYR}

\vspace{2mm}

In particular, the hypothesis $(\ref{phis})$ of both instances of the inequality is satisfied for $\phi_i=id$ and $K=0$ (triangle inequality). We will denote the weights $(\ref{weights})$  corresponding to this choice of $\phi_i$  as $u_\nu$, $u_\mu$ and $u_{-\kappa}$:
\begin{equation}
  u_\nu(x)=e^{|x|} |x|^\nu, \quad   u_\mu(x)=e^{|x|} |x|^\mu, \quad   u_{-\kappa}(x)=e^{|x|} |x|^{-\kappa}.
\end{equation}

\vspace{2mm}

Additionally, in our computations below, we will set $p=q$ and $\nu=\mu$.

\vspace{2mm}

Throughout the rest of this Section, $C$ will denote positive functions of parameters $d$, $p$, $q$, $r$, $\nu$, etc, which are bounded away from $0$ and from above, and whose values have no bearing on our arguments.   

\vspace{2mm}

\subsection{Existence of a renormalization invariant $\fRL^p_{u_\nu}$-space.} We are now ready to prove renormalization invariance of the space  $\fRL^p_{u_\nu}$.

\begin{proposition} \label{inv1}
  Let $r>1$, $p>1$, $\nu$ and $\kappa$ be such that
  \begin{equation}
    \tag{H1} \label{H1}  \left\{ p>1, \quad    {1 \over r}={2 \over p}+ {2 \nu+\kappa  \over d  }-1, \quad  \kappa < {d \over r}, \quad  { d-1 \over 2} \left(1 - {2 \over p} \right) < \nu  < d \ {p-1 \over p}, \atop
     r>1, \quad  2 r \ge p, \quad  \max\{\nu,\kappa\} >0, \quad \nu+\kappa \ge (d-1) \left({1 \over r} - {1 \over p} \right),
    \right.
  \end{equation}
  and
  \begin{equation}
    \tag{H2} \label{H2} \max\{0,d+p(c+\nu) \}<{d (r-p)\over r} +p(\nu + \kappa +1) <{\gamma p (r-1) \over r}.
  \end{equation}
  Then  $\fRL^p_{u_\nu}$ is invariant under $\cR_\beta$ for all $\beta \in (0,1)$.
\end{proposition}

\begin{proof}
  Consider the renormalization
  \begin{equation}
    \nonumber  \cR_\beta[\vartheta](\eta)=\beta^{c}  e^{-\left({ 1 \over \beta^\gamma}-1 \right) |\eta|^\gamma }  \vartheta\left({\eta \over \beta} \right) - \hspace{0.2mm} \gamma  \hspace{0.2mm}  |\eta|^{c-1}   \hspace{-1.0mm} \int \displaylimits_{ |\eta|}^{|\eta| \over \beta}  \hspace{-1.0mm}{ e^{-r^\gamma} \over  r^{c} } \hspace{-1.5mm}  \int \displaylimits_{\RR^d} \hspace{-1.0mm} \eta \cdot \vartheta  \hspace{-0.5mm}  \left(e_\eta r-x \right) P_\eta \vartheta(x)   \ d x   \ d r.
  \end{equation}

  Let $\vartheta \in  \fRL^p_{u_\nu}$, and let $\widehat{\vartheta} \in  \fR^p_{ u_\nu}$ be a bounding function.
  \begin{align}
    \nonumber \left|  \cR_\beta[\vartheta](\eta) \right| & \le \beta^{c}  e^{-\left({ 1 \over \beta^\gamma}-1 \right) |\eta|^\gamma } \left| \widehat{\vartheta}\left({\eta \over \beta} \right) \right| + \gamma \hspace{0.2mm}  |\eta|^{c}   \hspace{-1.0mm} \int \displaylimits_{ |\eta|}^{|\eta| \over \beta} { e^{-r^\gamma} \over  r^{c} } \hspace{-1.5mm}  \int \displaylimits_{\RR^d} \left|\vartheta    \left(e_\eta r-x \right) \right|  \left| P_\eta \vartheta(x) \right|   \ d x   \ d r \\
    \nonumber  & \le \beta^{c}  e^{-\left({ 1 \over \beta^\gamma}-1 \right) |\eta|^\gamma }   \left| \widehat{\vartheta}\left({\eta \over \beta} \right) \right| + C \hspace{0.2mm}  |\eta|^{c} \hspace{-1.0mm} \int \displaylimits_{ |\eta|}^{|\eta| \over \beta} { e^{-|y|^\gamma} \over  |y|^{c+d-1} } \hspace{-1.5mm}  \int \displaylimits_{\RR^d} \left| \widehat{\vartheta}   \left(y-x \right) \right|  \left| \widehat{\vartheta}(x) \right|   \ d x   \ d y \\
    &=: L(\eta)+ N(\eta),
  \end{align}
  where we have used Corollary $\ref{nicebound}$, and the fact that projection $P_\eta$ does not increase the norm. Notice, that both $L$ and $N$ are radial functions. It remains to show that $L+N \in R^p_{u_\nu}(\RR^d, \RR)$. 
  
  We first consider the linear term $L$:
  \begin{eqnarray}
    \nonumber   \|L\|_{u_\nu,p}&=& \beta^{c} \left( \int_{\RR^d}  e^{-\left({1 \over \beta^\gamma} -1 \right)  p {|\eta|^\gamma } }  e^{p  |\eta|}|\eta|^{p \nu} \left|\widehat{\vartheta}\left({\eta \over \beta} \right) \right|^p    \ d \eta  \right)^{1 \over p} \\
    \nonumber   &\le & \beta^{c+\nu + {d \over p}}  \left( \int_{\RR^d}  e^{-\left({1 \over \beta^\gamma} -1 \right)  p {|\eta|^\gamma } }  e^{p  {|\eta|\over \beta}} \left|{\eta \over \beta} \right|^{p \nu} \left|\widehat{\vartheta}\left({\eta \over \beta} \right) \right|^p    \ d \left( {\eta \over \beta} \right)  \right)^{1 \over p} \\
    \nonumber &\le&  \beta^{c+\nu + {d \over p}}  \ \| \widehat{\vartheta} \|_{u_\nu,p}.
    \end{eqnarray}
 We now turn to the non-linear term. 
\begin{eqnarray}
  \nonumber   \|N\|_{u_\nu,p}^p & \le & C  \int_{\RR^d} |\eta|^{p(c+\nu)}  e^{ p |\eta|^\gamma+ p |\eta|}  \hspace{-1.0mm} \left( \int \displaylimits_{ |\eta|}^{|\eta| \over \beta}  \hspace{-1.0mm}{ e^{-|y|^\gamma} \over  |y|^{c+d-1} } \int_{\RR^d} |\widehat{\vartheta}(y-x)| |\widehat{\vartheta}(x)|  d x   \ d y \right)^p \ d \eta \\
  \nonumber   & \le & C  \int_{\RR^d} |\eta|^{p(c+\nu)}  e^{ p |\eta|^\gamma+ p |\eta|}  I_{c,\widehat{\vartheta},\widehat{\vartheta}}(|\eta|)^p \ d \eta.
\end{eqnarray}
We rearrange the terms in $I_{c,\widehat{\vartheta},\widehat{\vartheta}}(|\eta|)$, and use H\"older inequality with $1/s+1/r=1$, to set up an instance of Young's convolution inequality in weighted $L^p$-spaces. In fact, since this estimate will be used several times in the course of the paper, we will state it as a separate Lemma:
\begin{lemma}\label{Ik}
  Let $\widehat{\omega}, \widehat{\vartheta} \in \fR_{u_\nu}^p$, and $e$ - any unit vector in $\RR^d$, and suppose that parameters $d$, $p$, $\nu$ and $\kappa$ satisfy hypothesis $(\ref{H1})$ for some $r>1$.

  Then the integral
  \begin{eqnarray}
        \nonumber I_{k,\widehat{\omega},\widehat{\vartheta}} (|\eta|) &= &  \int \displaylimits_{ |\eta|}^{|\eta| \over \beta}  \hspace{-1.0mm}{ e^{-|y|^\gamma} \over  |y|^k } \int_{\RR^d} |\widehat{\omega}(e r-x)| |\widehat{\vartheta}(x)|  d x   \ d r \\
    \nonumber  &= &  {1 \over \omega_{d-1} }\int \displaylimits_{ |\eta|}^{|\eta| \over \beta}  \hspace{-1.0mm}{ e^{-|y|^\gamma} \over  |y|^{k+d-1} } \int_{\RR^d} |\widehat{\omega}(y-x)| |\widehat{\vartheta}(x)|  d x   \ d y
  \end{eqnarray}
  satisfies
  \begin{equation}
    \nonumber  I_{k,\widehat{\omega},\widehat{\vartheta}}(|\eta|) \le C  w_k(|\eta|) \| \widehat{\omega} \|_{u_\nu,p}   \| \widehat{\vartheta} \|_{u_\nu,p},
  \end{equation}
  where
  \begin{equation}
    \label{wk} w_k(\eta):= e^{-|\eta|}  \left( \Gamma\left({d \over \gamma}-{s  \over \gamma}  \left(k+d-1-\kappa \right), s t^\gamma \right) \Bigg|_{ |\eta| \over \beta}^{|\eta|} \right)^{1 \over s}.
  \end{equation}
  Additionally, under the hypothesis $(\ref{H2})$,
  $$w_k(\eta) e^{|\eta|^\gamma} |\eta|^{k} \in R_{u_\nu}^p(\RR^d,\RR_+).$$
\end{lemma}
\begin{proof}
\begin{eqnarray}
  \nonumber  I_{k,\widehat{\omega},\widehat{\vartheta}}(|\eta|) & \le & C  \int \displaylimits_{ |\eta|}^{|\eta| \over \beta}  \hspace{-1.0mm}{ e^{-|y|^\gamma- |y|} \over  |y|^{k+d-1-\kappa}} \    e^{|y| } |y|^{-\kappa} \left(|\widehat{\omega} | \star |\widehat{\vartheta} | \right)(y)  \ d y \\
  \nonumber  & \le & C \left(  \int \displaylimits_{ |\eta|}^{|\eta| \over \beta} { e^{-s  |y|^\gamma- s   |y|} \over  |y|^{s \left(k+d-1-\kappa \right)} } \ d y  \right)^{1 \over s}   \left( \int_{\RR^d}   e^{r  |y| } |y|^{ -r\kappa }  \left(| \widehat{\omega} | \star |\widehat{\vartheta} | \right)^{r}(y) \ d y \right)^{1 \over r} \\
  \nonumber  & \le & C  e^{-|\eta|}  \left(  \int \displaylimits_{ |\eta|^\gamma}^{|\eta|^\gamma \over \beta^\gamma} { e^{-s t} \over  t^{{s\over \gamma} \left(k+d-1-\kappa \right)  +{\gamma-d \over \gamma}} } \ d t  \right)^{1 \over s} \ \| \widehat{\omega} \|_{u_\nu,p}  \| \widehat{\vartheta} \|_{u_\nu,p} \\
  \label{gamma_f}  & \le & C  e^{-|\eta|}   \| \widehat{\omega} \|_{u_\nu,p} \| \widehat{\vartheta} \|_{u_\nu,p} \left( \Gamma\left({d \over \gamma}-{s  \over \gamma}  \left(k+d-1-\kappa \right), s t^\gamma \right) \Bigg|_{ |\eta| \over \beta}^{|\eta|} \right)^{1 \over s}.
\end{eqnarray}
Set
\begin{equation}
    \label{lk} l(k) := {d \over s}-\left(k+d-1-\kappa \right).
  \end{equation}
Under hypothesis $(\ref{H2})$,
\begin{equation}
\label{H5} 0<s l(c),
\end{equation}
then the difference of the incomplete Gamma functions in $(\ref{gamma_f})$, entering the definition of $w_k$, can be bounded from above as follows:
\begin{eqnarray}
  \label{w_k} w_k(\eta) & \le &  \left\{  \left({1 \over \beta^{l(k)s} }-1  \right)^{1 \over s} e^{-|\eta|-|\eta|^\gamma} |\eta|^{l(k)}, \hspace{19.5mm}  {|\eta|/\beta}< 1, \atop  e^{-|\eta|-|\eta|^\gamma} |\eta|^{l(k)-{\gamma \over s}}  \left(1- e^{ -s \left({1 \over \beta^\gamma} -1  \right) |\eta|^\gamma}   \right)^{1 \over s}, \ {|\eta|/\beta}  \ge  1.  \right.
\end{eqnarray}
Now, consider the integral
$$ \int \displaylimits_{\RR^d} w_k(\eta)^p e^{p|\eta|^\gamma} |\eta|^{p k} u_\nu^p(\eta) d \eta =  \int \displaylimits_{\RR^d} w_k(\eta)^p e^{p|\eta|^\gamma+p |\eta|} |\eta|^{p (k+\nu)} d \eta$$ 
The conditions of convergence of this integral at infinity  and at zero coincide with hypothesis $(\ref{H2})$, and are independent of $k$.
\end{proof}

We have, therefore,
\begin{equation}
  \label{Nestimate} \left| N(\eta) \right| \le  C \ \| \widehat{\vartheta} \|_{u_\nu,p}^{2} w_c(\eta) e^{|\eta|^\gamma} |\eta|^c,
\end{equation}
and, according to Lemma $\ref{Ik}$,
\begin{equation}
  \label{cIeta}    \|N\|_{u_\nu,p}^p  \le  C  \ \| \widehat{\vartheta} \|_{u_\nu,p}^{2 p}.
\end{equation}
\end{proof}

\begin{remark}\label{rmk1}
  For $d=3$ and $\gamma=2$, the set of condition $(\ref{H1})-(\ref{H2})$ is non-empty. Specifically, for the choice $p=2$, the conditions are satisfied by
  \begin{eqnarray}
&& {1 \over 2} < \nu \le 1, \quad r \ge {4 \over 2 \nu -1}, \quad \kappa={3 \over r} -2 \nu \\
&& 1 <\nu \le {3 \over 2}, \quad {4 \over 2 \nu -1}< r \le {1 \over \nu-1}, \quad \kappa={3 \over r} -2 \nu.
  \end{eqnarray}
\end{remark}

\vspace{2mm}

\section{A-priori bounds at the trivial fixed point}\label{trivapb}

In the next two sections we will construct {\it a-priori} bounds - a precompact convex renormalization invariant set in two settings.

The first instance, discussed in this Section $\ref{trivapb}$,  which we choose to call {\it trivial a-priori} bounds, will be that of a precompact invariant  set defined through a bounded norm condition in a weighted $\fL^p_{u_\nu}$-space. Unfortunately, thus constructed set necessarily contains only the trivial fixed point.

In Section $\ref{nontrivapb}$, however, we will build a precompact invariant set of azimuthal vector fields which does not contain the zero vector field. This construction will explicitly provide a positive function $\pmb{\psi}$ with a singularity of type $|\eta|^{\sigma}$ in zero which serves as an upper bound on the absolute value of the azimuthal component of the vector fields in the set; such function  $\pmb{\psi} \in \fL^p_{u_\nu}$ whenever $p(\nu+\sigma)+d>0$.

\subsection{Existence of an invariant bounded convex set} \label{bounded_set}
We will now restrict the invariant domain of the operator $\cR_\beta$ to a bounded subset  in $\fRL^p_{u_\nu}$.

\begin{proposition} \label{inv2}
  Let $r>1$, $p>1$, $\nu$ and $\kappa$  satisfy the hypothesis of Proposition $\ref{inv1}$, and additionally
\begin{equation}
  \tag{H3} \label{H3}  {\gamma p (r-1) \over r}-{d (r-p)\over r} -p(\nu + \kappa +1)>\gamma,
\end{equation}
\begin{equation}
  \tag{H4} \label{H4}  p > {r \over r-1}.
\end{equation}
Then for every $0<\beta<1$ there exists $A>0$, such that the convex set
 \begin{equation}
   \cE_{u_\nu,p,A}:=\{\vartheta \in \fRL^p_{u_\nu}: \| \widehat{\vartheta} \|_{u_\nu,p}  \le A\}
 \end{equation}
where, as before, $\widehat{\vartheta}$ is the bounding function in $\fR^p_{u\nu}$, is invariant under $\cR_\beta$.
\end{proposition}

\begin{proof}
Consider the bound $(\ref{Nestimate})$. Set
 \begin{equation}
\label{vc} v_c(\eta):=w_c(\eta) e^{|\eta|^\gamma +|\eta|} |\eta|^{c+\nu}  
  \end{equation}
 We can use the bound $(\ref{w_k})$ in the function $v_c$. Additionally, we will somewhat simplistically bound the difference 
 $$\left(1- e^{-s \left({1 \over \beta^\gamma}-1 \right) |\eta|^\gamma }  \right)^{1 \over s}$$
by $C \  \left({1 / \beta^\gamma}-1 \right)^{1 \over s} \ |\eta|^{\gamma \over s}$  for $|\eta| \le  \left({1 / \beta^\gamma}-1 \right)^{ -{1 \over \gamma} }$, and by $1$ for  $|\eta| > \left({1 / \beta^\gamma}-1 \right)^{ -{1 \over \gamma} }$. Therefore,
 \begin{equation}
   \label{v_c} v_c(\eta)  \le  \begin{cases}
     &   \left({1 \over \beta^{l(c)s} }-1  \right)^{1 \over s} |\eta|^{l(c)+c+\nu}, \hspace{10.4mm} {|\eta|}< \beta, \\
     &   \left({1 \over \beta^\gamma}-1 \right)^{1 \over s} |\eta|^{l(c)+c+\nu},  \hspace{6mm}  \beta \le {|\eta| } \le   \left( {1 \over \beta^\gamma} -1 \right)^{-{1 \over \gamma}}, \\
     &    |\eta|^{l(c)+c+\nu-{\gamma \over s}} , \hspace{29.5mm} {|\eta|} > \left( {1 \over \beta^\gamma} -1 \right)^{-{1 \over \gamma}}.
   \end{cases}
 \end{equation}
 Therefore,
 \begin{eqnarray}
   \nonumber   \|N\|_{u_\nu,p}^p  \hspace{-1.5mm}   &   \hspace{-1.5mm}   \le  \hspace{-1.5mm}   &    \hspace{-1.5mm}  C   \| \widehat{\vartheta} \|_{u_\nu,p}^{2 p} \hspace{-1mm} \left( \hspace{-1.5mm} \left({1 \over \beta^{s l(c)}}-1 \right)^{p \over s}   \hspace{-4mm}  \int \displaylimits_{{|\eta| } <   \beta}   \hspace{-2mm}  |\eta|^{p(l(c)+c+\nu)} \ d \eta +     \left({1 \over \beta^\gamma}-1 \right)^{p \over s} \  \hspace{-14mm}  \int \displaylimits_{{\beta \le |\eta| } \le   \left( {1 \over \beta^\gamma} -1 \right)^{-{1 \over \gamma}}}   \hspace{-11mm}  |\eta|^{p(l(c)+c+\nu)} \ d \eta + \right.\\
   \nonumber  \hspace{-1.5mm}   &   \hspace{-1.5mm}  \phantom{ \le}  \hspace{-1.5mm}   &    \hspace{-1.5mm} \phantom{  C   \| \widehat{\vartheta} \|_{u_\nu,p}^{2 p}\left( \hspace{-1.5mm} \left({1 \over \beta^{s l(c)}}-1 \right)^{p \over s}   \hspace{-4mm}  \int \displaylimits_{{|\eta| } <   \beta}   \hspace{-2mm}  |\eta|^{p(l(c)+c+\nu)} \ d \eta \right. } + \hspace{-6mm}\left. \int \displaylimits_{|\eta|  >  \left({1 \over \beta^\gamma}-1 \right)^{ -{1 \over \gamma} }} \hspace{-6mm} |\eta|^{p( l(c) +c+\nu)-{\gamma  p\over s}}  \ d \eta    \right).
\end{eqnarray}
 All integrals converge under the hypothesis $(\ref{H2})$.
 \begin{equation}
  \nonumber   \|N\|_{u_\nu,p}^p   \le  C  \ \|  \widehat{\vartheta} \|_{u_\nu,p}^{2 p} \left( \hspace{-0.7mm}  \left({1 \over \beta^{s l(c)}}-1 \right)^{p \over s} \hspace{-2mm} \beta^{d{r-p \over r} +p(1+\kappa+\nu)} + \left({1 \over \beta^\gamma}-1 \right)^{{p \over s}+d {p-r \over r \gamma } -{p \over \gamma}(1+\kappa+\nu)  } \right).
\end{equation}
The power
\begin{equation}
\label{powers}  m:={p (r-1) \over r}- {p \over \gamma} (\kappa +1 + \nu)-{d (r-p) \over r \gamma}
\end{equation}
of  $\left({1 / \beta^\gamma}-1 \right)$ is larger than $1$, according to $(\ref{H3})$ and $(\ref{H4})$. So is the power $p/s$ of $\left({1 / \beta^{s l(c)}}-1 \right)$ 
Finally,
\begin{equation}
  \label{bound6} \| \widehat{\cR_\beta[\vartheta]} \|_{u_\nu,p}  \le \beta^{c+\nu+{d \over p}} \| \widehat{\vartheta} \|_{u_\nu,p} + C_1 \ \| {\widehat{\vartheta}} \|_{u_\nu,p}^{2}  \left(  \left({1 \over \beta^{s l(c)}}-1 \right)^{1 \over s} \beta^{\gamma \left( {1 \over s} -{m \over p}\right)}+   \left({1 \over \beta^\gamma}-1 \right)^{m \over p} \right),
\end{equation}
where, as before, $\widehat{\cR_\beta[\vartheta]}$ stands for a function in $\fR^p_{u\nu}$ bounding the function $\cR_\beta[\vartheta]$ in $\fRL^p_{u_\nu}$.

Notice that by  definition $(\ref{powers})$ of $m$, and by the hypothesis $(\ref{H2})$, $p/s>m$, hence, the power $\gamma ( 1/s -m / p)$ of $\beta$ is positive, and for every choice of constants $\nu$, $p$, $r$, $\kappa$ satisfying the hypothesis, there exists $A=A(\beta)>0$, such that
\begin{equation}
  \label{bound7}  \beta^{c+\nu+{d \over p}} A + C_2 A^{2} \left({1 \over \beta^\gamma} - 1\right)^{m \over p}  \le A    \Leftrightarrow    A \le C_* \ \left({1 \over \beta^\gamma}-1 \right)^{- {m \over p}} \left( 1-  \beta^{c+\nu+{d \over p}}\right).
\end{equation}
additionally, if we require that  $\| \widehat{\cR_\beta[\vartheta]}\|_{u_\nu,p} \le \alpha \| \widehat{\vartheta}\|_{u_\nu,p}$ for some $\beta^{c+\nu+{d \over p}}<\alpha<1$, then
\begin{equation}
  \label{bound8}    \beta^{c+\nu+{d \over p}}  + C_2 A \left({1 \over \beta^\gamma} - 1\right)^{m \over p} \hspace{-2mm} \le \alpha    \iff    A \le C_* \left({1 \over \beta^\gamma}-1 \right)^{- {m \over p}} \hspace{-1mm} \left( \alpha-  \beta^{c+\nu+{d \over p}}\right).
\end{equation}
\end{proof}

\begin{remark}
The upper bound on $A(\beta)$ degenerates to $0$ as $\beta \rightarrow 1$.
\end{remark}
  
In fact, a stronger result for the radial functions holds
\begin{proposition} \label{trivial}
  Let $r>1$, $p>1$, $\nu$ and $\kappa$  satisfy the hypothesis of Propositions $\ref{inv1}$ and $\ref{inv2}$.   Then for every $0<\beta<1$ there exists $\widehat{A}>0$, such that $\cR_\beta$ is a contraction on the convex set
 \begin{equation}
   \widehat \cE_{u_\nu,p,\widehat A}:=\{\widehat{\vartheta} \in \fR^p_{u_\nu}: \| \widehat{\vartheta} \|_{u_\nu,p}  \le \widehat A\}.
 \end{equation}
In particular, the trivial solution is the only renormalization fixed point.
\end{proposition}
\begin{proof}
Consider the difference $\cR_\beta[\widehat{\vartheta}_1]-\cR_\beta[\widehat{\vartheta}_2]$ with $\widehat{\vartheta}_i \in  \widehat \cE_{u_\nu,p,\widehat A}$. Largely repeating the argument of Proposition $\ref{inv2}$,  we get that
\begin{eqnarray}
  \nonumber  \|\cR_\beta[\widehat{\vartheta}_1] \hspace{-0.4mm}  - \hspace{-0.4mm} |\cR_\beta[\widehat{\vartheta}_2] \|_{u_\nu,p} \hspace{-1mm}   & \hspace{-1mm} \le \hspace{-1mm} & \hspace{-1mm} \beta^{c+\nu+{d \over p}} \| {\vartheta}_1\hspace{-0.4mm}- \hspace{-0.4mm} {\vartheta}_2 \|_{u_\nu,p} \hspace{-0.7mm} + \hspace{-0.7mm} C_1  \| \widehat{\vartheta_1 \hspace{-0.4mm} -\hspace{-0.4mm} \vartheta_2} \|_{u_\nu,p} \hspace{-0.4mm} \left( \| \widehat{\vartheta}_1 \|_{u_\nu,p}\hspace{-0.4mm} + \hspace{-0.4mm}\|\widehat{\vartheta}_2 \|_{u_\nu,p} \right)  \times \\
  \nonumber && \hspace{18.5mm} \times \left(  \left({1 \over \beta^{s l(c)}}-1 \right)^{1 \over s} \beta^{\gamma \left( {1 \over s} -{m \over p}\right)}+   \left({1 \over \beta^\gamma}-1 \right)^{m \over p} \right),
\end{eqnarray}
where $C_1$ can be taken identical to that in $(\ref{bound6})$. We obtain, that a sufficient condition for the operator $\cR_\beta$ to be a contraction is
$$ \beta^{c+\nu+{d \over p}} + 2 C_2 \widehat A \left({1 \over \beta^\gamma} -1 \right)^{m \over p}  \le 1  \iff     \widehat A \le {C_* \over 2} \left({1\over \beta^\gamma}-1  \right)^{-{m \over p}} \left(1-\beta^{c+\nu+{d \over p}}   \right),$$
with $C_*$ as in $(\ref{bound7})$.
\end{proof}

Notice that the space $\fR^p_{u\_nu}$ can not be substituted by $\fRL^p_{u\_nu}$ in the above argument, since in  the convolution term will necessarily have a factor $\|\widehat{\vartheta_1-\vartheta}_2 \|_{u_\nu,p}$, and the resulting bound would not be a bound involving exclusivley $\| \vartheta_1 - \vartheta_2\|_{u_\nu,p}$.

\begin{corollary}(\it \underline{Non-existence of radial Leray blow ups with small amplitude}). \label{trivial}
  There exists $\widehat A$, such that there are no non-trivial solutions to $(\ref{eq:integral:NSgamma})$   of the form
  $$v(y,t)=(T-t)^{d+1-\gamma \over \gamma} \widehat{\vartheta}\left(y (T-t)^{1 \over \gamma}  \right)$$
with $\widehat{\vartheta} \in  \fR^p_{u_\nu}$ such that $\| \widehat{\vartheta} \|_{u_\nu,p} < \widehat A$.
\end{corollary}

\subsection{Existence of an invariant bounded convex precompact set}  \label{precompactness}

In what follows we will  demonstrate that the set $\cE_{u_\nu,p,A}$ constructed in Proposition $\ref{inv2}$ has a precompact convex subset. The key result that  allows us to claim precompactness, and, consequently, existence of fixed points in the closure of this subset of  $\cE_{u_\nu,p,A}$, is the  following extension of the classical Frechet-Kolmogorov-Riesz-Weil Compactness Theorem (see \cite{COY}, Theorem 2.8. We quote only the sufficient conditions given in the theorem):

\medskip

\begin{FKRW}
  Let $p \in (0, \infty)$, and let $w$ be a weight on $\RR^d$ such that $w$, $w^{-\lambda} \in L_{loc}^1(\RR^d)$ for some $\lambda \in (0,\infty)$.    Let $\cE$ be a subset of  $\fL^{p}_w$, and let $\tau_\delta f$ denote the translation by $\delta$, $(\tau_\delta f)(x):=f(x-\delta)$. The subset $\cE$ is relatively compact if the following properties hold:

    \medskip
    
    \begin{itemize}
       \item[$1)$] {\it Uniform boundedness}: $sup_{f \in \cE} \| f \|_{w,p} < \infty$;

      \vspace{2mm}
         
       \item[$2)$] {\it Equicontinuity}: $\lim_{|\delta| \rightarrow 0} \| \tau_\delta f - f \|_{w,p}=0$ uniformly in $\cE$;

      \vspace{2mm}
      
    \item[$3)$] {\it Equitightness}: $\lim_{r \rightarrow \infty} \int_{|x|>r} |f|^p w^p =0$ uniformly on $\cE$.
    \end{itemize}
      
\end{FKRW}

\medskip

By analogy with the renormalization theory in complex dynamics, existence of a renormalization-invariant precompact convex set will be called {\it a-priori} bounds.

\medskip

In the following Proposition we construct an {\it equitight} renormalization invariant  subset of  $\cE_{u_\nu,p,A}$.
\medskip

\begin{proposition} \label{equitight} \underline{(\it Equitightness.)} Suppose that $\nu$ and $p$ satisfy $(\ref{H1})-(\ref{H4})$ for some $\kappa$ and $r$. Then there exists $M=M(\beta)>0$ such that  the convex subset $\cE_{u_\nu,p,A}^M$ of  $\cE_{u_\nu,p,A}$, defined by
\begin{equation}
\cE_{u_\nu,p,A}^M:=\left\{ \vartheta \in \cE_{u_\nu,p,A}: \left( \int_{|\eta|>\rho} |\vartheta(\eta)|^p u_\nu(\eta)^p d \eta \right)^{1 \over p}  \le  {M \over \rho^{\gamma m \over p}}, \ \rho>\beta \right\},
\end{equation}
is invariant under $\cR_\beta$.
\end{proposition}
\begin{proof}
  We first address the linear term,
  $$L(\eta)=\beta^c e^{-\left({1 \over \beta^\gamma}-1\right) |\eta|^\gamma} \vartheta \left({\eta \over \beta}  \right),$$
  in the renormalization operator.  Assume that
$$  \int \displaylimits_{|\eta| > \rho} |\vartheta(\eta)|^p u_\nu^p(\eta) \ d \eta  \le  {M^p \over \rho^{p n}}$$
for some positive $n$.  Then
  \begin{eqnarray}
    \nonumber \int \displaylimits_{|\eta| > \rho} |L(\eta)|^p u_\nu^p(\eta) \ d \eta  & \le & \beta^{p (c + \nu) } e^{-p \left({1 \over \beta^\gamma}-1\right) \rho^\gamma- p \left({1 \over \beta} -1 \right) \rho} \int_{|\eta|>\rho} \left|\vartheta \left({ \eta  \over \beta  }\right)\right|^p e^{p{|\eta| \over \beta}} {|\eta|^{p \nu} \over \beta^{p \nu}} \ d \eta \\
    \nonumber  & \le & \beta^{p (c + \nu) +d } e^{-p \left({1 \over \beta^\gamma}-1\right) \rho^\gamma- p \left({1 \over \beta} -1 \right) \rho} \int_{|z|>{\rho \over \beta}} |\vartheta (z)|^p e^{p|z} |z|^{p \nu} \ d z \\
      \label{Lpdecay}  & \le & \beta^{p (c + \nu) +d } e^{-p \left({1 \over \beta^\gamma}-1\right) \rho^\gamma- p \left({1 \over \beta} -1 \right) \rho} {M^p \beta^{n p} \over \rho^{p n}}.
  \end{eqnarray}
To compute the bound on  $ \int_{|\eta| > \rho} |N(\eta)|^p u_\nu^p(\eta) \ d \eta$ we reuse the estimate $(\ref{v_c})$, also taking into account that $\rho>\beta$. First, suppose $\rho <(1/\beta^\gamma-1)^{-1/\gamma}$. Then
 \begin{eqnarray}
      \nonumber    \int \displaylimits_{|\eta| > \rho} \hspace{-2mm} |N(\eta)|^p u_\nu^p(\eta) \ d \eta \hspace{-2.5mm} & \hspace{-2.5mm} \le \hspace{-2.5mm}  &  \hspace{-2.5mm}  C   \| { \widehat{\vartheta}} \|_{u_\nu,p}^{2 p}\hspace{-1mm} \left( \hspace{-1.5mm}  \left({1 \over \beta^\gamma}-1 \right)^{p \over s} \  \hspace{-15mm}  \int \displaylimits_{\rho<{|\eta| } \le   \left( {1 \over \beta^\gamma} -1 \right)^{-{1 \over \gamma}}}   \hspace{-12mm}  |\eta|^{p(l(c)+c+\nu)} \ d \eta  + \hspace{-9mm}  \int \displaylimits_{|\eta|  \ge   \left({1 \over \beta^\gamma}-1 \right)^{ -{1 \over \gamma} }} \hspace{-10mm} |\eta|^{p( l(c) +c+\nu)-{\gamma  p\over s}}  \ d \eta  \hspace{-1mm}   \right) \\
   \nonumber    & \le &  C  \ \|{ \widehat{\vartheta}} \|_{u_\nu,p}^{2 p} \int \displaylimits_{|\eta| > \rho}  |\eta|^{p( l(c) +c+\nu)-{\gamma  p\over s}}  \ d \eta\\
   \label{Npdecay}    & \le &  C  \ A^{2 p}  \ \rho^{-\gamma m},
 \end{eqnarray}
 where $m$ is defined in $(\ref{powers})$ and, as before,  $\widehat{\vartheta} \in \fR^p_{u_\nu}$ is a bounded function on $\vartheta \in \fRL^p_{u_\nu}$.

We can combine estimates $(\ref{Lpdecay})$ and $(\ref{Npdecay})$ to conclude that the sufficient condition for the invariance of the set   $\cE_{u_\nu,p,A}^M$ is 
\begin{equation}
\beta^{c + \nu +{d \over p} } e^{-\left({1 \over \beta^\gamma}-1\right) \rho^\gamma- \left({1 \over \beta} -1 \right) \rho} {M \beta^{n} \over \rho^{n}} + C \ A^2  \rho^{-\gamma {m \over p}} \le {M \over \rho^n}.
\end{equation}
We can see that the inequality holds if $\gamma m /p \ge n$, and if $M$ is chosen sufficiently large.

If $\rho \ge (1/\beta^\gamma-1)^{-1/\gamma}$, then we have right away,
 \begin{eqnarray}
   \nonumber    \int \displaylimits_{|\eta| > \rho} |N(\eta)|^p u_\nu^p(\eta) \ d \eta    & \le &  C  \ \| { \widehat{\vartheta}} \|_{u_\nu,p}^{2 p} \int \displaylimits_{|\eta|>\rho}  |\eta|^{p( l(c) +c+\nu)-{\gamma  p\over s}}  \ d \eta\\
   \nonumber   & \le &  C  \ A^{2 p}  \ \rho^{-\gamma m},
 \end{eqnarray}
and an identical conclusion follows.
\end{proof}

\medskip

The key result in proving precompactness is the following

\medskip

\begin{proposition} \underline{(\it Equicontinuity.)} \label{equicont}

\medskip
Suppose that $\nu$ and $p$ satisfy $(\ref{H1})-(\ref{H4})$ for some $\kappa$ and $r$, and additionally
\begin{equation}
  \tag{H5} \label{H5} p<{d (r-p)\over r} +p(\nu + \kappa +1) \hspace{2mm} {\rm and} \hspace{2mm} \nu<1,
\end{equation}
\begin{equation}
  \tag{H6} \label{H6} c-\theta+\nu+{d \over p} >0,
\end{equation}
for some
\begin{equation}
  \tag{H7} \label{H7} 0 < \theta  \le  \min\left\{1,\gamma-1, {\gamma m \over p(\gamma-1)}, \nu, {1 \over p} \right\}
\end{equation}
with $m$ as in $(\ref{powers})$.

Then, given $\delta_0>0$ and $0<\beta <1$, $\chi>\theta$ and $A$ as in $(\ref{bound7})$, there exist  constants $K>0$ and $k>0$,
\begin{equation} \label{Kineqq}
  K \ge    {C \ A  \ (1-\beta)^k \over 1-\beta^{c+\nu+{d \over p}-\theta}- C_2 A \left({1 \over \beta^\gamma}-1   \right)^{m \over p}  },
 \end{equation}
where $C_2$ is as in $(\ref{bound7})$, such that the subset $\cE^{M,K,\delta_0}_{u_\nu,p,A}$ of $\cE_{u_\nu,p,A}^M$ of H\"older  functions,
$$| \vartheta(\eta)-\tau_\delta \vartheta(\eta)|  < \omega_\vartheta(\eta) |\delta|^\theta,$$
where $|\delta| < \delta_0$ and  $\omega_\vartheta \in R^p_{u_\nu}(\RR^d,\RR_+)$ with $\| \omega_\vartheta \|_{u_\nu,p} \le K$, is $\cR_\beta$-invariant.

\end{proposition}

\begin{remark}  \underline{(\it Modulus of continuity.)} \label{modulus}
  It is clear that from the previous Proposition, that for any $\delta >0$, 
  $$|\vartheta(\eta) - \tau_\delta \vartheta(\eta)| \le \omega(\eta) \sigma(\delta),$$
  where $\omega \in L^p_{u_\nu}(\RR^d, \RR_+)$ and
  $$\sigma(\delta)=\left\{|\delta|^\theta, \hspace{5.5mm} |\delta| \le \delta_0, \atop  {\rm const},  \  |\delta|> \delta_0   \right.$$
\end{remark}

The long proof of Proposition $\ref{equicont}$ is marginally different from the proof of Propositon $\ref{non_triv_equicont}$, and will not be included here.

As it turns out, equicontinuity is of limited usefulness, since the only fixed point in the constructed precompact set is trivial.

\begin{remark}  \underline{(\it Triviality of the a-priori bounds.)} \label{trivial}
  According to $(\ref{Kineqq})$, $A$ necessarily satisfies
\begin{eqnarray}
  \nonumber   1-\beta^{c+\nu+{d \over p}-\theta}- C_2 A \left({1 \over \beta^\gamma}-1   \right)^{m \over p}>0  &\implies & A < C_* \left({1 \over \beta^\gamma} -1   \right)^{-{m \over p}} \left(1-\beta^{c+\nu+{d \over p} -\theta}   \right)\\
\nonumber  &\Longleftrightarrow& A < C_* \left({1 \over \beta^\gamma} -1   \right)^{-{m \over p}} \left(\alpha-\beta^{c+\nu+{d \over p}}   \right),
  \end{eqnarray}
where $C_*$ is as in $(\ref{bound7})$ and
$$\alpha =1-\left(\beta^{c+\nu+{d \over p}-\theta}-\beta^{c+\nu+{d \over p}}\right) < 1$$
for any fixed $\theta>0$.  This is a condition of type $(\ref{bound8})$, which means that
    $$ \| \widehat{\cR_\beta^n[\vartheta]} \|_{u_\nu,p} \le \alpha^n  \|\widehat{\vartheta} \|_{u_\nu,p}$$
for all $\vartheta \in \cE^{M,K,\delta_0}_{u_\nu,p,A}$. Consequently, the only fixed point in this precompact set is $0$.
\end{remark}

The immediate consequence of the preceding discussion is the non-existence of the non-trivial self-similar blow up solutions of small, but definite norm of the initial condition.

\begin{corollary}  \underline{(\it Triviality of blow ups.)} \label{trivial}
There exists a constant $D>0$ such that the equation $(\ref{eq:integral:NSgamma})$ admits no non-trivial solutions in $\fRL_{u_\nu}^p$ with the norm less than $D$.
\end{corollary}
\begin{proof}
  Set
$$F(\beta)= C_* \left( {1 \over \beta^\gamma}-1    \right)^{-{m \over p}} \left(1-\beta^{c+\nu+{d \over p}-\theta}  \right),$$
and consider 
  $$D= \sup_{\beta \in (0,1)} F(\beta),$$
Suppose $\beta_m$ is the maximizer.  Notice, $0<\beta_m<1$, since
$$\lim_{ \beta \rightarrow 0} F(\beta)=\lim_{ \beta \rightarrow 1} F(\beta)=0.$$

  Then any solution
  $$v(y,t)=i (T-t)^{d+1-\gamma \over \gamma} \vartheta\left(y (T-t)^{1 \over \gamma}  \right)$$
  with $\vartheta \in \fRL_{u_\nu}^p$, has to satisfy the equation  $(\ref{eq:integral:NSgamma})$ for $t=T(1-\beta_m^\gamma)$, i. e. such $\vartheta$ has to coincide with a fixed point $\vartheta_{\beta_m}$. By Remark $\ref{trivial}$ any such fixed point with $\|\vartheta_{\beta_m} \|_{u_\nu,p} <D$ is trivial. 
\end{proof}

\subsection{Trivial a-priori bounds}\label{apriori_subsection}

We would like to collect the hypotheses  $(\ref{H1})-(\ref{H7})$ again:

\begin{equation}
\label{H8} \tag{H8} \begin{cases}
  & {1 \over r}={2 \over p}+ {2 \nu+\kappa  \over d  }-1, \quad {r \over r-1} <p \le 2 r, \quad 0 < \theta \le \min\left\{\gamma-1,  {\gamma m \over p \ (\gamma-1)}, \nu, {1 \over p} \right\}, \\ 
  &  { d-1 \over 2} \left(1 - {2 \over p} \right) < \nu<1, \hspace{6mm} d \ {p-1 \over p}+1+\theta-\gamma  < \nu <  d \ {p-1 \over p}, \\
  & \max\{\nu,\kappa\} >0,  \quad \kappa < {d \over r}, \quad   \nu+\kappa \ge (d-1) \left({1 \over r} - {1 \over p} \right), \\
  & \max\{d+p(c+\nu),p\}<{d (r-p)\over r} +p(\nu + \kappa +1) <{\gamma p (r-1) \over r}- \gamma.
\end{cases}
\end{equation}


In subsections $\ref{bounded_set}$ and $\ref{precompactness}$ we have proved the following theorem
\begin{apthm} \label{aptheorem}
Let $\gamma$, $p$, $\nu$ and $\theta$ satisfy the hypothesis $(\ref{H8})$ for some $\kappa$ and $r$. Set
\begin{equation}
  u_\nu(x)=e^{|x|} |x|^\nu.
\end{equation}    
Then for every $\beta \in (0,1)$ there are constants $A>0$, $M>0$, $K>0$, $\delta_0>0$ and a convex set $\cE^{M,K,\delta_0}_{u_\nu,p,A}$ of functions $\vartheta \in \fL^p_{u_\nu}$  satisfying
\begin{itemize}
\item[$1)$] $\cE^{M,K,\delta_0}_{u_\nu,p,A}$ is $\cR_\beta$-invariant;
\item[$2)$] $\| \vartheta \|_{u_\nu,p} \le A$ for all $\vartheta \in \cE^{M,K,\delta_0}_{u_\nu,p,A}$;
\item[$3)$] $\vartheta \in \cE^{M,K,\delta_0}_{u_\nu,p,A}$ is H\"older-continuous:
  $$|\vartheta(\eta)-\vartheta(\eta-\delta)| \le \omega_\vartheta(\eta) |\delta|^\theta,$$
for all $|\delta|<\delta_0$,  with $\omega_\vartheta \in L^p_{u_\nu}$ satisfying  $\| \omega_\vartheta \|_{u_\nu,p} \le K$;
\item[$4)$] $\vartheta \in \cE^{M,K,\delta_0}_{u_\nu,p,A}$ satisfies
  $$\int \displaylimits_{|\eta|>r} \vartheta(\eta)^p u_\nu(\eta)^p \ d \eta  \le {M \over r^{\gamma \over p}};$$
\item[$5)$] the $\fL^p_{u_\nu}$-closure of $\cE^{M,K,\delta_0}_{u_\nu,p,A}$ is compact. 
\end{itemize}
\end{apthm}

The operator $\cR_\beta$ is a continuous map of $\cE^{M,K,\delta_0}_{u_\nu,p,A}$ into itself, therefore, under hypothesis $(\ref{H8})$, by Tychonoff fixed point theorem $\cR_\beta$ has a fixed point $\vartheta_\beta \in  \cE^{M,K,\delta_0}_{u_\nu,p,A}$. By Remark $\ref{trivial}$ this fixed point is unique and trivial.

\section{Non-trivial a-priori bounds}\label{nontrivapb}

\subsection{Existence of an invariant bounded convex set of non-trivial azimuthal fields} \label{nontriv_bounded_set}
Unlike in the previous Section, we will now consider azimuthal fields whose only non-zero component is the azimuthal one: $\vartheta= i \psi$, and
$$\psi(\eta) = \psi_{d-1}(\eta) e_{d-1}(\eta).$$

We will now derive a more specific form of the action of the renormalization operator on the azimuthal component of the field.

We project equation $(\ref{operatorR3})$ on the direction $e_{d-1}(\eta)$:
\begin{align}
  \nonumber\cR_\beta[\psi]_{d \shortminus 1}(\eta)  &  = \beta^c e^{|\eta|^\gamma \left(1\shortminus {1 \over \beta^\gamma}  \right)} \psi_{d \shortminus 1} \hspace{-0.7mm}   \left(\eta \over \beta \right) \hspace{-0.5mm}   - \hspace{-0.5mm}   \gamma \hspace{-1.3mm}   \int \displaylimits_{1}^{1 \over \beta} \hspace{-0.5mm}   { e^{|\eta|^\gamma \left(1 \shortminus t^\gamma\right)} \over  t^{c} } \hspace{-2mm}    \int \displaylimits_{\RR^d} \hspace{-1.3mm}   \eta \hspace{-0.3mm}   \cdot \hspace{-0.3mm}    \psi  \hspace{-0.5mm}  \left(\eta t \hspace{-0.5mm}   - \hspace{-0.5mm}    x \right) \psi(x) \hspace{-0.3mm}   \cdot \hspace{-0.3mm}   e_{d\shortminus 1}(\eta)   \ d x   \ d t \\
  \nonumber  &  =  \beta^c e^{|\eta|^\gamma \left(1\shortminus {1 \over \beta^\gamma}  \right)} \psi_{d \shortminus 1} \hspace{-0.7mm}   \left(\eta \over \beta \right) \hspace{-0.5mm}   - \hspace{-0.5mm}   \gamma \hspace{-1.3mm}   \int \displaylimits_{1}^{1 \over \beta} \hspace{-0.5mm}   { e^{|\eta|^\gamma \left(1 \shortminus t^\gamma\right)} \over  t^{c+ 1} } \hspace{-2mm}    \int \displaylimits_{\RR^d} \hspace{-1.3mm}   \eta t \hspace{-0.3mm}   \cdot \hspace{-0.3mm}    \psi  \hspace{-0.5mm}  \left(\eta t \hspace{-0.5mm}   \shortminus \hspace{-0.5mm}    x \right) \psi_{d\shortminus 1}(x)   \cos \hat{x}_{d\shortminus 1}   \ d x   \ d t \\
   \nonumber &=  \beta^c e^{|\eta|^\gamma \left(1\shortminus {1 \over \beta^\gamma}  \right)} \psi_{d \shortminus 1} \hspace{-0.7mm}   \left(\eta \over \beta \right) \hspace{-0.5mm}   - \hspace{-0.5mm}   \gamma \hspace{-1.3mm}   \int \displaylimits_{1}^{1 \over \beta} \hspace{-0.5mm}   { e^{|\eta|^\gamma \left(1 \shortminus t^\gamma\right)} \over  t^{c+ 1} } \hspace{-2mm}    \int \displaylimits_{\RR^d} \hspace{-1.3mm}   x \hspace{-0.3mm}   \cdot \hspace{-0.3mm}    \psi  \hspace{-0.5mm}  \left(\eta t \hspace{-0.5mm}   \shortminus \hspace{-0.5mm}    x \right) \psi_{d\shortminus1}(x)   \cos \hat{x}_{d\shortminus 1}   \ d x   \ d t, 
\end{align}
where we have used that $\eta t \cdot \psi(\eta t -x)= (\eta t -x) \cdot \psi(\eta t -x) +x \cdot \psi(\eta t -x)=x \cdot \psi(\eta t -x)$.

Denote $x_a$ the projection of vector $x$ on the azimuthal plane. Then 
\begin{align}
  \nonumber\cR_\beta[\psi]_{d-1}(\eta)  & =  \beta^c e^{|\eta|^\gamma \left(1- {1 \over \beta^\gamma}  \right)} \psi_{d-1} \left(\eta \over \beta \right) -\gamma  \int \displaylimits_{1}^{1 \over \beta}  { e^{|\eta|^\gamma \left(1-t^\gamma\right)} \over  t^{c+1} }   \int \displaylimits_{\RR^d}  \psi_{d-1}(x) \cos\hat x_{d-1}   \times \\
  \nonumber & \phantom{=} \times  \psi_{d-1}(\eta t  -   x)   {x_a \cdot   (|x_a| \sin \hat x_{d-1}, |\eta_a| t   -  |x_a| \cos \hat x_{d-1})    ) \hspace{-0.5mm}   \over \sqrt{|x_a|^2+|\eta_a|^2 t^2 -2 |x_a| |\eta_a| t \cos \hat x_{d-1}    }   }   \ d x   \ d t \\
   \nonumber & = \beta^c e^{|\eta|^\gamma \left(1- {1 \over \beta^\gamma}  \right)} \psi_{d-1} \left(\eta \over \beta \right) -\gamma  \int \displaylimits_{1}^{1 \over \beta}  { e^{|\eta|^\gamma \left(1-t^\gamma\right)} \over  t^{c+1} }   \int \displaylimits_{\RR^d}  \psi_{d-1}(x) \cos\hat x_{d-1}   \times \\
   \nonumber & \phantom{=}\phantom{ \beta^c e^{|\eta|^\gamma \left(1- {1 \over \beta^\gamma}  \right)} \psi_{d-1} \left(\eta \over \beta \right) } \hspace{3mm}  \times  \psi_{d-1}(\eta t -x) { |x_a| |\eta_a| t \sin \hat x_{d-1}     \over |x_a- \eta_a t | }   \ d x   \ d t.
\end{align}
Since $|x_a| |\eta_a| t \sin \hat x_{d-1} = \left(\eta_a t \pmb{\times} x_a \right) \cdot e_z =\left( \eta_a t \pmb{\times} (x_a - \eta_a t) \right)_z$, where $\pmb{\times}$ is the cross-product in the three-dimensional subspace spanned by the azimuthal plane and the polar axis, and $( \cdot)_z$ is the projection on the polar axis. Therefore,
\begin{eqnarray}
  \nonumber\cR_\beta[\psi]_{d-1}(\eta)   \hspace{-1.5mm} &   \hspace{-1.5mm}   =  \hspace{-1.5mm}  &  \hspace{-1.5mm}  \beta^c e^{|\eta|^\gamma \left(1- {1 \over \beta^\gamma}  \right)} \psi_{d-1} \left(\eta \over \beta \right) +  \gamma  \int \displaylimits_{1}^{1 \over \beta}  \hspace{-0.8mm}  { e^{|\eta|^\gamma \left(1-t^\gamma\right)}  |\eta| \prod_{k=1}^{d-2} \sin \hat \eta_k \over  t^{c} }   \times \\
 \label{operatorR4} && \hspace{15mm}   \times \int \displaylimits_{\RR^d}  \psi_{d-1}(x) \psi_{d-1}(\eta t  \hspace{-0.2mm} -  \hspace{-0.2mm}x) \cos\hat x_{d-1} \sin {\reallywidehat{(\eta t -x)}}_{d-1}   \ d x    \ d  t.
\end{eqnarray}

We obtain, therefore, that renormalization is defined on the component $\psi_{d-1}$ through $(\ref{operatorR4})$ by setting $\cR_\beta[\psi_{d-1}] :=\cR_\beta[\psi]_{d-1}$.

In what follows we will omit the subscript $d-1$ on the azimuthal component of $\psi$, and we will think of $\psi$ as a function from $\RR^d$ to $\RR$.

\begin{proposition} \label{inv3}
  Let $d \ge 2$.  For any $\sigma$ and $\gamma$, satisfying
  \begin{equation}
    \label{H9}\tag{H9} \sigma \in (\max\{-d,-3\},0) \setminus \left\{-2,-{d \over 2} \right\}, \hspace {3mm}  \gamma-d-1 \ge \sigma,
  \end{equation}
  there exist $k>0$ and $b>0$ such that  the  set of functions
  \begin{equation}
    \label{setM}  \cM:=\left\{\psi \ - \ {\rm measurable}: |\psi(\eta)| < \pmb{\psi}(\eta):=k |y|^{\sigma} e^{-b |\eta|}  \right \}
  \end{equation}
  is invariant under $\cR_\beta$ defined in $(\ref{operatorR4})$ for all $\beta \in (0,1)$.
\end{proposition}
\begin{proof}
We demonstrate the bound from above: 
$$     e^{-|\eta|^\gamma} \pmb{\psi}(\eta) \ge  e^{-|\eta|^\gamma} |\cR_\beta[\psi](\eta)| $$
for any $\psi \in \cM$. The bound from below,
$$     -e^{-|\eta|^\gamma} \pmb{\psi}(\eta) \le  e^{-|\eta|^\gamma} |\cR_\beta[\psi](\eta)|,$$
is obtain by changing the signs in all the inequalities.

It is sufficient to show that 
\begin{equation}
\label{eq0}   0 \ge  \beta^c e^{-{|\eta|^\gamma \over \beta^\gamma}} \pmb{\psi} \left(\eta \over \beta \right) -  e^{-{|\eta|^\gamma }} \pmb{\psi} (\eta ) +  \gamma  \hspace{-1.0mm}  \int \displaylimits_{1}^{1 \over \beta}  \hspace{-0.6mm}  { e^{-|\eta t|^\gamma}  \over  t^{c} }   \hspace{-1.0mm}  \int \displaylimits_{\RR^d}  \hspace{-0.6mm} \pmb{\psi}(x)  \pmb{\psi}(\eta t -  x) \   d x  \  d |\eta| t.
\end{equation}
Our objective will be to show that the right hand side of $(\ref{eq0})$ is smaller than an expression of the form  $\int_{|\eta|}^{|\eta| \over \beta} U(y) \ d y$ with the integrand $U(y)$ strictly negative.

The inequality $(\ref{eq0})$ is implied by
\begin{eqnarray}
  \nonumber   0 &\ge & {|\eta|^{\sigma} \over \beta^{\sigma} } e^{-{|\eta|^\gamma \over \beta^\gamma} - b {|\eta| \over \beta}} - |\eta|^{\sigma} e^{-{|\eta|^\gamma } - b |\eta|} + (\beta^c-1) {|\eta|^{\sigma} \over \beta^{\sigma}} e^{-{|\eta|^\gamma \over \beta^\gamma} - b {|\eta| \over \beta}}    +  \\
  \label{eq1} && \phantom{ {|y|^{\sigma} \over \beta^{\sigma} }  e^{-{|\eta|^\gamma \over \beta^\gamma} - b {|\eta| \over \beta}}} + k \gamma  \hspace{-1.0mm}  \int \displaylimits_{1}^{1 \over \beta}  \hspace{-0.5mm}  { e^{-|\eta t|^\gamma}  \over  t^{c} }   \hspace{-0.5mm}  \int \displaylimits_{\RR^d}    |x|^{\sigma} e^{- b {|x|}}  |\eta t -x|^{\sigma} e^{- b {|\eta t - x|}}     d x  \   d |\eta| t.
\end{eqnarray}
We notice that
\begin{eqnarray}
  \nonumber (\beta^c-1) {|\eta|^{\sigma} \over \beta^{\sigma}} e^{-{|\eta|^\gamma \over \beta^\gamma} - b {|\eta| \over \beta}} & = &\sigma {\beta^c -1 \over 1-\beta^{\sigma}} \  e^{-{|\eta|^\gamma \over \beta^\gamma} - b {|\eta| \over \beta}}  \int_{|\eta|}^{|\eta| \over \beta}  y^{\sigma-1} \ d y.
\end{eqnarray}
Therefore $(\ref{eq1})$ is implied by 
\begin{align}
  \nonumber   0 &\ge  \int_{|\eta|}^{|\eta| \over \beta} \left\{ y^{\sigma} (-\gamma y^{\gamma-1} - b ) +\sigma y^{\sigma-1} \right\}  e^{-{y^\gamma} - b {y}}   + \sigma {\beta^c -1 \over 1-\beta^{\sigma}}  y^{\sigma-1}  e^{-{|\eta|^\gamma \over \beta^\gamma} - b {|\eta| \over \beta}}   \  d y    +  \\
  \label{eq2} & \phantom{ \int_{|\eta|}^{|\eta| \over \beta} \hspace{3mm} y^{\sigma} (-\gamma y^{\gamma-1}  } \hspace{6.0mm}+ k \gamma  \hspace{-0.5mm}  \int \displaylimits_{1}^{1 \over \beta}  \hspace{-0.4mm}  { e^{-|\eta t|^\gamma}  \over  t^{c} }   \hspace{-0.5mm}  \int \displaylimits_{\RR^d}  \hspace{-0.5mm}  |x|^{\sigma} e^{- b {|x|}}  |\eta t -x|^{\sigma} e^{- b {|\eta t - x|}}     d x  \   d |\eta| t.
\end{align}
Consider the convolution
$$\int \displaylimits_{\RR^d}  \hspace{-0.5mm}  |x|^{\sigma} e^{- b {|x|}}  |y -x|^{\sigma} e^{- b {|y - x|}}d x.$$
According to Lemma $\ref{radial}$, this convolutions is a radial function of $|y|$ only, denoted $I(|y|)$:
\begin{eqnarray}
  \nonumber I(|y|) \hspace{-2mm}  &\le& \hspace{-2mm}  C  \hspace{-1.5mm}  \int \displaylimits_{0}^\infty  \hspace{-1.2mm}  |x|^{\sigma} e^{- b {|x|}}  \hspace{-2.3mm}  \int \limits_{-1}^1 \hspace{-0.5mm}  (|y|^2 \hspace{-0.5mm} + \hspace{-0.5mm} |x|^2 \hspace{-0.5mm} - \hspace{-0.5mm} 2 |y| |x| t)^{\sigma \over 2} e^{- b (|y|^2+|x|^2-2 |y| |x| t)^{1 \over 2}} (1\hspace{-0.5mm} \shortminus \hspace{-0.5mm}  t^2)^{d-3 \over 2} d t  |x|^{d-1} d |x| \\
  \nonumber \hspace{-2mm}   &\le& \hspace{-2.0mm} C \int \displaylimits_{0}^{\infty}  \hspace{-0.5mm}  |x|^{\sigma} e^{- b {|x|}}  {1 \over |x| |y|} \int \displaylimits_{(|y|-|x|)^2}^{(|y|+|x|)^2}  z^{\sigma \over 2} e^{- b z^{1 \over 2}} d z  \ |x|^{d-1} d |x| \\
  \nonumber \hspace{-2mm}   &\le& \hspace{-2.0mm} {C \over |\sigma +2|} \int \displaylimits_{0}^{\infty}  \hspace{-0.5mm}  r^{\sigma+d-2} e^{- b r-b||y|-r|}  {1 \over |y|} \left| ||y|-r|^{\sigma+2} -(|y|+r)^{\sigma+2} \right|  \  d r \\
    \label{eq00} \hspace{-2mm}   &=& \hspace{-2.0mm} I_1(|y|)+I_2(|y|),
\end{eqnarray}
where $I_1$ an $I_2$ are the integrals over ranges $(0,|y|)$ and $(|y|,\infty)$, respectively. In the integral above we have showed the singularity $(\sigma+2)^{-1}$ explicitly, rather than including it in a constant. We will continue doing this below.

\vspace{2mm}

\noindent \underline{\it $1)$ Case $\sigma <-2$}.  We start by evaluating $I_1$:
\begin{eqnarray}
  \nonumber I_1(|y|) \hspace{-2mm}   &\le& \hspace{-2.0mm} {C \over |\sigma +2|} \int \displaylimits_{0}^{|y|}  \hspace{-0.5mm}  r^{\sigma+d-2} e^{- b r-b||y|-r|}  {1 \over |y|} \left( ||y|-r|^{\sigma+2} -(|y|+r)^{\sigma+2} \right)  \  d r \\
  \nonumber \hspace{-2mm}   &\le& \hspace{-2.0mm} {C \over |\sigma +2|} |y|^{2 \sigma+d} e^{-b |y|} \int \displaylimits_{0}^{1}  \hspace{-0.5mm}  z^{\sigma+d-2}  \left( (1-z)^{\sigma+2} -(1+z)^{\sigma+2} \right)  \  d z.
\end{eqnarray}
To evaluate this integral, we will use that for $|z|<1$,
\begin{eqnarray}
  \nonumber   (1-z)^{n+1}-(1+z)^{n+1}&=& \sum_{k=0}^\infty \left(n+1 \atop k  \right) \left( (-z)^k-z^k \right)=-2z  \sum_{k=0}^\infty \left(n+1 \atop 2 k+1  \right) z^{2 k} \\
  \label{2F1} &=&-2(n+1)z \phantom{}_2F_1\left({1\over 2} -{n \over 2},-{n \over 2}, {3\over 2},z^2\right).
\end{eqnarray}
\begin{eqnarray}
  \nonumber I_1(|y|) \hspace{-2mm}   &\le& \hspace{-2.0mm} C |y|^{2 \sigma+d} e^{-b |y|} \int \displaylimits_{0}^{1}  \hspace{-0.5mm}  z^{\sigma+d-1} \phantom{}_2F_1\left({1 \over 2} -{\sigma+1 \over 2},-{\sigma+1 \over 2}, {3 \over 2}, z^2  \right)  d z.
\end{eqnarray}
Assume that $\sigma+2<0$. Then we can use the expansion of the Gauss hypergeometric function $\phantom{ }_2F_1 \left(a, b, c,x  \right)$ in the case $c-a-b<0$ (in our case $c-a-b=\sigma+2<0$ by $(\ref{H9})$):
\begin{equation}
\label{2F1neg} \phantom{ }_2F_1 \left(-{\sigma \over 2}, -{\sigma +1\over 2}, {3 \over 2},z^2  \right) \le C { \Gamma(-\sigma-2) \Gamma\left({3 \over 2} \right) \over  \Gamma\left(-{\sigma \over 2} \right)  \Gamma\left(-{\sigma+1 \over 2} \right)} \left(1-z^2 \right)^{\sigma+2}.
\end{equation}
For $\sigma<-2$ all $\Gamma$-functions in the constant are bounded, except $ \Gamma(-\sigma-2)$ which is proportional to $(\sigma+2)^{-1}$ for $\sigma$ close to $-2$, therefore,
\begin{equation}
\nonumber I_1(|y|)   \le  {C \over |2+\sigma|} |y|^{2 \sigma+d} e^{-b |y|} \int \displaylimits_{0}^{1}    z^{\sigma+d-1} \left(1-z^2  \right)^{\sigma+2}  d z.
\end{equation}
The last integral converges under the condition
\begin{equation}
\label {cond-3} \sigma>\max\{-d,-3\},
\end{equation}
and
\begin{equation}
\label{condd12} I_1(|y|)   \le  {C \over |\sigma+2||\sigma+3||\sigma+d|} |y|^{2 \sigma+d} e^{-b |y|}.
\end{equation}
We now turn to $I_2$:
\begin{align}
\nonumber I_2(|y|)     &\le  {C \over |\sigma+2|}  \int \displaylimits_{|y|}^{\infty}  \hspace{-0.5mm}  r^{\sigma+d-2} e^{- b r-b||y|-r|}  {1 \over |y|} \left( ||y|-r|^{\sigma+2} -(|y|+r)^{\sigma+2} \right)  \  d r \\
\nonumber   &\le {C \over |\sigma+2| |y|}  \int \displaylimits_{|y|}^{\infty}  \hspace{-0.5mm}  r^{\sigma+d-2} e^{- 2 b r+b |y|} r^{\sigma+2} \left( \left(1-{|y| \over r} \right)^{\sigma+2} -\left(1+{|y| \over  r} \right)^{\sigma+2} \right)  \  d r \\
\nonumber   &\le C   \int \displaylimits_{|y|}^{\infty}  \hspace{-0.5mm}  r^{2 \sigma+d-1} e^{- 2 b r+b |y|} \phantom{}_2F_1\left({1 \over 2} -{\sigma+1 \over 2},-{\sigma+1 \over 2}, {3 \over 2}, {|y|^2 \over r^2}   \right)  \  d r \\
\nonumber   &\le {C \over |\sigma+2|} \int \displaylimits_{|y|}^{\infty}  \hspace{-0.5mm}  r^{2 \sigma+d-1} e^{- 2 b r+b |y|}\left(1 -{|y|^2 \over r^2}   \right)^{\sigma+2}  \  d r \\
\nonumber   &\le  {C \over |\sigma+2|}  e^{b|y|} |y|^{2 \sigma+d} \int \displaylimits_{1}^{\infty}  \hspace{-0.5mm}  z^{\sigma+d-3} (z-1)^{\sigma+2} e^{- 2 b|y| z}  \  d z \\
\nonumber   &\le  {C \over |\sigma+2|}  e^{b|y|} |y|^{2 \sigma+d} \left( e^{- 2 b|y|}  \int \displaylimits_{1}^{2}  \hspace{-0.5mm}  z^{\sigma+d-3} (z-1)^{\sigma+2}   \  d z +\int \displaylimits_{2}^{\infty}  \hspace{-0.5mm}  z^{2 \sigma+d-1} e^{- 2 b|y| z}  \  d z     \right).
\end{align}
The first integral in the parenthesis converges for $\sigma > -3$, therefore, under this condition
\begin{equation}
  \nonumber I_2(|y|) \le  {C \over |\sigma+3| |\sigma+2|}  e^{b|y|} |y|^{2 \sigma+d} \left( e^{- 2 b|y|}  + |y|^{-d-2 \sigma} \Gamma(d+2 \sigma, 4 b |y|)   \right).
\end{equation}
We, have, therefore, for small $|y|$:
\begin{align}
  \nonumber I_2(|y|) & \le  {C \over |\sigma+3| |\sigma+2|}  e^{b|y|} |y|^{2 \sigma+d} \left( e^{- 2 b|y|}  + |y|^{-d-2 \sigma} \left(\Gamma(d+2 \sigma)- C {|y|^{d+2 \sigma}  \over |d+2 \sigma | }  \right)   \right) \\
  \nonumber  & \le  {C \over |\sigma+3|  |\sigma+2| |2 \sigma +d|}  e^{- b|y|} (1+|y|^{2 \sigma+d}),
\end{align}
and for large $|y|$:
\begin{align}
  \nonumber I_2(|y|) & \le  {C \over |\sigma+3| |\sigma+2|}  e^{b|y|} |y|^{2 \sigma+d} \left( e^{- 2 b|y|}  + |y|^{-1} e^{-4 b|y|} \right)\\
  \nonumber  & \le  {C \over |\sigma+3| |\sigma+2|} |y|^{2 \sigma+d}   e^{- b|y|}.
\end{align}
We summarize:
\begin{equation}
  \label{I} I(|y|) \le  {C \over |\sigma+3| |\sigma+d| |\sigma+2| |2 \sigma+d|} e^{-b |y|}  \left\{  |y|^{ \min\{0,2 \sigma +d\}}, \quad |y| < (2 b)^{-1}, \atop   |y|^{2 \sigma +d}, \hspace{14.6mm} |y|  \ge  (2 b)^{-1}. \right. . 
\end{equation}
Finally, we obtain that $(\ref{eq2})$ is implied by 
\begin{align}
  \nonumber    0   & \ge   \hspace{-0.8mm}  \int \displaylimits_{|\eta|}^{|\eta| \over \beta}  \left\{ y^{\sigma} (-\gamma y^{\gamma-1} \hspace{-1mm} - b ) \hspace{-0.5mm} +  {\hspace{-5mm} C k \over |\sigma+2| |\sigma+3||\sigma+d| |2 \sigma+d|} y^{l(y)} \right\} e^{-{y^\gamma} - b {y}}  + \\
  \label{eq22} & \hspace{31.5mm} +   \sigma {1 -\beta^{c-\sigma} \over 1 -   \beta^{-\sigma}} y^{\sigma-1} e^{-{|\eta|^\gamma \over \beta^\gamma}-b {|\eta| \over \beta} }  \ d y,
\end{align}
where $l(y)=\min\{0,2 \sigma +d\}$ for small $y$ and $l(y)=2 \sigma +d$ for large $y$.

The first of the two sufficient conditions for the integrand in $(\ref{eq22})$ to be non-positive for small $y$ is the non-positivity of the last term, which means that
$$c \ge \sigma,$$
the condition included in $(\ref{H9})$. The second of the two sufficient conditions is the dominance of the term  $ -b y^{\sigma}$ over the last term in the first line of $(\ref{eq22})$ for small $y$, or 
$$\sigma \le \min\{0,2 \sigma +d\}$$
which implies that
$$-d \le \sigma \le 0,$$
also included in $(\ref{H9})$.

A sufficient condition for the integrand in $(\ref{eq22})$ to be negative for large $y$  is the dominance of the term $-\gamma y^{\sigma+\gamma-1}$ over the last term in the first line of $(\ref{eq22})$  which is achieved if 
$$\sigma+\gamma-1 \ge 2 \sigma +d\,$$
which is included in $(\ref{H9})$.

Finally, for any $\sigma$ and $\gamma$ as in $(\ref{H9})$, the negativity of the integrand for all $y$ can be insured by a choice of constants $k>0$ and $b>0$.

\vspace{2mm}

\noindent \underline{\it $2)$  Case $\sigma>-2$}. We have in this case
\begin{eqnarray}
  \nonumber I(|y|) \hspace{-2mm} &\le& \hspace{-2.0mm} {C \over |\sigma +2|} \int \displaylimits_{0}^{\infty}  \hspace{-0.5mm}  r^{\sigma+d-2} e^{- b r-b||y|-r|}  {1 \over |y|} \left( ||y|+r|^{\sigma+2} -||y|-r|^{\sigma+2} \right)  \  d r \\
  \label{eq00} \hspace{-2mm}   &=& \hspace{-2.0mm} I_1(|y|)+I_2(|y|),
\end{eqnarray}

The calculation of $I_1$ is very similar to the case $\sigma+2<0$, up to  $(\ref{2F1neg})$. We substitute $(\ref{2F1neg})$ by an upper bound for the case $c-a-b>0$:
\begin{equation}
\label{2F1pos} \phantom{ }_2F_1 \left(-{\sigma \over 2}, -{\sigma +1\over 2}, {3 \over 2},z^2  \right) \le { \Gamma(\sigma+2) \Gamma\left({3 \over 2} \right) \over  \Gamma\left({\sigma +3 \over 2} \right)  \Gamma\left({\sigma+4 \over 2} \right)} \le {C \over |\sigma +2|}.
\end{equation}
Then
\begin{equation}
\nonumber I_1(|y|)   \le  {C \over |\sigma+2| } |y|^{2 \sigma+d} e^{-b |y|} \int \displaylimits_{0}^{1}    z^{\sigma+d-1}   d z.
\end{equation}
The last integral converges under the condition
\begin{equation}
\label {cond-d} \sigma>-d,
\end{equation}
and
\begin{equation}
\label{condd13} I_1(|y|)   \le  {C \over |2+\sigma| |\sigma+d|} |y|^{2 \sigma+d} e^{-b |y|}.
\end{equation}

The calculation of the integral $I_2$ is similar to the case $\sigma<-2$ with the exception that the hypergeometric function appearing under the integral is now bounded as
$$ \phantom{}_2F_1\left({1 \over 2} -{\sigma+1 \over 2},-{\sigma+1 \over 2}, {3 \over 2}, {|y|^2 \over r^2}   \right) \le {C \over |\sigma+2|}.$$
Therefore,
\begin{eqnarray}
  \nonumber I_2(|y|)
  \nonumber  \hspace{-2mm}   &\le& \hspace{-2.0mm}  {C \over |\sigma+2|} \int \displaylimits_{|y|}^{\infty}  \hspace{-0.5mm}  r^{2 \sigma+d-1} e^{- 2 b r+b |y|}  \  d r \\
  \nonumber  \hspace{-2mm}   &\le& \hspace{-2.0mm}   {C \over |\sigma+2|} e^{b |y|}   \Gamma \left(2 \sigma+d, -2 b |y|\right) \\
  \label{I21}  \hspace{-2mm}   &\le& \hspace{-2.0mm}   {C \over |\sigma+2|} \  e^{-b |y|}  \left\{ \Gamma(2 \sigma+d) - C { |y|^{2 \sigma +d }\over |2 \sigma+d|} , \hspace{4mm} |y|<{(2 b)^{-1}}, \atop |y|^{2 \sigma +d-1}, \hspace{24.8mm} |y|\ge{(2 b)^{-1}}.  \right.  
\end{eqnarray}
We combine bounds  $(\ref{I21})$ with  $(\ref{condd13})$ to get 
\begin{equation}
  \label{II} I(|y|) \le   {C \over |\sigma +2||\sigma+d| |2 \sigma+d|}  e^{-b|y|} \left\{  |y|^{ \min\{0,2 \sigma +d\}}, \quad |y| < (2 b)^{-1}, \atop   |y|^{2 \sigma +d}, \hspace{14.6mm} |y|  \ge  (2 b)^{-1}. \right. 
\end{equation}
Notice, that the powers of $y$ in  bound $(\ref{II})$ coincide in this case with the bound for $\sigma <-2$. Therefore, we obtain exactly the same conditions for the non-positivity of the integrand as in the case  $\sigma<-2$.
\end{proof}

The set $\cM$ clearly contains zero. It seems prohibitively complicated to introduce a renormalization invariant bound on the size of the functions in $\cM$ from below. We, therefore, will take a different approach and consider the integral
$$\cI(\psi):=\int \displaylimits_{\RR^d}   e^{-|\eta|^\gamma} \psi(\eta) |\eta|^\alpha  \ d \eta,$$
which may not be bounded for all functions in $\cM$ due to their possible singularity at $0$, but $\cM$ clearly contains a subset of functions with a bounded $\cI$.

Set
\begin{equation}
  \label{w} w(\eta)=|\eta|^\alpha.
\end{equation}
Notice that $\cM \cap L^1_w(\RR^d,\RR)$ is non-empty, as, e.g. any function $f(\eta) e^{-b |\eta|}$ with a measurable $f$,
\begin{equation}
\label{fM} |f(\eta)| < \min\{ |\eta|^\sigma, |\eta|^{-d-\alpha+\epsilon}\}
\end{equation}
for any $\epsilon>0$, belongs to this intersection.

Given $\mu>0$, consider the following set
\begin{equation}
\label{setMmu}  \cM_\mu^\alpha:=\left\{ \psi \in \cM \cap L^1_w(\RR^d,\RR):  \cI(\psi) \ge \mu \right\}.
\end{equation}
This set is convex and does not contain the trivial function.

\begin{proposition}
  For any $b>0$, $\epsilon>0$ and  $\alpha$ satisfying
\begin{eqnarray}
  \label{alpha_cond1}  \epsilon-d-\alpha &\in&   (\max\{-3,-d\},\infty) \setminus \ZZ/2, \\
  \label{alpha_cond2}  -2 d -1 < \alpha  <\epsilon-{d \over 2} \quad &{\rm or}& \quad  \epsilon+{1-d \over 2} \le  \alpha<1+2 \epsilon,\\
  \label{alpha_cond3}  \alpha&<&1-\gamma,
\end{eqnarray}
  there is a constant $\mu$,
  $$\mu \asymp k^2,$$
  where the commensurability constant is independent of $b$, such that  the set $\cM_\mu^\alpha$ is non-empty and invariant under $\cR_\beta$ for all $(0,1)$.
\end{proposition}
\begin{proof}
  We consider the integral $\cI(\cR_\beta[\psi])$ for $\psi \in \cM_\mu^\alpha$. Notice that the the function
\begin{equation}
\pmb{\phi_{\alpha,\epsilon}}(\eta)=  k|\eta|^{-d-\alpha+\epsilon} e^{-b|\eta|}
\end{equation}
  is a conservative upper bound on the absolute values of functions in $\cM_\mu^\alpha$ (in fact it is an overestimate, as the bound $k f(\eta) e^{-b |\eta|}$ with $f$ as in $(\ref{fM})$ would do, but we keep a constant power of $|\eta|$ for simplicity of calculations). 
  \begin{eqnarray}
    \nonumber    \cI(\cR_\beta[\psi])  \hspace{-1mm} &  \hspace{-1mm}  \ge  \hspace{-1mm}  &   \hspace{-1mm}  \beta^c  \hspace{-1mm}  \int \displaylimits_{\RR^d}   \hspace{-1mm}  e^{-{|\eta|^\gamma \over \beta^\gamma}} \psi \left({\eta \over \beta} \right) |\eta|^\alpha  \ d \eta -  \gamma   \hspace{-1mm}  \int \displaylimits_{\RR^d}   \hspace{-1mm}   |\eta|^\alpha  \hspace{-0.5mm}  \int \displaylimits_{1}^{1 \over \beta}  \hspace{-0.8mm}  { e^{-|\eta t|^\gamma} \hspace{-0.5mm}  \over  t^{c} }  \left( \pmb{\phi_{\alpha,\epsilon}} \star \pmb{\phi_{\alpha,\epsilon}} \right)(\eta t)     d |\eta| t \ d \eta.
\end{eqnarray}
  Computation of the convolution integral is identical to that in Proposition $\ref{inv3}$ with $\sigma$ substituted by $-d-\alpha+\epsilon$. Similarly to the hypothesis of Proposition $(\ref{inv3}$ for $\sigma$ which guarantees existence of the convolution integral, conditions $(\ref{alpha_cond1})$ must be satisfied (but not the the analogues of the conditions $c \ge \sigma \ge -d$ and $\sigma<0$ which were the conditions on the invariance of the set $\cM$).
  \begin{equation}
    \nonumber  \cI(\cR_\beta[\psi]) \ge   \beta^{c+\alpha+d}   \int \displaylimits_{\RR^d}   \hspace{-1mm}  e^{-|\eta|^\gamma} \psi(\eta) |\eta|^\alpha  \ d \eta - k^2 \ C \hspace{-1mm}  \int \displaylimits_{\RR^d}   \hspace{-1mm}    |\eta|^\alpha  \hspace{-0.5mm}  \int \displaylimits_{|\eta|}^{|\eta|  \over \beta} z^{m(z)} e^{-z^\gamma -b z }  d z  \ d \eta,
  \end{equation}
  where either $\epsilon-d-\alpha<-2$ and
  \begin{equation}
\nonumber m(\eta)=2(\epsilon-d-\alpha) +d,
  \end{equation}
  or $\epsilon-d-\alpha>-2$ and 
  \begin{equation}
\nonumber m(\eta)=\left\{ 2(\epsilon-d-\alpha)+d, \hspace{20.0mm} |\eta|> (2 b)^{-1}, \atop  \min\{0,2 (\epsilon-d-\alpha)+d\}, \quad |\eta| \le (2 b)^{-1}   \right.,
  \end{equation}

  That is,  when
\begin{equation}
\label{alpha1}  \epsilon-{d \over 2} \le \alpha <\epsilon -d +\min\{ d, 3\},
\end{equation}
    the integral becomes
  \begin{equation}
    \nonumber  \cI(\cR_\beta[\psi]) \ge  \beta^{c+\alpha+d} \mu  - k^2 \ C \hspace{-0.5mm} \left(1 \hspace{-0.5mm} + \hspace{-0.5mm}  b^{\epsilon-\alpha+{1-d \over 2} }  \right) \hspace{-0.5mm}  \left|{1 \over \beta^{2 \epsilon -d -2 \alpha+1}}-1  \right|    \int \displaylimits_{\RR^d}   \hspace{-1mm}  |\eta|^{2 \epsilon-d-\alpha+1}  e^{-|\eta|^\gamma -b |\eta| } \  d \eta.
  \end{equation}

  The last integral exists if
  \begin{equation}
    \label{alpha_bound1} 1-d-\alpha +2 \epsilon> -d \implies 1+2 \epsilon > \alpha.
  \end{equation}
  Therefore, under  $(\ref{alpha1})$ and $(\ref{alpha_bound1})$ 
  we have that
  \begin{equation}
    \label{bound11}  \cI(\cR_\beta[\psi])   \ge   \beta^{c+\alpha+d}  \mu  - k^2 \ C  \left(1-\beta  \right)  \hspace{-0.5mm} \left(1 \hspace{-0.5mm} + \hspace{-0.5mm}  b^{\epsilon-\alpha+{1-d \over 2} }  \right) \hspace{-0.5mm}.
  \end{equation}
  We will also impose the condition
  \begin{equation}
    \label{alpha22} \alpha > \epsilon +{1 -d  \over 2}. 
  \end{equation}
Under this condition the factor  $\left(1+b^{\epsilon-\alpha+{1-d \over 2} }  \right)$ is bounded for large $b$.

A similar computation (with the nonlinear term added rather than subtracted) shows that
  \begin{equation}
\|\cR_\beta[\psi] \|_w \le  \beta^{c+\alpha+d}  \mu  + k^2 \ C  \left(1-\beta  \right).
  \end{equation}
  with $w$ as in  $(\ref{w})$.

 Next, consider the case
\begin{equation}
  \label{alpha3}  \alpha < \epsilon -{d \over 2}.
  \end{equation}
  The lower bound on the integral $\cI(\cR_\beta[\psi])$ in this case is:
  \begin{eqnarray}
    \nonumber  \cI(\cR_\beta[\psi]) \ge  \beta^{c+\alpha+d} \mu  \hspace{-1.2mm}   &   \hspace{-1.2mm} -  \hspace{-1.2mm} &  \hspace{-1.2mm} k^2 \ C  \left|{1 \over \beta^{d+1}}-1  \right|   \int \displaylimits_{|\eta| < (2 b)^{-1}}   \hspace{-1mm}  |\eta|^{d+1+\alpha}  e^{-|\eta|^\gamma -b |\eta| } \  d \eta - \\
    \nonumber \hspace{-1.2mm}   &   \hspace{-1.2mm} -  \hspace{-1.2mm} &k^2 \ C  \left|{1 \over \beta^{2 \epsilon -d -2 \alpha+1}}-1  \right|    \int \displaylimits_{|\eta| \ge (2 b)^{-1}}   \hspace{-1mm}  |\eta|^{2 \epsilon-d-\alpha+1}  e^{-|\eta|^\gamma -b |\eta| } \  d \eta.
  \end{eqnarray}

  The second integral always converges, the first one does so under the condition 
  \begin{equation}
    \label{alpha4} d+1+\alpha> -d \implies \alpha>-2 d -1,
  \end{equation}
  and together with $(\ref{alpha3})$ this gives another range for $\alpha$:
 \begin{equation}
    \label{alpha5}  -2 d -1 < \alpha < \epsilon -{d \over 2}.
  \end{equation}
The lower bound becomes
 \begin{eqnarray}
   \nonumber  \cI(\cR_\beta[\psi])  \ge  \beta^{c+\alpha+d} \mu  \hspace{-1.2mm}   &   \hspace{-1.2mm} -  \hspace{-1.2mm} &  \hspace{-1.2mm} k^2 C  (1 \hspace{-0.5mm} - \hspace{-0.5mm} \beta)   \hspace{-0.5mm}  \left( \hspace{-0.8mm} \gamma \hspace{-0.5mm} \left({2 d \hspace{-0.5mm} + \hspace{-0.5mm} \alpha \hspace{-0.5mm}  + \hspace{-0.5mm} 1 \over \gamma},(2 b)^{-\gamma}  \right) \hspace{-0.5mm} +  \hspace{-0.5mm}  \Gamma \hspace{-0.5mm} \left( \hspace{-0.5mm} {2 \epsilon \hspace{-0.5mm}  - \hspace{-0.5mm}  \alpha \hspace{-0.5mm} +  \hspace{-0.5mm} 1\over \gamma},(2 b)^{-\gamma}  \hspace{-0.8mm} \right) \hspace{-0.5mm} \right).
  \end{eqnarray}
   For large $b$ this expression can be bounded as follows:
    \begin{equation}
      \nonumber  \cI(\cR_\beta[\psi])  \ge  \beta^{c+\alpha+d} \mu  -  k^2 \ C  (1-\beta)    \left( b^{-2 d -\alpha -1}  +   \Gamma \left({2 \epsilon -\alpha +1\over \gamma} \right) \right).
  \end{equation}
A sufficient condition for this bound o be independent of $b$ as $b$ is taken large, is 
\begin{equation}
  \nonumber 2 d + \alpha +1 >0,
\end{equation}
  which coincides with the lower bound in $(\ref{alpha5})$.

  We obtain that under $(\ref{alpha5})$,
  \begin{equation}
    \label{bound11}  \cI(\cR_\beta[\psi])   \ge   \beta^{c+\alpha+d}  \mu  - k^2 \ C  \left(1-\beta  \right).
  \end{equation}
  A similar computation (with the nonlinear term added rather than subtracted) shows that
  \begin{equation}
\|\cR_\beta[\psi] \|_w \le  \beta^{c+\alpha+d}  \mu  + k^2 \ C  \left(1-\beta  \right).
  \end{equation}
  with $w$ as in  $(\ref{w})$.

  The set $\cM_\mu^\alpha$ is renormalization invariant if $(\ref{bound11})$ is larger than $\mu$, that is if 
  \begin{equation}
    \label{alpha_bound2} c+\alpha+d<0 \implies \alpha < 1-\gamma,
  \end{equation}
 which is stronger than $(\ref{alpha_bound1})$,  and
  \begin{equation}
  \label{mu_bound} \mu >   {k^2 \ C  (1-\beta)  \over \beta^{c+\alpha+d} -1} \ge C k^2.
\end{equation} 

  Consider the function $\psi_{a,f}(\eta)=a k f(|\eta|) e^{-b |\eta|}$ with some measurable $f$. This function is in $\cM$ if  $f(|\eta|) \le |\eta|^\sigma$ and $0<a<1$, and, additionally, $\cI(\psi_{a,f})$ exists if $f$ has a mild singularity at $0$:, i.e. $ \psi_{a,f} \in L^1_{w}(\RR^d,\RR)$, with $w(\eta)=|\eta|^\alpha$, as before. Then
  \begin{eqnarray}
    \nonumber    \cI(\psi_{a,f})  \hspace{-1mm} &  \hspace{-1mm}  \ge  \hspace{-1mm}  &   \hspace{-1mm} a k   \hspace{-0.5mm}  \int \displaylimits_{\RR^d}   e^{-|\eta|^\gamma-b |\eta|} f(|\eta|) |\eta|^\alpha \ d \eta > C a k, 
  \end{eqnarray}
  where the constant $C$ depends on $f$. Clearly, $I(\psi_{a,f})$ can be made larger than $\mu$ in $\ref{mu_bound})$ by a choice of $k$.
\end{proof}

Notice that conditions $(\ref{alpha_cond1})-(\ref{alpha_cond3})$ imply for $d \ge 2$ that $\alpha$ belongs to the following set:
\begin{align}
 \label{H10}  \tag{H10} \alpha &\in S_1 \cap S_2, \quad {\rm where}  \\
  \nonumber S_1 &=\left( \left( -\infty, \min\{1-\gamma,  \epsilon-d+\min\{ 3,d\} \} \right)  \hspace{-0.5mm} \setminus (\ZZ/2 - d +  \epsilon)  \right), \\
\nonumber S_2 &= \left(-2 d -1, \epsilon -{d \over 2} \right)  \cup \left[\epsilon+{1-d \over 2},1+2 \epsilon  \right).
\end{align}

\subsection{Existence of an invariant bounded convex precompact set}  \label{nontriv_precompactness}
Let
\begin{equation}
\label{weight} v_\nu(\eta)=|\eta|^\nu e^{{b \over 2} |\eta|}.
\end{equation}

Set  $\cM$ defined in $(\ref{setM})$ is a subset of $L^p_{v_\nu}(\RR^d,\RR)$ as long as
\begin{equation}
  \label{MinLp} p (\sigma+\nu)>-d.
\end{equation}
$\cM$ is clearly equitight in $L^p_{v_\nu}$ since all its members are exponentially bounded.

We will now turn to the issue of equicontinuity. 

We start by giving an example of a function in $\cM_\mu^\alpha$, such that $\psi-\tau_\delta \psi$ is bounded in norm by $C |\delta|$.

\begin{lemma}\label{Nnonempty}
  Fix $\delta_0>0$ and set
  $$\psi_{a,f}(\eta)=a k f(|\eta|) e^{-b |\eta|},$$
  where $f$ is any Lipschitz-continuous function defined over the non-negative real line such that  $\ 0<f(x) < \min\{x^\sigma,x^\zeta\}$ for some fixed $\zeta>\max\{0,-d-\alpha \}$ and  $\sigma \in (\max\{-3,-d\},0) \setminus \ZZ/2$ satisfying $(\ref{MinLp})$. Then there exists $K>0$, such that 
  $$\|\psi_{a,f}-\tau_\delta \psi_{a,f}  \|_{v_\nu,p} \le K |\delta|$$
  whenever $|\delta| < \delta_0$.
\end{lemma}
\begin{proof}
  Denote $e(\eta)=e^{-b |\eta|}$ and write
  $$\|\psi -\tau_\delta \psi \|_{v_\nu,p} \le \|e (f- \tau_\delta f)\|_{v_\nu,p}+\|(e-\tau_\delta e) \tau_\delta  f  \|_{v_\nu,p}.$$
  
Since   $f$ is Lipschitz continuous,
$$\|e (f- \tau_\delta f)\|_{v_\nu,p} \le L k  |\delta| \left(\int_{\RR^d}  e^{-b p |\eta|} v_\nu(\eta)^p d \eta \right)^{1 \over p} \le C k |\delta|.$$

For the second norm we have
\begin{eqnarray}
  \nonumber \|(e-\tau_\delta e) \tau_\delta  f  \|_{v_\nu,p} &\le& a k \left( \int_{\RR^d} |e^{-b |\eta|} - e^{-b |\eta-\delta|}  |^p   f(|\eta-\delta|)^p v_\nu(\eta)^p d \eta \right)^{1 \over p} \\
  \nonumber &\le & a  k  \left( \int_{\RR^d} e^{-p {b \over 2} |\eta|} \left|1- e^{-b |\eta-\delta|+b |\eta|} \right|^p f(\eta-\delta)^p  |\eta|^{\nu p} d \eta \right)^{1 \over p}\\
  \label{nonempt} &\le & C k  |\delta|  \left(  \int_{\RR^d}   e^{-p {b \over 2} |\eta|} |\eta-\delta|^{\zeta p}  |\eta|^{\nu p} d \eta \right)^{1 \over p}.
\end{eqnarray}
Since $\zeta>0>\sigma$, we have that
$$( \sigma + \nu) p>-d, \quad (\zeta + \nu) p>-d, \quad \nu  p > -d, \quad \zeta p > -d.$$
Therefore the integral in $(\ref{nonempt})$ is convergent and the claim follows.
\end{proof}

We will require another version of Young's Convolution inequality for weighted $L^p$-spaces for non-radial functions (see Proposition $2.2$ and Example $3.2$ in \cite{BiSwa}).

\begin{proposition} \label{Young2}
  Suppose $1 < t \le \min\{p,q\}$, $r \ge 1$ and
  $${1 \over p} +{1 \over q}={1 \over r} +{1 \over t}.$$
  Let
  $$\nu, \mu < {d \over t'}, \quad \nu+\mu > {d \over t'},$$
  where $t'$ is the conjugate of $t$.  Then for any two $f \in L^p_{v_\nu}$ and $g \in  L^q_{v_\mu}$, the convolution $f \star g \in  L^r_{v_\kappa}$,
  $$\kappa=\mu+\nu-{d \over t'},$$
  with
  $$ \|f \star g \|_{v_{\kappa},r} \le C \| f \|_{v_\nu,p}  \| g \|_{v_\mu,q}.$$
\end{proposition}

 We can now prove the following result.

 \begin{proposition} \underline{(\it Equicontinuity.)} \label{non_triv_equicont}
   Let
   \begin{equation}
     \label{H11} \tag{H11}
     \begin{cases}
          & \hspace{-3mm}  c \ge \sigma \ge  -d, \hspace{8mm}  \sigma \in \left(-{p-1 \over p},0 \right) \setminus \ZZ/2, \hspace{6.5mm}  \gamma > d-2,    \hspace{14.2mm} p>2, \\
       &  \hspace{-3mm} c \hspace{-0.5mm}   +  \hspace{-0.5mm}  \nu  \hspace{-0.5mm} +  \hspace{-0.5mm}  {d \over p}>\theta, \hspace{4mm}  (2 \sigma  \hspace{-0.5mm}  +  \hspace{-0.5mm} d  \hspace{-0.5mm}  +  \hspace{-0.5mm} \nu   \hspace{-0.5mm} -  \hspace{-0.5mm} 1)p > -d, \hspace{5mm} d {p  \hspace{-0.5mm} -  \hspace{-0.5mm} 2 \over 2 p} < \nu < d {p  \hspace{-0.5mm} -  \hspace{-0.5mm}  2 \over p}, \hspace{4.7mm} \kappa < d, \\
       &  \hspace{-3mm} (\sigma + \nu) p > -d, \hspace{3.0mm}   \theta -c - \kappa+1>0, \hspace{10.5mm}  -d <(\nu - \kappa) p, \hspace{6.7mm} \gamma > d+\sigma .
     \end{cases}
   \end{equation}
  Fix $\delta_0>0$. For any $\beta_0 \in (0,1)$ there exists, $A>0$, $K>0$, $1>\theta>0$, $k>0$, $\tau>0$ and $b>0$ such that the subset $\cN$ of $\cM$ satisfying
$$| \psi(\eta) -\tau_\delta \psi(\eta) |  \le \omega_\psi(\eta)  |\delta|^\theta$$
   for all vectors $\delta \in \RR^d$, $|\delta|<\delta_0$, and $\omega_\psi \in L_{v_\nu}^p(\RR^d,\RR_+)$  such that $\| \omega_\psi \|_{v_\nu,p} \le K$, and
   \begin{equation}
\label{omegabound}    \omega_\psi(\eta) \le  A |\eta|^\tau e^{-b |\eta|}
   \end{equation}
whenever $|\eta|>1$,  is non-empty and invariant under the operator $\cR_\beta$ for all $\beta \in (\beta_0,1)$.
\end{proposition}
\begin{proof}
  Fix $\beta_0 \in (0,1)$. Throughout the proof we will use the fact that for all $\beta \in (\beta_0,1)$ the factors $1-\beta^m$ and $1/\beta^m-1$ for a given positive $m$ are commensurate with  $1-\beta$.

  \bigskip
  
\noindent \textit{\textbf{1)  Linear terms, case $\bm{|\delta|<  \beta \delta_0}$. } }    First, setting $I(\eta,\delta)=  \left|L(\eta)-L(\eta-\delta)\right|$, for all $|\delta| < \beta \delta_0$,
  \begin{eqnarray}
    \nonumber  I(\eta,\delta) \hspace{-1.2mm} &  \hspace{-1.2mm}  \le   \hspace{-1.2mm}  &   \hspace{-1.2mm}  \beta^c \left| e^{ -(1-\beta^\gamma) {|\eta|^\gamma \over \beta^\gamma}} \psi\left( {\eta \over \beta} \right) -  e^{ -(1-\beta^\gamma) {|\eta-\delta|^\gamma \over \beta^\gamma}} \psi\left( {\eta -\delta \over \beta} \right) \right|  \\
    \nonumber  \hspace{-1.2mm} &  \hspace{-1.2mm}  \le   \hspace{-1.2mm}  &   \hspace{-1.2mm} \beta^c \left| e^{ -{1-\beta^\gamma \over \beta^\gamma} {|\eta|^\gamma}} \hspace{-1.2mm} -  \hspace{-0.8mm}  e^{ -{1-\beta^\gamma \over \beta^\gamma} {|\eta-\delta|^\gamma}} \right| \left| \psi\left( {\eta  \over \beta} \right) \right|  \hspace{-0.5mm} +  \hspace{-0.5mm}  \beta^c  e^{ -{1-\beta^\gamma \over \beta^\gamma} {|\eta-\delta|^\gamma}} \left| \psi\left( {\eta \over \beta} \right)  \hspace{-0.8mm} -  \hspace{-0.8mm}  \psi\left( {\eta -\delta \over \beta} \right) \right|  \\
    \nonumber  \hspace{-1.2mm} &  \hspace{-1.2mm}  \le   \hspace{-1.2mm}  &   \hspace{-1.2mm} \beta^c \left| e^{ -{1-\beta^\gamma \over \beta^\gamma} {|\eta|^\gamma}} \hspace{-1.2mm} -  \hspace{-0.8mm}  e^{ -{1-\beta^\gamma \over \beta^\gamma} {|\eta-\delta|^\gamma}} \right| \left| \psi\left( {\eta  \over \beta} \right) \right|  \hspace{-0.3mm} +  \hspace{-0.3mm}  \beta^{c-\theta}  e^{ -{1-\beta^\gamma \over \beta^\gamma} {|\eta-\delta|^\gamma}}   \omega_\psi \left({\eta \over \beta} \right) |\delta|^\theta \\
 \nonumber  \hspace{-1.2mm} &  \hspace{-1.2mm}  =   \hspace{-1.2mm}  &   \hspace{-1.2mm} I_1(\eta,\delta)+I_2(\eta,\delta).
  \end{eqnarray}
Suppose $|\eta|>|\eta-\delta|$ then 
\begin{equation}
  \nonumber   \left| e^{ -(1-\beta^\gamma) {|\eta|^\gamma \over \beta^\gamma}} -  e^{ -(1-\beta^\gamma) {|\eta-\delta|^\gamma \over \beta^\gamma}} \right| \le   \left|1- e^{ -{ 1-\beta^\gamma \over \beta^\gamma} \left( |\eta|^\gamma -|\eta-\delta|^\gamma\right) } \right|.
\end{equation}
Otherwise, if $|\eta| \le |\eta-\delta|$ then 
\begin{equation}
  \nonumber   \left| e^{ -(1-\beta^\gamma) {|\eta|^\gamma \over \beta^\gamma}} -  e^{ -(1-\beta^\gamma) {|\eta-\delta|^\gamma \over \beta^\gamma}} \right| \le  \left|1- e^{ -{ 1-\beta^\gamma \over \beta^\gamma} \left( |\eta-\delta|^\gamma -|\eta|^\gamma \right) } \right|.
\end{equation}
In both cases, if $|\eta| > 2 |\delta|$, the following very conservative estimate holds:
\begin{equation}
  \label{expdiff}   \left| e^{ -(1-\beta^\gamma) {|\eta|^\gamma \over \beta^\gamma}} -  e^{ -(1-\beta^\gamma) {|\eta-\delta|^\gamma \over \beta^\gamma}} \right|\le  \left|1- e^{ -{ 1-\beta^\gamma \over \beta^\gamma} D |\eta|^{\gamma-1} |\delta| } \right| \le  { 1-\beta^\gamma \over \beta^\gamma} D |\eta|^{\gamma-1} |\delta|
\end{equation}
for some constant $D>0$. Therefore,
\begin{eqnarray}
  \nonumber \| I_1 \|^p_{v_\nu,p} & \le & \beta^{p c} \hspace{-3mm}  \int \displaylimits_{|\eta| \le 2 |\delta|  }  \hspace{-1mm} \left| e^{ -{ 1-\beta^\gamma \over \beta^\gamma} {|\eta|^\gamma}} -  e^{ -{ 1-\beta^\gamma \over \beta^\gamma} {|\eta-\delta|^\gamma}} \right|^p  \left| \psi\left( {\eta  \over \beta} \right) \right|^p v_\nu^p(\eta)  \  d \eta +  \\
  \nonumber && \hspace{20mm} +\beta^{p c}  \hspace{-3mm}  \int \displaylimits_{|\eta| > 2 |\delta|  }  \hspace{-1mm} \left| e^{ -{ 1-\beta^\gamma \over \beta^\gamma} {|\eta|^\gamma }} -  e^{ -{ 1-\beta^\gamma \over \beta^\gamma} {|\eta-\delta|^\gamma}} \right|^p  \left| \psi\left( {\eta  \over \beta} \right) \right|^p v_\nu^p(\eta)   \ d \eta \\
  \nonumber  & \le & \beta^{p c} \left(1-e^{-{1-\beta^\gamma  \over \beta^\gamma } |3 \delta|^\gamma   }    \right)^p \hspace{-2mm} \int \displaylimits_{|\eta| \le 2 |\delta|  }  \hspace{-1mm}   \left| \psi\left( {\eta  \over \beta} \right) \right|^p v_\nu^p(\eta)  \  d \eta + \\
  \nonumber && \hspace{20mm} + C k  (1-\beta^\gamma)^p |\delta|^p \hspace{-3mm} \int \displaylimits_{|\eta| > 2 |\delta|  }  \hspace{-1mm} |\eta|^{p(\sigma+\nu+\gamma-1)+d-1} e^{-{b \over 2} p |\eta|} \ d |\eta| \\
  \nonumber  & \le & C k  (1 \hspace{-0.5mm}  - \hspace{-0.5mm} \beta)^p \hspace{-1mm} \left( \hspace{-0.5mm}  |\delta|^{p \gamma}  \hspace{-5mm}  \int \displaylimits_{|\eta| \le 2 |\delta|  }  \hspace{-3.5mm}    |\eta|^{p(\sigma+\nu)+d-1}  \  d |\eta| \hspace{-0.5mm} + \hspace{-0.5mm}  |\delta|^{p} \hspace{-5mm}  \int \displaylimits_{|\eta| > 2 |\delta|  }  \hspace{-3.5mm} |\eta|^{p(\sigma+\nu+\gamma-1)+d-1} e^{-{b \over 2} p |\eta|} \ d |\eta| \hspace{-1mm} \right) \\
   \nonumber  & \le & C k  (1-\beta)^p \left( |\delta|^{p (\sigma+\nu)+d+p \gamma} +|\delta|^p \hspace{-3mm}  \int \displaylimits_{|\eta| > 2 |\delta|  }  \hspace{-3mm} |\eta|^{p(\sigma+\nu+\gamma-1)+d-1} e^{-{b \over 2} p |\eta|} \ d |\eta| \right).
\end{eqnarray}
The second integral is proportional to the incomplete $\Gamma$-function $\Gamma(p(\sigma+\nu+\gamma-1)+d, b |\delta|)$, and since  $p (\sigma + \nu) +d >0$ by $(\ref{H11})$, this $\Gamma$-function is bounded by a constant.  We have, therefore, 
\begin{equation} \label{Lestimate2}
  \| L-\tau_\delta L \|_{v_\nu,p}   \le   C k (1-\beta)  |\delta|^{\min\{1,\gamma\}} + \beta^{c-\theta+\nu+{d \over p}}  K  |\delta|^\theta.
  \end{equation}
At the same time, the bound on $I(\eta,\delta)$  itself for $|\eta|>1$ can be straightforwardly obtain using $(\ref{expdiff})$, the bound $(\ref{setM})$ on $\psi \in \cM$ and the assumption $(\ref{omegabound})$ 
\begin{equation}
 \label{Lterms} I(\eta,\delta) \le C k  (1-\beta) |\eta|^{\sigma+\gamma-1}  |\delta| e^{-b|\eta|} + A \beta^{c-\theta-\tau}  e^{- {1-\beta^\gamma \over \beta^\gamma} |\eta-\delta|^\gamma} |\eta|^\tau e^{-b {|\eta| \over \beta}} |\delta|^\theta. 
\end{equation}

\bigskip

 \noindent \textit{\textbf{2)  Non-linear terms, case $\bm{|\delta|<  \beta \delta_0}$. } }    First,  for all $|\delta| < \beta \delta_0$,
 Consider the nonlinear terms of the renormalization operator $(\ref{operatorR4})$:
 \begin{eqnarray}
  \nonumber N(\eta)&=&  \gamma  \int \displaylimits_{1}^{1 \over \beta}  \hspace{-0.8mm}  { e^{|\eta|^\gamma \left(1-t^\gamma\right)}  |\eta_a|  \over  t^{c} }  \int \displaylimits_{\RR^d}  \psi(x) \psi(\eta t  \hspace{-0.2mm} -  \hspace{-0.2mm}x) \cos\hat x_{d-1} \sin {\reallywidehat{(\eta t -x)}}_{d-1}   \ d x    \ d  t.
 \end{eqnarray}
 We have
 $$|N(\eta)-N(\eta-\delta)| \le |N_1(\eta,\delta)|+|N_2(\eta,\delta)|+|N_3(\eta,\delta)|+|N_4(\eta,\delta)|,$$
 where
 \begin{eqnarray}
   \nonumber |N_1(\eta,\delta)|  & \le  &  C  \int \displaylimits_{1}^{1 \over \beta}   { \left| e^{|\eta|^\gamma \left(1-t^\gamma\right)}  - e^{|\eta-\delta|^\gamma \left(1-t^\gamma\right)}\right|  |\eta_a|  \over  t^{c} } \int \displaylimits_{\RR^d}   |\psi(x)| |\psi(\eta t  -  x)|   d x    d  t \\
   \nonumber |N_2(\eta,\delta)|  & \le & \int \displaylimits_{1}^{1 \over \beta}   {e^{|\eta-\delta|^\gamma \left(1-t^\gamma\right)} |  |\eta_a|-|\eta_a -\delta|  | \over  t^{c} }   \int \displaylimits_{\RR^d} |\psi(x)| |\psi(\eta t-x)|   d x    d  t \\
   \nonumber |N_3(\eta,\delta) | &\le& C \int \displaylimits_{1}^{1 \over \beta}  {e^{|\eta-\delta|^\gamma \left(1-t^\gamma\right)} |\eta_a - \delta|  \over  t^{c} }   \int \displaylimits_{\RR^d} |\psi(x)| |\psi(\eta t  - x)| \times \\
   \nonumber && \hspace{+30mm} \times \left| \sin {  \reallywidehat{(\eta t  -  x)}}_{d-1} -  \sin {  \reallywidehat{((\eta-\delta) t - x)}}_{d-1} \right|   d x    d  t \\
   \nonumber |N_4(\eta,\delta)|  &\le& C   \int \displaylimits_{1}^{1 \over \beta}   {e^{|\eta-\delta|^\gamma \left(1-t^\gamma\right)} |\eta_a- \delta|  \over  t^{c} }  \int \displaylimits_{\RR^d}  \hspace{-1mm} |\psi(x)| \left| \psi(\eta t -x)-\psi((\eta-\delta) t -x) \right|   d x    d  t
 \end{eqnarray}

 We consider each $N_i$ separately

 \bigskip
 
 \noindent \underline{\it $a)$ Calculation of $N_1$.}  We use $(\ref{I})$ to bound the convolution:
 \begin{eqnarray}
   \nonumber |N_1(\eta,\delta)|  & \le &  C k^2 \int \displaylimits_{1}^{1 \over \beta}  { \left| e^{|\eta|^\gamma \left(1-t^\gamma\right)}  -  e^{|\eta-\delta|^\gamma \left(1-t^\gamma\right)}\right|  |\eta_a|  \over  t^{c} }  e^{-b |\eta t|} |\eta t|^{2 \sigma+d}  d  t.
 \end{eqnarray}

 We will first evaluate the difference of the exponentials.
\medskip

\begin{lemma}
For every $\gamma > 1$, there exist a constant $C$, such that
\begin{equation}
  \label{L5}   \left| e^{ -(1-\beta^\gamma) {|\eta|^\gamma \over \beta^\gamma}} -  e^{ -(1-\beta^\gamma) {|\eta-\delta|^\gamma \over \beta^\gamma}} \right| < C \ \left({1 \over \beta^\gamma}-1 \right)^{1 \over \gamma} |\delta|.
\end{equation}
\end{lemma}
\begin{proof}
We have in the case $|\eta|>|\eta-\delta|$ and $|\eta|>|\delta|$
\begin{eqnarray}
  \nonumber   \left| e^{ -{1-\beta^\gamma \over \beta^\gamma} {|\eta|^\gamma}} -  e^{ -{1-\beta^\gamma \over \beta^\gamma} {|\eta-\delta|^\gamma }} \right| & <  & e^{ -{1-\beta^\gamma \over \beta^\gamma} {|\eta-\delta|^\gamma}}  \left| 1 -  e^{ {1-\beta^\gamma \over \beta^\gamma} \left( |\eta-\delta|^\gamma - |\eta|^\gamma    \right)} \right| \\
  \nonumber & < &   e^{ -{1-\beta^\gamma \over  \beta^\gamma}  (|\eta|-|\delta|)^\gamma}  \left( 1 -  e^{ -{ 1-\beta^\gamma \over \beta^\gamma } D  |\eta|^{\gamma-1} |\delta|} \right) \\
\label{bound18} & < &   e^{ {1-\beta^\gamma \over  \beta^\gamma}  |\delta|^\gamma}   e^{ -{2 \over 2^\gamma} {1-\beta^\gamma \over  \beta^\gamma}  |\eta|^\gamma}  \left( 1 -  e^{ -{ 1-\beta^\gamma \over \beta^\gamma } D  |\eta|^{\gamma-1} |\delta|} \right)
\end{eqnarray}
for some constant $D>0$. To get a bound on this expression, proportional to a power of $|\delta|$, we first notice that
\begin{equation}
  \label{bound9} e^{-x^\gamma}  \left(1 - e^{-a x^{\gamma-1}} \right)<{ 2 a x^{\gamma-1} \over (1 + a x^{\gamma-1}) (1 +  x^\gamma )  } < { 2 a x^{\gamma-1} \over  (1 + x^\gamma )  } \le 2 a  { (\gamma -1)^{\gamma-1 \over \gamma} \over \gamma }.
\end{equation}
We can now use this estimate on $(\ref{bound18})$ with
$$x={2^{1 \over \gamma} \over 2} \left({1 \over \beta^\gamma}-1\right)^{1\over \gamma} |\eta|, \quad a= D |\delta| 2^{(\gamma-1)^2 \over \gamma} \left({1 \over \beta^\gamma}-1 \right)^{1\over \gamma},$$
\begin{equation}
  \label{L1}   \left| e^{ -{1-\beta^\gamma \over \beta^\gamma} {|\eta|^\gamma}} -  e^{ -{1-\beta^\gamma \over \beta^\gamma} {|\eta-\delta|^\gamma}} \right|  <  C \ e^{ {1-\beta^\gamma \over  \beta^\gamma} |\delta|^\gamma}    \left({1 \over \beta^\gamma} - 1 \right)^{1\over \gamma}   \ |\delta|.
\end{equation}
In the case $|\eta|>|\eta-\delta|$ and $|\eta|<|\delta|$,
\begin{equation}
  \label{L2}   \left| e^{ -{1-\beta^\gamma \over \beta^\gamma} {|\eta|^\gamma}} -  e^{ -{1-\beta^\gamma \over \beta^\gamma} {|\eta-\delta|^\gamma }} \right|  <     \left( 1 -  e^{ -{ 1-\beta^\gamma \over \beta^\gamma } D |\delta|^{\gamma-1} |\eta|} \right) \le  C \   \left({1 \over \beta^\gamma} - 1 \right)  |\delta|^\gamma .
\end{equation}
In the case $|\eta-\delta|>|\eta|>|\delta|$, for some constant $D>0$,
\begin{equation}
  \nonumber   \left| e^{ -{1-\beta^\gamma \over \beta^\gamma} {|\eta|^\gamma }} -  e^{ -{1-\beta^\gamma \over \beta^\gamma} {|\eta-\delta|^\gamma}} \right|  <   e^{ -{1-\beta^\gamma \over \beta^\gamma} {|\eta|^\gamma }} \left(1-e^{-{1-\beta^\gamma \over \beta^\gamma} D|\delta| |\eta|^{\gamma-1}} \right).
\end{equation}
We can again use estimate $(\ref{bound9})$ in this case:
\begin{equation}
  \label{L3}   \left| e^{ -(1-\beta^\gamma) {|\eta|^\gamma \over \beta^\gamma}} -  e^{ -(1-\beta^\gamma) {|\eta-\delta|^\gamma \over \beta^\gamma}} \right| <  C\  \left({1 \over \beta^\gamma}-1 \right)^{1\over \gamma}   \ |\delta|.
\end{equation}
In the case $|\eta-\delta|>|\eta|$ and $|\delta|>|\eta|$,
\begin{eqnarray}
  \nonumber   \left| e^{ -{1-\beta^\gamma \over \beta^\gamma} {|\eta|^\gamma}} -  e^{ -{1-\beta^\gamma \over \beta^\gamma} {|\eta-\delta|^\gamma}} \right| &< & e^{ -{1-\beta^\gamma \over \beta^\gamma} {|\eta|^\gamma }} \left(1-e^{-{1-\beta^\gamma \over \beta^\gamma} D|\eta| |\delta|^{\gamma-1}} \right) \\
  \label{bound10}  & <&  e^{ -{1-\beta^\gamma \over \beta^\gamma} {|\eta|^\gamma }} \left(1-e^{-{1-\beta^\gamma \over \beta^\gamma} D|\delta|^{\gamma}} \right).
\end{eqnarray}
\begin{equation}
  \label{L4}   \left| e^{ -(1-\beta^\gamma) {|\eta|^\gamma \over \beta^\gamma}} -  e^{ -(1-\beta^\gamma) {|\eta-\delta|^\gamma \over \beta^\gamma}} \right| <  C \ \left({1 \over \beta^\gamma} -1\right)  |\delta|^{\gamma}.
\end{equation}
We can now combine bound $(\ref{L1})$, $(\ref{L2})$,  $(\ref{L3})$,  $(\ref{L4})$ to claim  the result. 
\end{proof}

\medskip

We use the bound $(\ref{L5})$ on the difference of exponentials in $N_1$ as follows:
   \begin{equation*}
  \nonumber   \left| e^{ -(t^\gamma-1) {|\eta|^\gamma}} -  e^{ -(t^\gamma-1) {|\eta-\delta|^\gamma}} \right|  <    C (t^\gamma-1)^{1 \over \gamma} |\delta|.
\end{equation*}
Therefore,
  \begin{eqnarray}
    \nonumber |N_1(\eta,\delta)|   &   \le  &  C k^2 |\eta|^{2 \sigma+d+1  }   |\delta|  \int \displaylimits_{1}^{1 \over \beta}  (t-1)^{1 \over \gamma} e^{-b |\eta|}  d  t \\
    \label{N1terms}  &\le &   C k^2 ( 1 -\beta)^{1+\gamma \over \gamma} |\eta|^{2 \sigma+d+1} e^{-b |\eta| }   |\delta|.
\end{eqnarray}
  The condition for finiteness of $\|N_1\|_{v_\nu,p}$ is
  \begin{equation} \label{condd1}
    (2 \sigma+d+1+\nu) p>-d,
  \end{equation}
  while the norm itself can be bounded as
  \begin{equation}
\label{N1}
 \| N_1\|_{v_\nu,p}   \le C k^2 (1 -\beta )^{\gamma +1 \over \gamma}    |\delta|.  
\end{equation}
  
 \bigskip
 
 \noindent \underline{\it $b)$ Calculation of $N_2$.} 
 \begin{eqnarray}
   \nonumber |N_2(\eta,\delta)|  & \le &  C k^2  \int \displaylimits_{1}^{1 \over \beta}  { e^{|\eta-\delta|^\gamma \left(1-t^\gamma\right)}  |\delta|  \over  t^{c} }  e^{-b |\eta t|} |\eta t|^{2 \sigma+d}  d  t  \\
   \nonumber &\le&   C k^2 |\eta|^{2 \sigma +d} |\delta| \int \displaylimits_{1}^{1 \over \beta} e^{-b |\eta t|} |t|^{2 \sigma+d-c}  d  t \\
   \nonumber &\le&   C k^2 \left({1 \over \beta^{2 \sigma +d-c+1}}-1  \right)  e^{-b |\eta|} |\eta|^{2 \sigma +d} |\delta|  \\
      \label{N2terms} &\le&   C k^2 (1- \beta) |\eta|^{2 \sigma +d}  e^{-b |\eta|}  |\delta|.
 \end{eqnarray}
The conditions of finiteness of  $\|N_2\|_{v_\nu,p}$ is
\begin{equation}
\label{condd2} (2\sigma+d+\nu) p >-d,
 \end{equation}
while the norm itself is bounded as
 \begin{equation}
   \label{N2} \|N_2\|_{v_\nu,p} \le   C k^2 (1-\beta)   |\delta|.
 \end{equation}

 \bigskip
 
 \noindent  \underline{\it $c)$ Calculation of $N_3$.}
 
 \noindent {\bf i) Case $\pmb{|\eta t-x| > |\delta|/2}$}.

   In this case,
 \begin{eqnarray}
   \nonumber \left| \sin { \hspace{-0.5mm} \reallywidehat{(\eta t \hspace{-0.4mm} - \hspace{-0.4 mm} x)}}_{d-1} -  \sin { \hspace{-0.5mm} \reallywidehat{((\eta-\delta) t \hspace{-0.4mm} - \hspace{-0.4 mm} x)}}_{d-1} \right| &\le & C \sin \left({|  \reallywidehat{((\eta-\delta) t \hspace{-0.4mm} - \hspace{-0.4 mm} x)}_{d-1}  -  \reallywidehat{(\eta t \hspace{-0.4mm} - \hspace{-0.4 mm} x)}_{d-1}      | \over 2 }  \right) \\
       \nonumber  &\le & C { |\delta t|  \over |\eta t- x|}.  
 \end{eqnarray}
 Therefore, setting $\eta t=y$ and largely repeating the calculations for the convolution in Proposition $\ref{inv3}$ for
 $$I(|y|)=\int \displaylimits_{|y-x|>{|\delta| \over 2}}  \hspace{-0.5mm}  |x|^{\sigma} e^{- b {|x|}}  |y -x|^{\sigma-1} e^{- b {|y - x|}} d x =\int \displaylimits_{|x|>{|\delta| \over 2}}  \hspace{-0.5mm}  |x|^{\sigma-1} e^{- b {|x|}}  |y -x|^{\sigma} e^{- b {|y - x|}}d x,$$
 we get
\begin{eqnarray}
  \nonumber I(|y|) \hspace{-2mm}   &\le& \hspace{-2.0mm} C \int \displaylimits_{{|\delta| \over 2 }}^{\infty}  \hspace{-0.5mm}  r^{\sigma+d-3} e^{- b r-b||y|-r|}  {1 \over |y|} \left| ||y|-r|^{\sigma+2} -(|y|+r)^{\sigma+2} \right| \  d r.
\end{eqnarray}
We first look at {\bf the case $\pmb{|y|>|\delta|/2}$}. Then
\begin{equation}
  \label{eq00}  I(|y|)=  I_1(|y|)+I_2(|y|),
\end{equation}
where $I_1$ an $I_2$ are the integrals over ranges $(|\delta| / 2,|y|)$ and $(|y|,\infty)$, respectively. We start by evaluating $I_1$:
\begin{eqnarray}
  \nonumber I_1(|y|)  \hspace{-2mm}   &\le& \hspace{-2.0mm}  C |y|^{2 \sigma+d-1} e^{-b |y|} \int \displaylimits_{{|\delta| \over 2 |y|}}^{1}  \hspace{-0.5mm}  z^{\sigma+d-3}  \left| (1-z)^{\sigma+2} -(1+z)^{\sigma+2} \right|  \  d z \\
  \nonumber \hspace{-2mm}   &\le& \hspace{-2.0mm} C |y|^{2 \sigma+d-1} e^{-b |y|} \int \displaylimits_{{|\delta| \over 2 |y|} }^{1}  \hspace{-0.5mm}  z^{\sigma+d-2} \ \phantom{}_2F_1\left({1 \over 2} -{\sigma +1 \over 2},-{\sigma +1 \over 2}, {3 \over 2}, z^2  \right)  d z.
\end{eqnarray}
We will consider only the case $\sigma >-1$ as this is a weaker condition than $\sigma >-(p-1)/p$ appearing in the hypothesis. If $\sigma>-2$ then the Gauss hypergeometric function $\phantom{}_2F_1(a,b,c,z)$ can be bounded by a constant (since $c-a-b>0$). Additionally, for all $d \ge 2$ and $\sigma >-1$, $\sigma+d-1>0$, and
\begin{equation}
  \nonumber I_1(|y|)  \le  C |y|^{2 \sigma+d-1} e^{-b |y|} \int \displaylimits_{|\delta| \over 2 |y|}^{1}  \hspace{-0.5mm}  z^{\sigma+d-2}  d z\le  C  |y|^{2 \sigma+d-1} e^{-b |y|}   \left(1-\left({|\delta| \over 2 |y|   } \right)^{\sigma+d-1} \right).
\end{equation}
We now turn to $I_2$:
\begin{eqnarray}
  \nonumber I_2(|y|)    \hspace{-2mm}   &\le& \hspace{-2.0mm} C  e^{b|y|} |y|^{2 \sigma+d-1} \int \displaylimits_{1}^{\infty}  \hspace{-0.5mm}  z^{\sigma+d-3} ((z+1)^{\sigma+2}-(z-1)^{\sigma+2}) e^{- 2 b|y| z}  \  d z \\
  \nonumber  \hspace{-2mm}   &\le& \hspace{-2.0mm} C  e^{b|y|} |y|^{2 \sigma+d-1} \int \displaylimits_{1}^{\infty}  \hspace{-0.5mm}  z^{2 \sigma+d-1} ((1+z^{-1})^{\sigma+2}-(1-z^{-1})^{\sigma+2}) e^{- 2 b|y| z}  \  d z \\
  \nonumber  \hspace{-2mm}   &\le& \hspace{-2.0mm} C  e^{b|y|} |y|^{2 \sigma+d-1} \int \displaylimits_{1}^{\infty}  \hspace{-0.5mm}  z^{2 \sigma+d-2} \phantom{}_2F_1 \left( {1 \over 2} -{\sigma+1 \over 2  },-{ \sigma+1 \over 2}, {3 \over 2}, z^{-2} \right) e^{- 2 b|y| z}  \  d z \\
  \nonumber  \hspace{-2mm}   &\le& \hspace{-2.0mm} C  e^{b|y|} |y|^{2 \sigma+d-1} \int \displaylimits_{1}^{\infty}  \hspace{-0.5mm}  z^{2 \sigma+d-2} e^{- 2 b|y| z}  \  d z \\
  \nonumber  \hspace{-2mm}   &\le& \hspace{-2.0mm} C \ e^{b|y|} \ \Gamma( 2 \sigma+d-1,2 b|y|) \\
\label{I22}  \hspace{-2mm}   &\le& \hspace{-2.0mm} C \  \left\{1-{|2 b y|^{2 \sigma +d -1} \over 2 \sigma+d-1}, \hspace{5.4mm}  |2 b y|<1, \atop   e^{-b|y|} |y|^{2 \sigma+d-2}, \quad |2 b y| \ge 1.     \right.
\end{eqnarray}

We summarize:
\begin{equation}
  \label{II3}I(|y|) \le C e^{-b|y|} \left( 1-{|y|^{2 \sigma +d -1} \over 2 \sigma+d-1} + |y|^{2 \sigma+d-1}  \left(1-{|\delta|^{\sigma+d-1}  \over 2^{\sigma+d-1} |y|^{\sigma+d-1}}  \right) \right)
  \end{equation}
for small $|y|$, and
\begin{equation}
  \label{II4}I(|y|) \le C e^{-b|y|} \left( |y|^{2 \sigma+d-2}+|y|^{2 \sigma+d-1}  \left(1-{|\delta|^{\sigma+d-1}  \over 2^{\sigma+d-1} |y|^{\sigma+d-1}}  \right) \right)
\end{equation}
for large $|y|$. 

Now we turn to {\bf the case} $\pmb{|y| \le |\delta|/2}$. In this case the computation of $I$ is similar to to $(\ref{I22})$ with the lower bound substituted by $|\delta|/2 |y|$, and, eventually, can be bounded by $I_2$:
\begin{equation}
\nonumber I(|y|)\le  C \left\{1-{|2 b y|^{2 \sigma +d -1} \over 2 \sigma+d-1}, \hspace{6mm}  |2 by|<1, \atop   \ e^{-b|y|} |y|^{2 \sigma+d-2}, \quad |2 b y| \ge 1,     \right.
\end{equation}
however, since $|y| \le |\delta|/2$, we simply get that
\begin{equation}
\nonumber I(|y|)\le C-C {| y|^{2 \sigma +d -1} \over 2 \sigma+d-1}.
\end{equation}

To summarize, in the case $|\eta t-x| >  |\delta|/2$,
\begin{equation}
  \label{II5}I(|y|) \le C e^{-b|y|} \left\{ 1-{|y|^{2 \sigma +d -1} \over 2 \sigma+d-1},  \hspace{5mm} |y| < |\delta|/2, \atop  |y|^{2 \sigma+d-1}, \hspace{9mm} |y| \ge  |\delta|/2\right.
\end{equation}

\noindent {\bf ii) Case $\pmb{|\eta t-x| \le  |\delta|/2}$}.

  We now turn to the case  $|\eta t -x| \le  |\delta|/2$. The difference of two sines can be now bounded by a constant, and we get that, as in Proposition $\ref{inv3}$,
\begin{eqnarray}
\nonumber J(|y|) &:= & \int \displaylimits_{|x|<{|\delta| \over 2}}  \hspace{-0.5mm}  |x|^{\sigma} e^{- b {|x|}}  |y -x|^{\sigma} e^{- b {|y - x|}} d x \\
\nonumber    &\le&  C \int \displaylimits_{0}^{|\delta| \over 2}   r^{\sigma+d-2} e^{- b r-b||y|-r|}  {1 \over |y|} \left| ||y|-r|^{\sigma+2} -(|y|+r)^{\sigma+2} \right|  \  d r.
\end{eqnarray}
{\bf If $\pmb{|y| < |\delta| /2}$}, we can bound $J$ as
\begin{eqnarray}
  \nonumber   J(|y|) &\le&  C \int \displaylimits_{0}^{|y|}   r^{\sigma+d-2} e^{- b r-b||y|-r|}  {1 \over |y|} \left| ||y|-r|^{\sigma+2} -(|y|+r)^{\sigma+2} \right|  \  d r +\\
  \nonumber  && \hspace{10mm} +  C \int \displaylimits_{|y|}^{|\delta| \over 2}   r^{\sigma+d-2} e^{- b r-b||y|-r|}  {1 \over |y|} \left| ||y|-r|^{\sigma+2} -(|y|+r)^{\sigma+2} \right|  \  d r \\
  \nonumber &=& J_1(|y|)+J_2(|y|).
\end{eqnarray}
$J_1$ is nothing but the integral $I_1$ from Proposition  $\ref{inv3}$:
\begin{equation}
 \label{J1} J_1(|y|) \le C |y|^{2 \sigma +d} e^{-b |y|}. 
\end{equation}
Bounding integral $J_2$ is similar to $I_2$ in the previous subsection:
\begin{eqnarray}
  \nonumber J_2(|y|)   \hspace{-2mm}   &\le& \hspace{-2.0mm} C  e^{b|y|} |y|^{2 \sigma+d} \int \displaylimits_{1}^{{|\delta| \over 2 |y|} }  \hspace{-0.5mm}  z^{\sigma+d-2} \left|(z-1)^{\sigma+2} - (z+1)^{\sigma+2} \right| e^{- 2 b|y| z}  \  d z \\
  \nonumber   \hspace{-2mm}   &\le& \hspace{-2.0mm} C  e^{b|y|} |y|^{2 \sigma+d} \int \displaylimits_{1}^{{|\delta| \over 2 |y|} }  \hspace{-0.5mm}  z^{2 \sigma+d} \left((1+z^{-1})^{\sigma+2} - (1-z^{-1})^{\sigma+2} \right) e^{- 2 b|y| z}  \  d z \\
  \nonumber  \hspace{-2mm}   &\le& \hspace{-2.0mm} C  e^{b|y|} |y|^{2 \sigma+d} \int \displaylimits_{1}^{{|\delta| \over 2 |y|} }  \hspace{-0.5mm}  z^{2 \sigma+d-1} e^{- 2 b|y| z}  \  d z \\
   \nonumber  \hspace{-2mm}   &\le& \hspace{-2.0mm} C  e^{-b|y|}  \left({|\delta|^{2 \sigma+d} \over 2^{2 \sigma+d}}-|y|^{2 \sigma+d}   \right).
\end{eqnarray}

{\bf In case} $\pmb{|y| \ge |\delta| /2}$, we bound $J$ again using the fact that the hypergeometric function for its values of the arguments is bounded by a constant:
\begin{eqnarray}
  \nonumber J(|y|)    &\le&  C |y|^{2 \sigma+d} e^{-b |y|} \int \displaylimits_{0}^{{|\delta| \over 2 |y|}}    z^{\sigma+d-1} \phantom{}_2F_1 \left({1 \over 2} -{\sigma+1 \over 2}, -{\sigma+1 \over 2}, {3 \over 2}, z^2  \right)  d z \\
    \label{J2}    &\le&  C |y|^{\sigma} e^{-b |y|} |\delta|^{d +\sigma}.
\end{eqnarray}
Therefore, in the case  $|\eta t-x| \le  |\delta|/2$ we get the same bound for all $y$:
\begin{equation}
 \label{J3} J(|y|) \le C  e^{-b |y|}  |y|^{\sigma} |\delta|^{d +\sigma}.
\end{equation}

We conclude that the bound on the convolution integral in $N_3$ is the sum of  $(\ref{II5})$ and $(\ref{J3})$:
\begin{eqnarray}
 \nonumber  \int \displaylimits_{\RR^d} |\psi(x)| |\psi(y \hspace{-1.5mm}  & \hspace{-1.5mm} - \hspace{-1.5mm} & \hspace{-1.5mm} x)| \left| \sin {  \reallywidehat{(y  -  x)}}_{d-1} -  \sin {  \reallywidehat{((y-\delta t) - x)}}_{d-1} \right|   d x  \\
  \nonumber  \hspace{-1.5mm}  & \hspace{-1.5mm} \le \hspace{-1.5mm} & \hspace{-1.5mm} C e^{-b|y|} \left\{ { \left(1-{|y|^{2 \sigma+d-1} \over 2 \sigma +d-1} \right) |\delta t| + |y|^\sigma|\delta|^{\sigma +d} , \quad \hspace{6.1mm}  |y| \le { |\delta t| \over 2}, \atop  |y|^{2 \sigma+d-1} |\delta t| \hspace{-0.5mm} + \hspace{-0.5mm} |y|^{\sigma} |\delta t|^{d +\sigma},\hspace{20.5mm} |y|>{|\delta t| \over 2}. } \right.
\end{eqnarray}

We get eventually for $|\eta| \le |\delta|/2$:
\begin{eqnarray}
  \nonumber |N_3(\eta,\delta)|  \hspace{-0.5mm} &  \hspace{-0.5mm} \le  \hspace{-0.5mm} &  \hspace{-0.5mm}  C k^2 \int \displaylimits_1^{1 \over \beta} { e^{|\eta-\delta|(1-t^\gamma)} |\eta| \over t^c}  e^{-b |\eta t|}  \left(
  \left(1-{|\eta t|^{2 \sigma+d-1} \over 2 \sigma +d-1} \right) |\delta t| + |\eta t|^\sigma|\delta|^{\sigma +d} 
  \right) d t \\
  \nonumber  \hspace{-0.5mm} &  \hspace{-0.5mm} \le  \hspace{-0.5mm} &  \hspace{-0.5mm}  C k^2 (1-\beta) \  e^{-b |\eta|}   \left(
  \left(1-{|\eta|^{2 \sigma+d-1} \over 2 \sigma +d-1} \right) |\delta| + |\eta|^\sigma|\delta|^{\sigma +d} 
  \right),
\end{eqnarray}
and for $|\eta| > |\delta|/2$:
\begin{eqnarray}
  \nonumber |N_3(\eta,\delta)|  \hspace{-0.5mm} &  \hspace{-0.5mm} \le  \hspace{-0.5mm} &  \hspace{-0.5mm}  C k^2 \int \displaylimits_1^{1 \over \beta} { e^{|\eta-\delta|(1-t^\gamma)} |\eta| \over t^c}  e^{-b |\eta t|}  \left(
 |\eta t |^{2 \sigma+d-1} |\delta t| \hspace{-0.5mm} + \hspace{-0.5mm} |\eta t|^{\sigma} |\delta t|^{d +\sigma}
  \right) d t \\
  \label{N3terms}  \hspace{-0.5mm} &  \hspace{-0.5mm} \le  \hspace{-0.5mm} &  \hspace{-0.5mm}  C k^2 (1-\beta) e^{-b |\eta|}  \left( |\eta|^{2 \sigma+d-1} |\delta| \hspace{-0.5mm} + \hspace{-0.5mm} |\eta|^{\sigma} |\delta|^{d +\sigma} \right).
\end{eqnarray}
Therefore the condition
\begin{eqnarray}
  \label{condd3}  (\sigma+\nu)p&>&-d,\\
  \label{condd4}  (2 \sigma+d-1+\nu)p&>&-d,
\end{eqnarray}
are sufficient for 
$N_3$ to be in $L_{v_\nu,p}$:
\begin{align}
  \nonumber \| N_3\|_{v_\nu,p}^p   &  \le   C k^{2 p}  (1 -\beta  )^p  \hspace{-1mm} \int  \displaylimits_{0}^{|\delta| \over 2} \hspace{-1mm} e^{-{b \over 2} p |\eta|} |\eta|^{\nu p+d-1}  \hspace{-1mm}  \left( \hspace{-1.0mm}  \left(1\hspace{-0.5mm} - \hspace{-0.5mm} {|\eta|^{2 \sigma+d-1} \over 2 \sigma \hspace{-0.3mm}  +  \hspace{-0.3mm}  d  \hspace{-0.3mm}   -  \hspace{-0.3mm}   1} \right) |\delta|  \hspace{-0.5mm}   +   \hspace{-0.5mm}  |\eta|^\sigma|\delta|^{\sigma +d} 
  \right)^p  \hspace{-1.0mm}  d |\eta| +\\
  \nonumber    &   \phantom{ \le  C k^{2 p}  (1}  \hspace{1.5mm}  + C k^{2 p}  \left(1 -\beta  \right)^p  \hspace{-1mm}  \int  \displaylimits_{|\delta| \over 2}^\infty e^{-{b \over 2} p |\eta|}  |\eta|^{\nu p +d -1}  \left( |\eta|^{2 \sigma+d-1} |\delta| \hspace{-0.5mm} + \hspace{-0.5mm} |\eta|^{\sigma} |\delta|^{d +\sigma} \right)^p  d |\eta| \\
  \label{N3} & \le   C k^{2 p}  \left( 1 -\beta  \right)^p \left( |\delta|^{(2 \sigma+d +\nu)p+d} +|\delta|^{(\nu+1)p +d } + |\delta|^p+|\delta|^{(d+\sigma)p} \right).
\end{align}

\bigskip
 
\noindent  \underline{\it $d)$ Calculation of $N_4$.} 
We use Young's Convolution Inequality $\ref{Young2}$ with $p=q$, $r=\infty$, i.e. with
\begin{equation}
\label{kappap}\kappa=2 \nu -d {p -2 \over p},
\end{equation}
to get:
\begin{eqnarray}
  \nonumber | N_4(\eta,\delta)|  \hspace{-1mm} & \hspace{-1mm}  \le \hspace{-1mm} & \hspace{-2mm}  \gamma  \hspace{-1mm} \int \displaylimits_{1}^{1 \over \beta}  \hspace{-0.8mm}  {e^{|\eta-\delta|^\gamma \left(1-t^\gamma\right)} |\eta_a \hspace{-0.4mm}- \hspace{-0.3mm}\delta_a|  \over  t^{c} }  \hspace{-1.0mm}  \int \displaylimits_{\RR^d}  \hspace{-1mm} |\psi(x)|  \omega_\psi(\eta t -x ) |\delta t|^\theta   d x    d  t \\
  \nonumber  \hspace{-1mm} & \hspace{-1mm}  \le \hspace{-1mm} & \hspace{-2mm}  C  \  |\delta|^\theta  |\eta|^{c-\theta  } \ e^{|\eta-\delta|^\gamma } { |\eta-\delta| \over |\eta|} \hspace{-3mm} \int \displaylimits_{B_{|\eta| \over \beta} \setminus B_{|\eta|}}   \hspace{-5.0mm}  e^{-{|\eta-\delta|^\gamma \over |\eta|^\gamma} |y|^\gamma} |y|^{\theta-c+1-d}  \int \displaylimits_{\RR^d}  \hspace{-1mm} |\psi(x)|  \omega_\psi(y -x )   d x    d y \\
  \nonumber  \hspace{-1mm} & \hspace{-1mm}  \le \hspace{-1mm} & \hspace{-2mm}  C    |\delta|^\theta  e^{|\eta-\delta|^\gamma }  \hspace{-0.5mm}  { |\eta \shortminus \delta| \over |\eta|^{1+\theta-c}}  \hspace{-4.0mm} \int \displaylimits_{B_{|\eta| \over \beta} \setminus B_{|\eta|}}   \hspace{-5.5mm}  e^{\shortminus {|\eta \shortminus \delta|^\gamma \over |\eta|^\gamma} |y|^\gamma} |y|^{\theta \shortminus c + 1 \shortminus d \shortminus \kappa} e^{\shortminus {b \over 2}|y|} \hspace{-0.7mm}  \left( \hspace{-0.5mm}  (|\psi| \star \omega_\psi)(y) |y|^{\kappa} e^{{b \over 2}|y|} \right)  d y \\
  \nonumber  \hspace{-1mm} & \hspace{-1mm}  \le \hspace{-1mm} & \hspace{-2mm}  C  \  |\delta|^\theta  \ e^{|\eta-\delta|^\gamma }   { |\eta \shortminus \delta| \over |\eta|^{1+\theta-c}}  \hspace{-4.0mm} \int \displaylimits_{B_{|\eta| \over \beta} \setminus B_{|\eta|}}   \hspace{-4.0mm}  e^{-{|\eta-\delta|^\gamma \over |\eta|^\gamma} |y|^\gamma} |y|^{\theta-c+1-d-\kappa} e^{-{b \over 2} |y|} d y   \  \| \psi \|_{v_\nu,p} \| \omega_\psi\|_{v_\nu,p} \\
  \nonumber  \hspace{-1mm} & \hspace{-1mm}  \le \hspace{-1mm} & \hspace{-2mm}  C  \  |\delta|^\theta   e^{-{b \over 2} |\eta|} e^{|\eta-\delta|^\gamma}   { |\eta \shortminus \delta| \over |\eta|^{1+\theta-c}}  \hspace{-1.0mm} \int \displaylimits_{|\eta|^\gamma}^{|\eta|^\gamma \over \beta^\gamma}  \hspace{-1.0mm}   e^{-{|\eta-\delta|^\gamma \over |\eta|^\gamma} z} z^{{\theta-c-\kappa+1 \over \gamma} -1}  d z   \  \| \psi \|_{v_\nu,p} \| \omega_\psi\|_{v_\nu,p} \\
  \nonumber  \hspace{-1mm} & \hspace{-1mm}  \le \hspace{-1mm} & \hspace{-2mm}  C  \  |\delta|^\theta { |\eta|^{-\kappa}  \over |\eta-\delta|^{\theta-c-\kappa}}   e^{-{b \over 2} |\eta| -|\eta-\delta|^\gamma}  \left|  \hspace{0.2mm} \Gamma\hspace{-0.6mm}\left(\hspace{-0.6mm} {\theta \hspace{-0.6mm} - \hspace{-0.6mm} c \hspace{-0.6mm} -\hspace{-0.6mm} \kappa \hspace{-0.6mm} +\hspace{-0.6mm} 1 \over \gamma}, |\eta-\delta|^\gamma \right)  \hspace{-0.6mm} - \right.\\
  \nonumber  && \left. \hspace{50mm} - \hspace{-0.6mm}  \Gamma\hspace{-0.6mm}\left(\hspace{-0.6mm}{\theta \hspace{-0.6mm} -\hspace{-0.6mm} c\hspace{-0.6mm} -\hspace{-0.6mm}\kappa\hspace{-0.6mm}+\hspace{-0.6mm}1 \over \gamma},{ |\eta-\delta|^\gamma \over \beta^\gamma} \right) \hspace{-0.5mm}   \right| \hspace{-0.6mm}  \| \psi \|_{v_\nu,p} \| \omega_\psi\|_{v_\nu,p}.
\end{eqnarray}
For small values of $|\eta-\delta|$ the above can be bounded as follows:
\begin{eqnarray}
\nonumber | N_4(\eta,\delta)|  \hspace{-1mm} & \hspace{-1mm}  \le \hspace{-1mm} & \hspace{-1mm}  C  \  |\delta|^\theta  { |\eta|^{-\kappa}  e^{-{b \over 2} |\eta|}  \over |\eta-\delta|^{\theta-c-\kappa}}    \| \psi \|_{v_\nu,p} \| \omega_\psi\|_{v_\nu,p} \left| 1-\beta^{\theta-c-\kappa+1} \right| |\eta-\delta|^{\theta-c-\kappa+1} \\
\nonumber  \hspace{-1mm} & \hspace{-1mm}  \le \hspace{-1mm} & \hspace{-1mm}  C   (1-\beta) |\delta|^\theta   |\eta|^{-\kappa} |\eta-\delta|   e^{-{b \over 2} |\eta|}     \| \psi \|_{v_\nu,p} \| \omega_\psi\|_{v_\nu,p} \\
\label{N4termssmall}  \hspace{-1mm} & \hspace{-1mm}  \le \hspace{-1mm} & \hspace{-1mm}  C K k  (1-\beta) |\delta|^\theta   |\eta|^{-\kappa} (|\eta|+|\delta|)   e^{-{b \over 2} |\eta|}.
\end{eqnarray}
For $|\eta|>1$ we revisit the calculation from start, and use the assumption $(\ref{omegabound})$ to compute the convolution:
\begin{eqnarray}
  \nonumber  (|\psi| \star \omega_\psi)(\eta t )   \hspace{-1.2mm} &  \hspace{-1.2mm} \le  \hspace{-1.2mm}&   \hspace{-1.2mm}  \int \displaylimits_{|x|\le 1}  \hspace{-1mm} |\psi(\eta t -x)|  \omega_\psi(x )    \  d x  +   \int \displaylimits_{|x|>1}   \hspace{-1mm} |\psi(\eta t - x)| \  \omega_\psi(x)   \  d x \\
  \nonumber   \hspace{-1.2mm} &  \hspace{-1.2mm} \le  \hspace{-1.2mm}&   \hspace{-1.2mm}   k  \hspace{-2mm} \int \displaylimits_{|x|\le 1}  \hspace{-1mm} |\eta t-x|^\sigma e^{-b |\eta t -x|} \omega_\psi(x )   \   d x+  A k \hspace{-2mm}  \int \displaylimits_{|x|>1}   \hspace{-2mm} |\eta t -x|^\sigma e^{-b |\eta t -x|}  e^{-b |x|} |x|^\tau  \  d x  \\
  \nonumber   \hspace{-1.2mm} &  \hspace{-1.2mm} \le  \hspace{-1.2mm}&   \hspace{-1.2mm}   k  \left( \hspace{2mm}  \int \displaylimits_{|x|\le 1}  \hspace{-1mm} |\eta t-x|^{q \sigma} e^{-q b |\eta t -x|} |x|^{- q \nu} e^{-q {b \over 2} |x|} d x \right)^{1 \over q} \|\omega_\psi \|_{v_\nu,p} +  \\
  \nonumber  \hspace{-1.2mm} &  \hspace{-1.2mm} \phantom{\le}  \hspace{-1.2mm}&   \hspace{-1.2mm} \hspace{48.2mm} +A k \hspace{-4mm}  \int \displaylimits_{1<|x|<|\eta t|}   \hspace{-4mm} |\eta t-x|^\sigma e^{-b |\eta t -x|}  e^{-b |x|} |x|^\tau    d x +\\
  \nonumber  \hspace{-1.2mm} &  \hspace{-1.2mm} \phantom{\le}  \hspace{-1.2mm}&   \hspace{-1.2mm} \hspace{48.2mm} + A k \hspace{-2mm}  \int \displaylimits_{|x|\ge |\eta t|}   \hspace{-2mm} |\eta t-x|^\sigma e^{-b |\eta t -x|}  e^{-b |x|} |x|^\tau    d x  \\
 \nonumber   \hspace{-1.2mm} &  \hspace{-1.2mm} \le  \hspace{-1.2mm}&   \hspace{-1.2mm}   C K k   e^{-b |\eta| t}  \hspace{-1mm} \left( \hspace{1mm}  \int \displaylimits_{|x|\le 1}  \hspace{-2mm} |\eta t \hspace{-0.5mm} \shortminus \hspace{-0.5mm} x|^{q \sigma} |x|^{- q \nu}  d x \hspace{-1mm} \right)^{1 \over q} \hspace{-3mm}  +\hspace{-0.5mm}A k e^{-b |\eta t|}  \hspace{-6.0mm}  \int \displaylimits_{1<|x|<|\eta t|}   \hspace{-5.5mm} |\eta t\hspace{-0.5mm} - \hspace{-0.5mm} x|^\sigma |x|^\tau    d x +\\
 \label{N4bound1}  \hspace{-1.2mm} &  \hspace{-1.2mm} \phantom{\le}  \hspace{-1.2mm}&   \hspace{-1.2mm}  \phantom{ C K k   e^{-b |\eta| t}  \hspace{-1mm} \left( \hspace{1mm}  \int \displaylimits_{|x|\le 1}  \hspace{-2mm} |\eta t \hspace{-0.5mm} \shortminus \hspace{-0.5mm} x|^{q \sigma} |x|^{- q \nu}  d x \hspace{-1mm} \right)^{1 \over q} } \hspace{-3mm}   +\hspace{-0.5mm} A k \hspace{-4mm}  \int \displaylimits_{|x|\ge |\eta t|}   \hspace{-4mm} |\eta t \hspace{-0.7mm} \shortminus \hspace{-0.7mm} x|^\sigma e^{-b (2 |x|\shortminus |\eta t|)} |x|^\tau    d x.
\end{eqnarray}
Since $\sigma<0$ and $|\eta t|> 1$, the first integral is bounded by
\begin{equation}
\label{bound1}  \int \displaylimits_{|x|\le 1}  |\eta t - x|^{q \sigma} |x|^{- q \nu}  d x  \le C  \int \displaylimits_{z\le 1}  (1- r)^{q \sigma} r^{- q \nu+d-1}  d r,
\end{equation}
and the convergence condition for the first integral in $(\ref{N4bound1})$ is 
\begin{eqnarray}
  \label{qunudcond}  -q \nu +d & > & 0 \Leftrightarrow -\nu {p \over p-1}>-d \Leftrightarrow \kappa <d, \\
  \label{sigmacond} \sigma&>&-{p-1 \over p}.
\end{eqnarray}
The second integral, under the condition $\sigma>-1$, can be bounded by 
\begin{equation}
\label{bound2}  \int \displaylimits_{1<|x|<|\eta t|}  |\eta t - x|^\sigma |x|^\tau    d x \le C |\eta t|^{ \sigma+\tau+d}  \hspace{-4mm}  \int \displaylimits_{{1 \over |\eta t| } < |x|<1}   \hspace{-4mm}  |1 - z|^\sigma z^{\tau}    d z \le C |\eta t|^{\sigma+d+\tau}.
\end{equation}
The third  integral can be estimated as follows:
\begin{eqnarray}
  \nonumber  \int \displaylimits_{|x|\ge |\eta t|}  |\eta t - x|^\sigma e^{-b (2 |x|\shortminus |\eta t|)} |x|^\tau    d x   & \le& e^{|\eta t|}   \int \displaylimits_{r\ge |\eta t|}  (r-|\eta t|)^\sigma  r^{\tau+d-1}  e^{-2 b r}   d r \\
  \nonumber &\le& C |\eta t|^{ \tau +d +\sigma-1 \over 2}  W_{{\sigma+1-\tau-d \over 2},-{\tau+d+\sigma \over 2}}(2 b |\eta t|).
\end{eqnarray}
The asymptotic of the Whittaker function appearing in the last bound is
$$|\eta t|^{\sigma+1-\tau-d \over 2} e^{ -b |\eta t|},$$
therefore,
\begin{equation}
  \label{bound3}  \int \displaylimits_{|x|\ge |\eta t|}  |\eta t - x|^\sigma e^{-b (2 |x|\shortminus |\eta t|)} |x|^\tau    d x  \le C |\eta t|^{\sigma} e^{-b |\eta t|}.
\end{equation}
We collect bounds $(\ref{bound1})$, $(\ref{bound2})$ and $(\ref{bound3})$  to get 
\begin{equation}
  \nonumber  (|\psi| \star \omega_\psi)(\eta t )   \le     C K k   e^{-b |\eta| t} +  A k |\eta t|^{\sigma+d+\tau} e^{-b |\eta t|} + A k |\eta t|^\sigma e^{- b |\eta t|},
\end{equation}
therefore, taking into account that $\sigma+d+\tau>\sigma$, we get for $|\eta|>1$,
\begin{eqnarray}
  \nonumber | N_4(\eta,\delta)|  &  \le &  C |\delta|^\theta |\eta| \hspace{-1mm} \int \displaylimits_{1}^{1 \over \beta}  \hspace{-0.8mm}  e^{|\eta-\delta|^\gamma (1-t^\gamma) - b |\eta t|} \left( C K k +  A k |\eta t|^{\sigma+d+\tau} + A k |\eta t|^\sigma  \right) t^{\theta -c}     d  t \\
  \nonumber   & \le &  C K k (1 \hspace{-0.6mm}  - \hspace{-0.6mm}  \beta)  |\eta|  e^{- b |\eta|} |\delta|^\theta \hspace{-0.5mm}  +  \hspace{-0.5mm}  C A k e^{- b |\eta|+|\eta-\delta|^\gamma} |\delta|^\theta {|\eta|^{\sigma+d+\tau} \over  |\eta-\delta|^{\gamma} } \hspace{-1mm}  \int \displaylimits_{|\eta-\delta|}^{|\eta-\delta| \over \beta} \hspace{-2mm}  e^{-y^\gamma }  d y^\gamma  \\
  \nonumber  & \le &  C K k (1 \hspace{-0.6mm}  -  \hspace{-0.6mm} \beta)  |\eta|  e^{- b |\eta|} |\delta|^\theta \hspace{-0.5mm}  +  \hspace{-0.5mm}  C A k e^{- b |\eta|} |\delta|^\theta {|\eta|^{\sigma+d+\tau} \over  |\eta-\delta|^{\gamma} } \left(1 \hspace{-0.6mm}  - \hspace{-0.6mm} e^{-{1 -\beta^\gamma \over \beta^\gamma} |\eta-\delta|^\gamma} \right).
\end{eqnarray}
Since for $|\eta|>1$, $|\eta-\delta|$ and $|\eta|$ are comparable, we can simplify the above as
\begin{equation}
  \label{N4termslarge} | N_4(\eta,\delta)| \le  C k \left( K (1 \hspace{-0.6mm}  -  \hspace{-0.6mm} \beta)  |\eta| \hspace{-0.5mm}  +  \hspace{-0.5mm}  A |\eta|^{\sigma+d+\tau-\gamma} \left(1 \hspace{-0.6mm}  - \hspace{-0.6mm} e^{-{1 -\beta^\gamma \over \beta^\gamma} |\eta-\delta|^\gamma} \right) \right) e^{-b|\eta| } |\delta|^\theta.
\end{equation}
 Bounds $(\ref{N4termssmall})$ and $(\ref{N4termslarge})$ imply that the condition $|N_4| \in L^p_{v_\nu}(\RR^d,\RR_+)$ is
\begin{equation}
  \label{condd3} (\nu-\kappa) p >-d.
\end{equation}
Together with
\begin{equation}
  \label{condd4} \theta-c-\kappa+1 \not \in \ZZ_-.
 \end{equation}
Under conditions $(\ref{condd2})$ and $(\ref{condd3})$,
\begin{equation}
 \label{N4} \|N_4 \|_{v_\nu,p} \le   C (K+A)  (1-\beta )  |\delta|^\theta.
\end{equation}

\bigskip

\noindent \textit{\textbf{3)  A bound on ${\bm A}$}. } We collect estimates  $(\ref{Lterms})$, $(\ref{N1terms})$, $(\ref{N2terms})$, $(\ref{N3terms})$ and $(\ref{N4termslarge})$, and get the following condition for $\omega_{\cR_\beta[\psi]}(\eta) \le A |\eta|^\tau e^{-b |\eta|}$:
\begin{align}
\nonumber  A |\eta|^\tau   \ge  {A \over \beta^{\tau+\theta-c}} e^{ -{1-\beta^\gamma \over \beta^\gamma } |\eta-\delta|^\gamma -b {1-\beta \over \beta } |\eta| } |\eta|^\tau  \hspace{-0.7mm}  &+  \hspace{-0.4mm}  C k (1 \hspace{-0.5mm} - \hspace{-0.5mm}  \beta) \hspace{-0.7mm} \left(|\eta|^{\sigma+\gamma-1} \delta_0^{1-\theta} \hspace{-0.7mm} + \right. \\
\nonumber & + \hspace{-0.4mm}   k (1 \hspace{-0.5mm}  - \hspace{-0.5mm}  \beta)^{1 \over \gamma} |\eta|^{2 \sigma+d+1} \delta_0^{1-\theta} \hspace{-0.4mm} +\hspace{-0.4mm}  k |\eta|^{2 \sigma+d} \delta_0^{1-\theta} \hspace{-0.4mm}  + \\
\nonumber  &+ \left.\hspace{-0.8mm}  k |\eta|^{2 \sigma+d-1} \delta_0^{1-\theta} \hspace{-0.4mm} + \hspace{-0.5mm}  k |\eta|^{\sigma} \delta_0^{d+\sigma-\theta} \hspace{-0.4mm}   +  \hspace{-0.6mm}   K  |\eta| \hspace{-0.4mm} \right)  \\
 &+  C A k |\eta|^{\sigma+d+\tau-\gamma} \left(1-e^{-{1 -\beta^\gamma \over \beta^\gamma}|\eta-\delta|^\gamma  }   \right).  \label{Acond1}
\end{align}
The following conditions guarantee that $|\eta|^\tau$ dominates all powers of $|\eta|$ in the right hand side:
\begin{equation}
  \nonumber \tau \ge \sigma+\gamma-1, \hspace{5mm} \tau \ge 2 \sigma+d+1,  \hspace{5mm} \tau \ge \sigma, \hspace{5mm} \sigma+d -\gamma \le 0.
\end{equation}
The first three of these conditions can be satisfied by a choice of $\tau$. The last one is essential:
\begin{equation}
\label{gammadsigma} \gamma \ge d+\sigma.
\end{equation}
Under $(\ref{gammadsigma})$, a sufficient condition for $(\ref{Acond1})$ is
\begin{equation}
  \label{Acond2} A \hspace{-0.7mm} \left(1 \hspace{-0.5mm}  -  \hspace{-0.5mm} \beta^{c-\theta-\tau} e^{\shortminus{1\hspace{-0.2mm}  - \hspace{-0.2mm} \beta^\gamma \over \beta^\gamma } (1 \shortminus |\delta|)^\gamma \shortminus b  {1\hspace{-0.2mm} -\hspace{-0.2mm}  \beta \over \beta  }}  \hspace{-0.7mm}  -   \hspace{-0.5mm}C k \hspace{-0.2mm} \left(\hspace{-0.2mm} 1 \hspace{-0.6mm}  -  \hspace{-0.5mm}  e^{\shortminus {1 \shortminus \beta^\gamma \over \beta^\gamma}(1\shortminus |\delta|)^\gamma  }  \hspace{-0.5mm}  \right) \hspace{-0.5mm} \right) \hspace{-0.7mm} \ge  \hspace{-0.4mm}  C k \hspace{-0.1mm}  (1  \hspace{-0.6mm} -  \hspace{-0.6mm} \beta) \hspace{-0.5mm} \left( \delta_0^{1-\theta}  \hspace{-0.8mm} +  \hspace{-0.5mm}  K  \hspace{-0.4mm} \right).
\end{equation}
As $\beta \rightarrow 1$, the condition $(\ref{Acond2})$ is asymptotic to
\begin{equation}
  \label{Acond3} A \left(\gamma (1-|\delta|)^\gamma (1-C k)+b+c -\theta -\tau\right) \ge   C k \left(\delta_0^{1-\theta} +   K \hspace{-0.4mm} \right).
\end{equation}
A necessary condition for the existence of such $A$ is
\begin{equation}
\label{condonconstants} \gamma(1-\delta_0)^\gamma (1- C k)+b+\gamma>d+1+\theta +\tau.
\end{equation}

\bigskip

\noindent \textit{\textbf{4)  A bound on ${\bm K}$}. } We can now combine $(\ref{Lestimate2})$, $(\ref{N1})$, $(\ref{N2})$, $(\ref{N3})$ and $(\ref{N4})$ to get the following condition of invariance of the set:
\begin{align}
  \nonumber K |\delta|^\theta \ge K  \left( \beta^{c-\theta+\nu+{d \over p}}+ C k(1 -  \beta) \right)  |\delta|^\theta &+ Ck A (1-\beta) |\delta|^\theta+ C (1-\beta) k  |\delta|^{\min\{1,\gamma\}} + \\
 \nonumber  & + \hspace{-0.5mm} C(1-\beta) k^2 \left( (1 \hspace{-0.5mm} - \hspace{-0.5mm}  \beta)^{1 \over \gamma}|\delta| \hspace{-0.5mm}  + \hspace{-0.5mm} |\delta|^{2 \sigma +d +\nu+{d \over p}}   +  \right. \\
  \label{Kcond1}  &+\left.\hspace{-0.5mm}  |\delta|^{\sigma+d}  \hspace{-0.5mm}   +  \hspace{-0.5mm}  |\delta|^{\nu+1+{d \over p}  }  \hspace{-0.5mm}  + \hspace{-0.5mm}  |\delta| \right).
\end{align}
Choosing $\theta$ as the minimal of the exponents of $|\delta|$:
\begin{equation}
\label{def_theta} \theta \le \min\left\{1,\gamma, \sigma+d,2 \sigma +d +\nu+1+ {d \over p}, \nu +1 +{d \over p}   \right\},  
\end{equation}
we get that $(\ref{Kcond1})$ is implied by
\begin{align}
  \nonumber  K \left((c-\theta+\nu+{d \over p} -C k   \right) \ge C k A &+ C k \delta_0^{\gamma-\theta} + \\
  \label{Kcond2}&+ C k^2 \left(\delta_0^{1-\theta}+\delta_0^{\sigma+d-\theta}+\delta_0^{2 \sigma+d+\nu+{d \over p}-\theta}+\delta_0^{\nu+{d \over p} -\theta }   \right).
\end{align}
For large $A$ and $K$, the conditions $(\ref{Acond3})$ and $(\ref{Kcond2})$ are asymptotic to
\begin{eqnarray}
  \nonumber A \left(\gamma (1-\delta_0)^\gamma (1-C k)+b+c -\theta -\tau\right) &\ge &   C k K, \\
  \nonumber  \label{Kcond3} K \left(c-\theta+\nu+{d \over p} -C k   \right) &\ge& C k A,
\end{eqnarray}
where the two constants $C$ are not necessarily the same. A sufficient condition for  existence of a pair of constants $A$ and $K$ is
\begin{equation}
{ \gamma (1-\delta_0)^\gamma (1-C k)+b+c -\theta -\tau  \over C k} \ge {C k \over  c- \theta +\nu +{d \over p} -C k }
\end{equation}
which can be guaranteed by a choice of $k$ and $b$.

\bigskip

 \noindent \textit{\textbf{5)  Case $\bm{\beta \delta_0 \le  |\delta|<  \delta_0}$. } }
 We notice that since for every $\psi \in \cM$,
 $$\|\psi \|_{v_\nu,p} \le B$$
 for some $B$, we get that for $\beta \delta_0 \le  |\delta|<  \delta_0$,
 $$\|\psi - \tau_\delta \psi \|_{v_\nu,p} \le 2 B.$$
Now, take $K$ as the maximum of $K$ from the previous part, and $2 B /\beta^\theta \delta_0^\theta$.

\medskip

The fact that the subset $\cN$ is non-empty follows from Lemma $\ref{Nnonempty}$.
\end{proof}

\medskip

\begin{remark}
  The modulus of continuity of $\psi \in \cN$ over the whole of $\RR^d$ can be described as follows:
  \begin{equation}\label{modcont}
  |\psi(\eta)-\tau_\delta \psi(\eta)| \le \lambda(\eta,\delta):=\left\{\omega_\psi(\eta) |\delta|^\theta, \hspace{3mm} |\delta| \le \delta_0, \atop {\rm const}, \hspace{7.5mm} |\delta|>\delta_0.   \right.  
  \end{equation}
  
\end{remark}

\subsection{A-priori bounds}\label{apriori_subsection}

We would like to collect $(\ref{MinLp})$ and  hypotheses  $(\ref{H9})-(\ref{H11})$ again.

 \begin{equation}
     \label{H12} \tag{H12}
     \begin{cases}
       &  \hspace{-3mm} ({\bf 1}) \ p>2, \hspace{34.4mm} ({\bf 2})  \ \sigma \in \left(-{p-1 \over p},0 \right) \setminus \ZZ/2, \\
        &   \hspace{-3mm}  ({\bf 3})  \ (\sigma + \nu) p > -d, \hspace{18.0mm} ({\bf4}) \ d {p-2 \over 2 p} < \nu < d {p-2 \over p}, \\
        &  \hspace{-3mm} ({\bf 5}) \  \nu+{d \over p}>\theta -c>\kappa-1, \hspace{4.7mm} ({\bf 6}) \ -d < (\nu - \kappa) p \\
        &  \hspace{-3mm} ({\bf 7})   \ (2 \sigma \hspace{-0.5mm} + \hspace{-0.5mm}  d \hspace{-0.5mm}  +  \hspace{-0.5mm} \nu \hspace{-0.5mm}  -  \hspace{-0.5mm} 1) p > -d,  \hspace{4.6mm} ({\bf 8})  \ -2 d -1 < 1-\gamma \ {\rm or}  \ \epsilon+{ 1-d \over 2} < 1-\gamma  \\
       &  \hspace{-3mm} ({\bf 9})   \ c \ge \sigma \ge -d,  \hspace{22.9mm} ({\bf 10})   \ \gamma > d+\sigma, \\
       &  \hspace{-3mm} ({\bf 11})   \ \kappa < d.
      \end{cases}
   \end{equation}

We have proved the following Theorem.
\begin{napthm} \label{naptheorem}
Let $d \ge 2$, $\gamma$, $p$, $\nu$, $\sigma$ an  $\theta$ satisfy the hypothesis $(\ref{H12})$. Set
\begin{equation}
  v_\nu(x)=e^{-{b \over 2}|x|} |x|^\nu.
\end{equation}    
Then for every $\beta_0 \in (0,1)$ and $\delta_0>0$ there are constants $b>0$,$K>0$, $ \theta>0$, $B >0$  and a convex set $\cN_\mu^\alpha$ of azimuthal vector fields $\psi(\eta)=\psi_{d-1}(\eta) e_{d-1}(\eta)$ with  $\psi_{d-1} \in L^p_{v_\nu}(\RR^d,\RR)$  satisfying for all $\beta \in (\beta_0,1)$:
\begin{itemize}
\item[$1)$] $\cN_\mu^\alpha$ is $\cR_\beta$-invariant;
\item[$2)$] $\| \psi_{d-1} \|_{v_\nu,p} \le B$ for all $\psi_{d-1} \in \cN_\mu^\alpha$;
\item[$3)$] $\psi_{d-1} \in \cN_\mu^\alpha$ is H\"older-continuous:
  $$|\psi_{d-1}(\eta)-\psi_{d-1}(\eta-\delta)| \le \omega_\psi(\eta) |\delta|^\theta,$$
for all $|\delta|<\delta_0$,  with $\omega_\psi \in L^p_{v_\nu}$ satisfying  $\| \omega_\psi \|_{v_\nu,p} \le K$;
\item[$4)$] $\psi_{d-1} \in \cN_\mu^\alpha$ satisfies
  $$\int \displaylimits_{|\eta|>r}| \psi_{d-1}(\eta)|^p v_\nu(\eta)^p \ d \eta  \rightarrow 0$$
uniformly in $\cN^\alpha_\mu$:
\item[$5)$] the $L^p_{v_\nu}$-closure of $\cN_\mu^\alpha$ is compact;
\item[$6)$] $0 \not \in  \cN_\mu^\alpha$.
\end{itemize}
\end{napthm}

We refer to this result as {\it nontrivial bounds} because, unlike in Theorem $\ref{aptheorem}$,  $\cN_\mu^\alpha$ does not contain the trivial fixed point.

\subsection{Hypothesis $(\ref{H12})$}

We proceed to rewrite the conditions $(\ref{H12})$.

\medskip

Condition  $(\ref{H12}) ({\bf 6})$ is equivalent to
\begin{equation}
\label{H126a} \tag{H12) ({\bf 6a}}  \kappa <d.
\end{equation}
which is identical to  $(\ref{H12}) ({\bf 11})$. Since according to  $(\ref{H12}) ({\bf 9})$, $c \ge \sigma \implies \gamma\ge \sigma+d+1$,  by $(\ref{H12}) ({\bf 7})$  $2+(2 \sigma+d+\nu-1+d/p)>2$ and by $(\ref{H12}) ({\bf 4})$ $1+\nu+d/p>1$, we have that $\theta$, defined in $(\ref{def_theta})$ is given by  $ \theta \le \min\{1,\sigma+d \}$. Furthermore, since $\sigma > -(p-1)/p>-1$ for any $d \ge 2$, $\sigma+d>1$ and 
  \begin{equation}
    \label{theta_new}  \theta \le 1.
  \end{equation}
  The condition $(\ref{H12}) ({\bf 9})$, $\gamma-1 \ge \sigma+d$, is stronger than  $(\ref{H12}) ({\bf 10})$.
  
Next, condition  $(\ref{H12}) ({\bf 3})$ together with $(\ref{kappap})$  is equivalent to
\begin{equation}
\label{H123a} \tag{H12) ({\bf 3a}} \kappa > -2 \sigma -d,
\end{equation}
while condition $(\ref{H12})  ({\bf 7})$ is equivalent to
\begin{equation}
\label{H127a} \tag{H12) ({\bf 7a}}  \kappa > -4 \sigma -3 d+2,
\end{equation}
and $(\ref{H127a})$ is weaker than  $(\ref{H123a})$  for all $d \ge 2$ and  $\sigma \ge -1$.

At the same time condition  $(\ref{H12}) ({\bf 5})$ is equivalent to
\begin{equation}
  \label{kappa2}  {\kappa +d \over 2} > \theta-c > \kappa-1,
\end{equation}  
which simplifies to
\begin{equation}
  \label{H125a} \tag{H12) ({\bf 5a}}  {d - \kappa \over 2}+\theta+1< \gamma < 2+d+\theta-\kappa.
\end{equation}
Condition $(\ref{H12}) ({\bf 4})$ is equivalent to
\begin{equation}
 \label{H124a} \tag{H12) ({\bf 4a}} 0 <\kappa < d {p-2 \over p}.
\end{equation}
The upper bound here is stronger than the upper bound in $(\ref{H126a})$.

Consider the inequalities $(\ref{H12}) ({\bf 8})$. The second inequality has an empty solution for all $\gamma > d$ and $d \ge 2$. The first one becomes
\begin{equation}
\label{H128a} \tag{H12) ({\bf 8a}} \gamma <  2d+2.
\end{equation}
At the same time, $c \ge \sigma$  is equivalent to $\gamma \ge \sigma+d+1$. Therefore the bound on $\gamma$ becomes
$$\max\left\{ {d -\kappa \over 2}+\theta+1, \sigma+d+1  \right\} < \gamma < \min\{2 d+2, 2+d +\theta -\kappa\}.$$
Additionally, for all $d \ge 2$ and $\theta \le 1$, then
$$2 d+2  > 2 +d +\theta -k.$$
The lower bound similarly simplifies: for all $d \ge 2$
$$\sigma +d +a  \ge {d -\kappa \over 2} +\theta+4 \Leftrightarrow d \ge 2 \theta - \kappa - 2 \sigma$$

The only remaining bound for $\kappa$ becomes
$$\max\{-d-2 \sigma, 0 \} <\kappa < d{p-2 \over p}.$$
The lower bound in this expression can be simplified further, taking in to account that $\sigma>-(p-1)/p>-1$, and for any $d \ge 2$, $\max\{-d-2 \sigma, 0 \}=0$.

To summarize,
  \begin{equation}
    \label{H13} \tag{H13}
    \begin{cases}
   &  ({\bf 1}) \ \max\left\{ {d -\kappa \over 2}+\theta+1, \sigma+d+1  \right\} < \gamma < 2 d+2,
   \\
   &    ({\bf 2})  \   \sigma  \in \left(-{p-1 \over p},0 \right) \setminus  \ZZ/2, \hspace{28.0mm}  ({\bf 3})  \ p > 2,  \\
   & ({\bf 4})   \  0  <  \kappa  <   d{p -    2 \over p}, \hspace{41.3mm}  ({\bf 5})\  0 < \theta \le 1.  
\end{cases}
  \end{equation}  
In particular we have
  \begin{lemma}\label{gammabounds}
    Conditions $(\ref{H13})$  have a non-empty solution for
    \begin{equation}
    \label{gammaall} \quad d < \gamma <   2 d+2.
    \end{equation}
  \end{lemma}

\section{Embeddings: Non-trivial {\it a-priori} bounds in $L^t$} \label{properties}

We will now investigate if a fixed point of $\cR_\beta$ in $\cN_\mu^\alpha$  belongs to $L^l(\RR^d,\RR)$ for  $1 \le l \le 2$. 

\subsection{Embedding of $L^p_{v_\nu}$ into  $L^t$.} The space  $L^p_{v_\nu}(\RR^d,\RR)$  is not well suited for a prove of existence of the inverse Fourier transform and its quadratic integrability: after all, ideally, we would like to use Plancherel's theorem in $L^2$ to demonstrate boundedness of the energy for times $0\le t < T$ and its blowup for time $T$.
However, it is possible to show that there exists a $t_0$ such that  $\cN_\mu^\alpha \subset L^t(\RR^d,\RR)$ for all $1 \le t \le t_0$.

We start by considering the following bound. Let $t>1>\tilde t>  0$, and set $r=1/\tilde t$, and $r'=r/(r-1)$, then
  \begin{eqnarray}
    \nonumber  \int_{ \RR^d } |\psi(\eta)|^t d \eta & \le &  \int_{ \RR^d } \left( |\psi(\eta)|^{\tilde t} |\eta|^{\alpha \tilde t} \right) \left( |\psi(\eta)|^{t-\tilde t} |\eta|^{-\alpha \tilde t}  e^{{b \over 2} l  |\eta|}  \right) \ d \eta  \\
    \nonumber   & \le &  \left( \int_{ \RR^d } \left( |\psi(\eta)|^{\tilde t} |\eta|^{\alpha \tilde t} \right)^{r} \ d \eta \right)^{1 \over r}   \left( \int_{ \RR^d }  \left(  |\psi(\eta)|^{t-\tilde t} |\eta|^{-\alpha \tilde t}  e^{{b \over 2} l  |\eta|}  \right)^{r'} \ d \eta \right)^{1 \over r'} \\
    \nonumber   & \le &  \left( \int_{ \RR^d } |\psi(\eta)| |\eta|^{\alpha} \ d \eta \right)^{1 \over r}   \left( \int_{ \RR^d }  \left(  |\psi(\eta)|^{t-\tilde t} |\eta|^{-\alpha \tilde t}  e^{{b \over 2} l  |\eta|}  \right)^{r'} \ d \eta \right)^{1 \over r'}.
  \end{eqnarray}
  We set
  $$p=(t- \tilde t) r' \implies \tilde t={p-t \over p-1},$$
  and
\begin{equation}
 \nonumber -\alpha \tilde t r'\ge \nu p \implies -\alpha { p- t  \over t-1 } \ge \nu p, \quad l r'=p.
  \end{equation}
The condition for $\alpha$ can be rewritten as 
\begin{equation}
   \label{alphacond}   \kappa  < -2 \alpha { p -t \over p(t-1)} - d {p-2 \over p}.  
\end{equation}
Under $(\ref{alphacond})$ and with the choice of $l$ such that $l r'=p$,
\begin{eqnarray}
  \nonumber  \int_{ \RR^d } |\psi(\eta)|^t d \eta   & \le &  \left( \int_{ \RR^d } |\psi(\eta)| |\eta|^{\alpha} \ d \eta \right)^{1 \over r}   \left( \int_{ \RR^d }    |\psi(\eta)|^{p} |\eta|^{\tilde \nu p}  e^{{b \over 2} p  |\eta|} \ d \eta \right)^{1 \over r'},
\end{eqnarray}
for some $\tilde \nu \ge \nu$, both integrals being bounded for a function in $\cN_\mu^\alpha$.

Recall that  by $(\ref{H10})$
$$-\alpha > \max\{ d-\epsilon+\max\{-d,-3 \}, \gamma-1\}$$
for any fixed $\epsilon>0$. At the same time, by Lemma $\ref{gammabounds}$,
$$\gamma-1>d-1> d-\epsilon+\max\{-d,-3 \}$$
for any $\epsilon>0$ and all $d \ge 2$. We have, therefore, that
$$-\alpha> \gamma-1.$$
Additionally, by  $(\ref{H10})$,
$$-\alpha \in \left(-1-2 \epsilon, {d -1 \over 2} -\epsilon  \right] \cup \left({d \over 2} -\epsilon, 2 d+1  \right).$$
  Now, notice that since $\gamma > d$ according to Lemma $\ref{gammabounds}$, then
  $$\gamma-1>d-1>{d \over 2 } -\epsilon$$
  for all $d \ge 2$. Therefore, we get that
  $$-\alpha \in \left( \gamma-1, 2 d +1 \right).$$

We remark, that the derivative of the right hand side of $(\ref{alphacond})$ with respect to $t$ is equal to
$$2 \alpha { p-1 \over p (t-1)^2},$$
and is negative for all $1\le t \le 2$. Thus, the strongest condition $(\ref{alphacond})$ is realized by $t=2$:
\begin{equation}
  \label{t2}  \kappa  < -(2 \alpha+d) {p-2 \over p}.
\end{equation}
We conclude that a sufficient condition for $(\ref{t2})$ is
$$\alpha \le - d.$$
Indeed, under this condition, $(\ref{t2})$ is weaker than  $(\ref{H13}){\bf (4)}$.

We have, therefore, the following result:
\begin{napthmLt} \label{L1andL2}
  For any $d \ge 2$ and  $d<\gamma < 2 d+2$ there exists $0<\beta_0<1$ and a  set $\cN_\mu^\alpha \subset L^t(\RR^d,\RR)$, $1 \le t \le 2$, which is convex, $L^t$-precompact and $\cR_\beta$-invariant for any $\beta_0<\beta <1$. 
\end{napthmLt}

\section{Existence of blow ups in $L^m$, $m \ge 2$} \label{blowups}
We will now identify a function  in the set $\overline{\cN^\alpha_\mu}$, where the closure is taken in $L^p_{u_\nu}$, such that it is fixed by all operators $\cR_\beta$ for all $\beta_0<\beta <1$.

Consider any family of fixed points $\{ \psi_{\beta} \}_{\beta \in (\beta_0,1)}$.  By compactness of $\overline{\cN^\alpha_\mu}$, there is a converging subsequence $\{ \psi_{\beta_i} \}_{i=0}^\infty$ as $\beta_i \rightarrow 1$, to a function $\psi_*$. The function $\psi_*$ is in  $\overline{\cN^\alpha_\mu}$, and, therefore, is non-trivial.

Next, for any $\beta \in (\beta_0,1)$, there exist a subsequence $\{\beta_{i_k} \}_{k=0}^\infty $ and a diverging sequence of integers $n_k$, such that $\beta_{i_k}^{n_k} \rightarrow \beta$. Indeed, for any $\epsilon>0$, a sufficient condition for $\beta_{i_k}^{n_k} \in (\beta-\epsilon, \beta +\epsilon)$ is
$$n_k \in \left( {\ln (\beta+\epsilon )  \over \ln \beta_{i_k} },   {\ln (\beta-\epsilon )  \over \ln \beta_{i_k} } \right).$$
The length of this interval is of the order $-\epsilon / \beta \ln \beta_{i_k}$. Therefore, for any $\epsilon$ there exists $k \in \NN$, such that this length is larger than $1$, and the interval contains some positive integer $n_k$.

To conclude that $\psi_*$ is fixed by all $\cR_\beta$ we require the following
\begin{lemma}
A fixed point $\psi_\beta$ of the operator $\cR_\beta$ is also fixed by  $\cR_{\beta^n}$ for all $n \in \NN$.
\end{lemma}
\begin{proof}  We have at the base of induction.
  \begin{eqnarray}
    \nonumber     e^{-\beta^{2 \gamma} |\eta|^\gamma} \psi_\beta(\beta^{2} \eta)&=& \beta^c e^{-\beta^{\gamma} \eta^\gamma} \psi_\beta(\beta \eta)+ \\
    \nonumber && + i \gamma \beta^{c} | \beta \eta|^{c-1} \int_{\beta^{2} |\eta|}^{\beta |\eta|} { e^{-r^\gamma} \over r^c} \int \displaylimits_{\RR^d} \beta \eta \cdot \psi_\beta\left({ \eta \over |\eta|} r-x  \right) P_\eta \psi_\beta(x) \ d x \ d r \\
    \nonumber     &=&  \beta^{2c} e^{-|\eta|^\gamma} \psi_\beta(\eta) + \\
    \nonumber && + i \gamma \beta^{2 c} | \eta|^{c-1} \int_{\beta |\eta|}^{ |\eta|} { e^{-r^\gamma} \over r^c} \int \displaylimits_{\RR^d} \eta \cdot \psi_\beta\left({ \eta \over |\eta|} r-x  \right) P_\eta \psi_\beta(x) \ d x \ d r \\
    \nonumber &&  + i \gamma \beta^{2 c} | \eta|^{c-1} \int_{\beta^{2} |\eta|}^{\beta |\eta|} { e^{-r^\gamma} \over r^c} \int \displaylimits_{\RR^d} \eta \cdot \psi_\beta\left({ \eta \over |\eta|} r-x  \right) P_\eta \psi_\beta(x) \ d x \ d r 
 \end{eqnarray}
  \begin{eqnarray}
    \nonumber  \phantom{   e^{-\beta^{2 \gamma} |\eta|^\gamma} \psi_\beta(\beta^{2} \eta)}   &=&  \beta^{2c} e^{-|\eta|^\gamma} \psi_\beta(\eta) + \\
    \nonumber &&  + i \gamma \beta^{2 c} | \eta|^{c-1} \int_{\beta2 |\eta|}^{ |\eta|} { e^{-r^\gamma} \over r^c} \int \displaylimits_{\RR^d} \eta \cdot \psi_\beta\left({ \eta \over |\eta|} r-x  \right) P_\eta \psi_\beta(x) \ d x \ d r.
  \end{eqnarray}
  Assume the result for $n=k$. Then
  \begin{eqnarray}
    \nonumber     e^{-\beta^{ \gamma (k+1)} |\eta|^\gamma} \psi_\beta(\beta^{k+1} \eta) \hspace{-2mm} & =& \hspace{-2mm} \beta^c e^{-\beta^{\gamma k} |\eta|^\gamma} \psi_\beta(\beta^{k} \eta) +  \\
    \nonumber && + i \gamma \beta^{c} |\beta^k \eta|^{c-1}  \hspace{-3mm}  \int  \displaylimits_{\beta |\beta^k \eta|}^{|\beta^k \eta|}  \hspace{-3mm}  { e^{-r^\gamma} \over r^c} \int \displaylimits_{\RR^d} \beta^k \eta \cdot \psi_\beta\left({ \eta \over |\eta|} r-x  \right) P_\eta \psi_\beta(x) \ d x \ d r   \\
    \nonumber     & =& \hspace{-2mm} \beta^{c(k+1)} e^{-|\eta|^\gamma} \psi_\beta( \eta) +  \\
    \nonumber &&  + i \gamma \beta^c \beta^{ck} | \eta|^{c-1}  \hspace{-3mm}  \int  \displaylimits_{\beta^k |\eta|}^{|\eta|}  \hspace{-3mm}  { e^{-r^\gamma} \over r^c} \int \displaylimits_{\RR^d} \eta \cdot \psi_\beta\left({ \eta \over |\eta|} r-x  \right) P_\eta \psi_\beta(x) \ d x \ d r\\
    \nonumber && + i \gamma \beta^{c(k+1)} | \eta|^{c-1}  \hspace{-5mm} \int  \displaylimits_{\beta^{k+1} |\eta|}^{\beta^k |\eta|}  \hspace{-3mm}  { e^{-r^\gamma} \over r^c} \int \displaylimits_{\RR^d} \eta \cdot \psi_\beta\left({ \eta \over |\eta|} r-x  \right) P_\eta \psi_\beta(x) \ d x \ d r
  \end{eqnarray}
\begin{eqnarray}
    \nonumber     & =& \hspace{-2mm} \beta^{c(k+1)} e^{-|\eta|^\gamma} \psi_\beta( \eta) +  \\
    \nonumber &&  + i \gamma \beta^{c(k+1)} | \eta|^{c-1} \hspace{-3mm}  \int \displaylimits_{\beta^{k+1} |\eta|}^{|\eta|} \hspace{-3mm}   { e^{-r^\gamma} \over r^c} \int \displaylimits_{\RR^d} \eta \cdot \psi_\beta\left({ \eta \over |\eta|} r-x  \right) P_\eta \psi_\beta(x) \ d x \ d r.
  \end{eqnarray}
\end{proof}
Therefore, according to the preceding Lemma,
\begin{equation}
\label{Rbetaik} \cR_{\beta_{i_k}^{n_k}}[\psi_{\beta_{i_k}}]=\psi_{\beta_{i_k}}.
\end{equation}
We conclude that the right and the left hand sides $(\ref{Rbetaik})$ converge to $\psi_*$ and  $\cR_{\beta}[\psi_*]$, respectively.

We are now ready to finish the proof of Theorem B.

As before, we denote
$$\vartheta(\eta)=i \psi(\eta) e_{d-1}(\eta)$$
{\it Proof of Theorem B.}  Part  $\mathbf{(1)}$  is immediate. To prove property $(\ref{decay})$ we use $(\ref{modcont}$ in the following standard computation:
  \begin{equation}
    \nonumber   \left| \cF^{-1}[\psi_*](y) \right| =  {1 \over 2} \left| \int_{\RR^d} \left( \psi_*(\eta)  -\tau_{\pi y \over |y|^2} \psi_*(\eta) \right) e^{i y \eta} d \eta \right| \le C \ \int_{\RR^d} \omega_{\psi_*}(\eta) \lambda \left( \pi { y \over |y|^2}  \right) d \eta.
  \end{equation}
  Notice that $L_{v_\nu}^p$ embeds continuously into $L^1$: for any $\omega \in L_{v_\nu}^p$ and $1=1/p+1/q$, 
  \begin{equation}
    \nonumber  \| \omega \|_1  =  \int \displaylimits_{\RR^d} | \omega(\eta)| v_\nu(\eta)  v_\nu(\eta)^{-1} d \eta \le \left(   \hspace{2mm}  \int \displaylimits_{\RR^d} | \omega(\eta)|^p v_\nu(\eta)^p d \eta \right)^{1 \over p} \hspace{-2mm}  \left( \hspace{2mm}  \int \displaylimits_{\RR^d} |\eta|^{-\nu q} e^{-{b \over 2} q  |\eta|} d \eta \right)^{1 \over q}.
  \end{equation}
  The condition of convergence of the last integral is
 $$ -\nu p/(p-1)>-d \implies \kappa < d,$$ 
which is weaker than $(\ref{H13}) ({\bf 4})$.
    
We conclude that for  $|y|>\pi/\delta_0$,
  $$\left| \cF^{-1}[\psi_*](y) \right| < C |y|^{-\theta}.$$ 

  \vspace{2mm}

 \noindent \underline{\it Part $\mathbf{(3)}$.} Consider the integral
  \begin{equation}
  \nonumber   \int \displaylimits_{\RR^d} \psi_*(\eta) (1+|\eta|)^k \ d \eta
  \end{equation}
  for $k \in \NN$. For $1=1/q+1/p$,
  \begin{equation}
    \label{psiweight} \   \int \displaylimits_{\RR^d} | \psi_*(\eta)| (1+|\eta|)^k \ d \eta  \le   \left( \hspace{2mm} \int \displaylimits_{\RR^d} (1+|\eta|)^{k q} v_{\nu}(\eta)^{-q} \right)^{1 \over q} \| \psi_* \|_{v_\nu,p}.
  \end{equation}
 Similarly to the proof of the property $(\ref{decay})$, the first integral in the last line converges for any $k \in \NN$ under the condition  $- \nu p /(p-1) > -d$ which is weaker than $(\ref{H13}) ({\bf 4})$. Now apply the Riemann-Lebesgue Lemma to $\cF^{-1}[ (i \eta)^\alpha \psi_*]$, where $\alpha$ is a multi-index, taking into account $(\ref{psiweight})$.

\vspace{2mm}
  
\noindent \underline{\it Part $\mathbf{(2)}$.} Since $\psi_*$ is a fixed point of $\cR_\beta$ for all $\beta \in (0,(1-\beta_0^\gamma)T)$, the function $u_*$ satisfies the mild Navier-Stokes system for all times $t \in (0,(1-\beta_0^\gamma) T)$. Since $\psi_* \in L^t(\RR^d,\RR)$ for any $1 \le t \le 2$, we have that $\cF^{-1}[\psi_*] \in  L^m(\RR^d,\RR)$ for any $m \ge 2$, and, therefore, so is $u_*$ for all $0 \le t < (1-\beta_0^\gamma) T$.

We can now invoke the following existence and uniqueness result by Giga (see Theorems $1$ and $2$ in \cite{Giga}), which in turn are based on the works of Kato \cite{Kato} and Kato and Fujita \cite{KaFu}:

\begin{theorem} \underline{\it (Local uniqueness and existence in $L^m$.)} \label{Giga}
  Let
  $$m_0={d \over \gamma - \chi}, \quad  m \ge l \ge m_0$$
  and
  $$ \rho={d \over \gamma} \left({1 \over l} -{1 \over m}  \right), \quad \chi=1 +\gamma \rho.$$
 Let $u_0 \in \bPL^l$.

  \vspace{2mm}
  
\noindent \underline{Existence.}  Then there exists $T_0$ and a unique solution $u$ to the mild Navier-Stokes system with the initial condition $u_0$, such that
\begin{eqnarray}
  \label{trho}  t^\rho u \hspace{-1mm} & \hspace{-1mm} \in \hspace{-1mm} & \hspace{-1mm} BC\left([0,T_0): \bPL^m   \right), \quad {\rm for} \quad l \le m < \infty,\\
    \label{T0} if \hspace{3mm} l \hspace{-1mm} &  \hspace{-1mm} > \hspace{-1mm} &\hspace{-1mm} m_0, \hspace{3mm} then \hspace{3mm}   T_0 \asymp C \|u_0 \|_l^{- {\gamma \over \gamma-\chi -{d \over l}}},
\end{eqnarray}
with $C$ independent of $u_0$.
\vspace{2mm}

\noindent \underline{Uniqueness}. Suppose that either
$$0 <\rho < {1 \over 2}, \quad m>2$$ 
or
$$l>\max\{m_0,2\}.$$ 
Then the solutions with the property $(\ref{trho})$ are unique.
\end{theorem}

We use this theorem with $l=m>\max\{m_0,2\}$, which also implies $\rho=0$.

Since $u_*$ is bounded in $\bPL^m$  (by the corresponding norms of $\pmb{\psi}$)  independently of $\beta \in (\beta_0,1)$, the Giga blow up time $T_0$ is independent of $\beta$, in fact, a careful reading of the proof of Theorem $\ref{Giga}$ reveals that a sufficient condition for existence is 
$$\nonumber T_0  \le   C \|u_*(\cdot, 0) \|_m^{ - {\gamma \over \gamma-1 - {d \over m}}}$$
for a specific constant $C$. Taking into account the definition $(\ref{ubeta})$  of $u_*$, this sufficient condition becomes
\begin{equation}
\nonumber T_0 \le  C \left( \tau(0)^{1+{d \over m} -\gamma} \| \pmb{\psi} \|_{m'} \right)^{ - {\gamma \over \gamma-1 - {d \over m} }} =  C  T \| \pmb{\psi} \|_{m'}^{ - {\gamma \over \gamma-1 - {d \over m} }},
\end{equation}
where $\pmb{\psi}$ is the bound in the set $\cN$ and $m'$ is the H\"older conjugate of $m$. We can now choose $T$ sufficiently large, and $\beta_1$ sufficiently close to $1$ so that
$$(1-\beta_1^\gamma) T <T_0.$$
Then, the function $u_*(x,t)$ that satisfies the mild Navier-Stokes system  $(\ref{eq:integral:NSgamma})$ for  all $t \in (0,(1-\beta_1^\gamma)T)$,  has to coincide with the unique solution described in Theorem $\ref{Giga}$ for the initial data   $u_*(x,0)$.

To summarize, we have argued that $u_*$ is the unique solution of  $(\ref{eq:integral:NSgamma})$ in $\bPL^m$, $m>\max\{d/(\gamma-1),2\}$, for times $[0,(1-\beta_1^\gamma)T]$.

Next, we prove that $u_*$ is the unique solution corresponding to the initial data $u_*(\cdot,0)$ for times $[0,(1-\beta_1^{n \gamma}) T$ for all $n \in \NN$. We will do this by induction. We assume the result for $n=k$ and prove it for $n=k+1$, i.e. we solve the initial value problem $(\ref{mildNS})$ with the initial condition $u_*(\cdot, t_k)$, $t_k=(1-\beta_1^{k \gamma} T)$, on $[(1-\beta_1^{k \gamma} T),(1-\beta_1^{(k+1) \gamma}) T]$. Equivalently, we solve
\begin{equation}
\label{tkproblem1}v(y,t) = \ee^{-|y|^\gamma (t-t_k)} v_*(y,t_k) - i \int_{t_k}^t \int_{\RR^d} (y \cdot v(y-z,s)) P_y v(z,s)\ee^{-|y|^\gamma (t-s)} d z  \ d s\,,
\end{equation}
with
$$v_*(y,t_k) \hspace{-0.6mm} = \hspace{-0.6mm}  i \tau(t_k)^{d\hspace{-0.2mm} +\hspace{-0.2mm} 1 \hspace{-0.2mm} \shortminus \hspace{-0.2mm} \gamma} \psi_*(y \tau(t_k))\hspace{-0.6mm}  = \hspace{-0.6mm}  i \beta_1^{(d\hspace{-0.2mm}+\hspace{-0.2mm}1 \hspace{-0.2mm}  \shortminus \hspace{-0.2mm}  \gamma)k}  \tau(0)^{d\hspace{-0.2mm} +\hspace{-0.2mm} 1\hspace{-0.2mm} \shortminus \hspace{-0.2mm} \gamma} \psi_*(\beta_1^k \hspace{-0.5mm}  y \tau(0)) \hspace{-0.6mm}  = \hspace{-0.6mm}  \beta_1^{(d \hspace{-0.2mm} + \hspace{-0.2mm} 1 \hspace{-0.2mm}  \shortminus \hspace{-0.2mm} \gamma) k}  v_*(\beta_1^k \hspace{-0.5mm}  y,0)$$
on $[(1-\beta_1^{k \gamma} T),(1-\beta_1^{(k+1) \gamma}) T]$. After the time change $t=\beta_1^{k \gamma} t'+t_k$, and the domain variable change $y'=\beta_1^k y$, the initial value problem $(\ref{tkproblem1})$ becomes
\begin{eqnarray}
  \nonumber v\hspace{-3.2mm} & \hspace{-3.2mm} &\hspace{-3.2mm} \left({y'\over \beta_1^k} , \beta_1^{\gamma k} t'+t_k \right)  \hspace{-0.8mm} =  \hspace{-0.8mm}   \beta_1^{(d+1-\gamma) k}  \ee^{-|y'|^\gamma t'} v_*(y',0) - \\
  \label{tkproblem2}  && \hspace{4mm} - i  \beta_1^{(\gamma \hspace{-0.2mm} \shortminus \hspace{-0.2mm} 1 \hspace{-0.2mm} \shortminus \hspace{-0.2mm}d)k} \hspace{-1.2mm}  \int \displaylimits_{0}^{t'} \hspace{-0.8mm} \int \displaylimits_{\RR^d} \hspace{-1mm} (y' \cdot v\left({y' \hspace{-0.6mm}  \shortminus \hspace{-0.6mm}  z' \over \beta_1^k},\beta_1^{k \gamma} s'\hspace{-0.6mm}  + \hspace{-0.6mm}  t_k \right) ) P_y v\left({z' \over \beta_1^k},\beta_1^{k \gamma} s'\hspace{-0.6mm}  +\hspace{-0.6mm}  t_k \right)\ee^{\shortminus |y'|^\gamma (t'\shortminus s')} d z' \hspace{+0.5mm} d s'\hspace{-0.5mm} ,
\end{eqnarray}
on  $[0,(1-\beta_1^{ \gamma}) \beta_1^{k \gamma}T]$.  The initial value problem $(\ref{tkproblem2})$ for the function
$$\tilde v(y',t') :=  \beta_1^{-(d+1-\gamma) k}  v \left({y'\over \beta_1^k}, \beta_1^{\gamma k} t'+t_k \right)$$
coincides with the initial value problem  $(\ref{eq:integral:NSgamma})$ on the subset $[0,(1-\beta_1^{ \gamma}) \beta_1^{k \gamma}T]$ of $[0,(1-\beta_1^{ \gamma}) T]$ with the initial condition $v_*(y',0)$. We have already argued that it is uniquely solved on $[0,(1-\beta_1^{ \gamma}) \beta_1^{k \gamma}T]$  by  $\tilde v(y',t')= v_*(y',t')$.
  
We conclude that $u_*$ is the unique  $\bPL^m$-solution for times $[0,(1-\beta_1^{n \gamma}) T)$ for any natural $n$, therefore, it is the unique $\bPL^m$-solution in $[0,T)$.

When $d<\gamma<2 d+2$ and $d \ge 2$,  we have that $2>d/(\gamma-1)$, and $u_*$ is the unique solution in  $BC\left([0,T_0): \bPL^m   \right)$, for any $m>2$.

Additionally, we see that the norms
$$\int_{\RR^d} \left( 1+|\eta|^{2} \right)^{k \over 2} |\psi(\eta)|^2 d \eta$$
are bounded for any $k>0$ and any $\psi \in \overline{\cN^\alpha_\mu}$ (since these functions are exponentially  bounded). This implies that $u_* \in \bH^k$ for any positive $k$.

\vspace{2mm}
  
  \noindent \underline{\it Part $\mathbf{(4)}.$}  This is an instance of Paley-Wiener theorem for functions with the exponential decay. 

  \vspace{2mm}

  \noindent \underline{\it Part $\mathbf{(5)}.$}   Plancherel's theorem implies 
\begin{eqnarray}
  \nonumber  \|u(\cdot,t)_*\|_2^2 &=& \int_{\RR^d} |v_*(y,t)|^2 \dif y = {C \over \tau(t)^{2 c}} \int_{\RR^d} {|\psi_*(y \tau(t) )|^2} \dif y \\
  \nonumber &=&  {C \over \tau(t)^{2 c+d}}  \int_{\RR^d} {|\psi_*(\xi)|^2}\dif \xi= {C \over \tau(t)^{2 \gamma-2-d}}  \int_{\RR^d} {|\psi_*(\xi)|^2 }\dif \xi.
\end{eqnarray}

More generally, since 
\begin{eqnarray}
  \nonumber  \|u(\cdot,t)_*\|_m &=& {1 \over \tau(t)^{\gamma-1-{d \over m}}} \|\cF^{-1}[\vartheta_*]\|_m. 
\end{eqnarray}
Since $d/(\gamma-1)<2$ for all $\gamma>d$ and $d \ge 2$, we get that the $L^m$-norm diverges in finite time for all $m \ge 2$, $\gamma>d$ and $d \ge 2$. 

{\flushright $\Box$}

\begin{remark} \label{notstrong}  We recall that a  mild solution $u$ of a generalized Navier-Stokes system is  strong if
$$\partial_t u, \ (u \cdot \nabla) u, \  (-\Delta)^{\gamma \over 2} u \in L^2((0,T); \bL^2).$$
According to the definition of $u_*$,
$$ \|(-\Delta)^{\gamma \over 2} u_* \|_2 = C   \tau^{-2 \gamma+1+{d \over 2}} \|(1+|\eta|^2)^{\gamma \over 2} \psi_* \|_2.$$
Therefore, $(-\Delta)^{\gamma \over 2} u_* \in L^2\left((0,T),\bL^2 \right)$ iff
\begin{equation}
\label{gd} -4 \gamma +2 +d > -1 \implies \gamma <{d +3 \over 4}.
\end{equation}

However, for $d \ge 2$, the conditions $d < \gamma < 2 d+2$ and $(\ref{gd})$ have an empty solution. Therefore, the mild solutions that we have found are not strong, specifically, $(-\Delta)^{\gamma \over 2} u_*$ is not in $L^2((0,T),\bL^2)$.

For completeness, we will analyze $ (u_* \cdot \nabla) u_*$ and $\partial_t u_*$ as well

The non-linear term can be bounded as follows:
\begin{eqnarray}
  \nonumber \int_0^T \hspace{-2.5mm} \|( u_* \cdot \nabla) u_*\|^2_2 \hspace{0.6mm}  d t \hspace{-1.0mm}&\hspace{-1.0mm} \le \hspace{-1.0mm} &  \hspace{-1.0mm}  \int_0^T  \int_{\RR^d} |u_*(x,t)|^2 |\nabla u_*(x,t)|^2 d x \ d t \\
  \nonumber  \hspace{-1.0mm}&\hspace{-1.0mm} \le \hspace{-1.0mm} &  \hspace{-1.0mm}   \int_0^T \left(  \int_{\RR^d} |u_*(x,t)|^{2 l}  d x\right)^{1 \over l} \left( \int_{\RR^d}|\nabla u_*(x,t)|^{2 l'} d x \right)^{1 \over l'} \ d t.
\end{eqnarray}
where $1/l+1/l'=1$.
\begin{eqnarray}
  \nonumber \int_0^T \hspace{-2.5mm} \|( u_* \cdot \nabla) u_*\|^2_2 d t \hspace{-1.0mm}&\hspace{-1.0mm} \le \hspace{-1.0mm} &  \hspace{-1.0mm}   C \hspace{-1mm} \int_0^T \hspace{-3.5mm} \tau(t)^{2 -4 \gamma+d} \hspace{-0.7mm} \left( \hspace{1.0mm}  \int \displaylimits_{\RR^d} \hspace{-1.0mm} |\cF^{-1}[\vartheta_*](y)|^{2 l} d y\right)^{\hspace{-1.5mm}{1 \over l}} \hspace{-2mm} \left( \hspace{1.0mm} \int \displaylimits_{\RR^d} \hspace{-1.0mm} |\nabla \cF^{-1}[\vartheta_*](y)|^{2 l'} d y \right)^{\hspace{-1.5mm} {1 \over l'}} \hspace{-2.0mm} d t.
\end{eqnarray}
By Part $(\mathbf{2})$, for any $t \in (0,T)$, $\cF^{-1}[\psi_*] \in L^m(\RR^d,\RR)$ for any $m \ge 2$, in particular, for $m=2 l$ for any $l>1$. Furthermore, 
$$ \left( \int_{\RR^d} |\eta  \psi_*(\eta)|^{2l' \over 2l'-1} d \eta \right)^{3 \over 4} \le \left(  \int_{\RR^d} |\eta|^q u_\nu^{-q}(\eta)   d \eta  \right)^{1 \over q} \| \psi_* \|_{u_\nu,p},$$
where $1/q+1/p=(2 l'-1) /2 l'$. The first integral in this inequality is bounded if $1-\nu>-d/q$ which is equivalent to  
$$\kappa< 2+d {l'-1 \over l'}.$$
This inequality is weaker than in $(\ref{H13}) ({\bf 4})$ if one makes a choice $l'=p/2$, for example. We obtain that $\eta \psi_* \in L^{2 l'-1 \over 2 l'}$, which implies that $\nabla \cF^{-1} [\psi_*] \in L^{2 l'}$. We conclude that $(u \cdot \nabla) u \in L^2((0,T),\bL^2)$ if, again,  $(\ref{gd})$ holds.

Finally, consider $\partial_t u_*$:
\begin{equation}
\nonumber \partial_t u_*(x,t)={\gamma -1 \over \gamma} \tau(t)^{1-2\gamma} \cF^{-1}[\vartheta]_*(x \tau(t)^{-1})+ {1 \over \gamma} \tau(t)^{-2 \gamma} ( x \cdot \nabla) \cF^{-1}[\vartheta_*](x \tau(t)^{-1}),
\end{equation}
and
\begin{eqnarray}
  \nonumber \| \partial_t u_*( \cdot, t) \|_2^2 &=& \tau(t)^{2-4\gamma+d} \int_{\RR^d} \left| {\gamma-1 \over \gamma} \cF^{-1}[\vartheta_*](y)+ {1 \over \gamma} (y \cdot \nabla ) \cF^{-1}[\vartheta_*](y) \right|^2 d y \\
  \nonumber &=&C  \tau(t)^{2-4\gamma+d}, 
\end{eqnarray}
and $\partial_t u_*$ is not in $L^2((0,T),\bL^2)$ either.
\end{remark}

\section{Appendix A: Proof of Proposition $\ref{symmetries}$}\label{appendixA}

\noindent {\it Proof of Proposition  $\ref{symmetries}$.}   We omit trivial proofs of parts $i)$ and $ii)$.

  \vspace{2mm}

  \noindent $iii)$ Consider
  \begin{align}
    \nonumber      \cR_\beta[\psi](R \eta) &= \beta^c e^{|R \eta|^\gamma \left(1- {1 \over \beta^\gamma}  \right)} \psi \left(R \eta \over \beta \right) -\\
    \nonumber & \hspace{5mm}-  \gamma  \int \displaylimits_{1}^{1 \over \beta}  { e^{|R \eta|^\gamma \left(1-t^\gamma\right)} \over  t^{c} }   \int \displaylimits_{\RR^d}  R \eta \cdot \psi  \hspace{-0.5mm}  \left(R \eta t - x \right)  \left( \psi(x) - {R \eta \cdot \psi(x) \over |R \eta|^2   } R \eta  \right)   \ d x   \ d t \\
    \nonumber      &= \beta^c e^{|\eta|^\gamma \left(1- {1 \over \beta^\gamma}  \right)} R \psi \left(\eta \over \beta \right)- \gamma  \int \displaylimits_{1}^{1 \over \beta}  { e^{|\eta|^\gamma \left(1-t^\gamma\right)} \over  t^{c} }   \int \displaylimits_{\RR^d}  \eta \cdot R^{-1} \psi  \hspace{-0.5mm}  \left(R  \left(\eta t - R^{-1} x \right)  \right) \times \\
    \nonumber & \hspace{30mm} \times  \left( R R^{-1} \psi(R R^{-1} x) - {\eta \cdot R^{-1} \psi(R R^{-1} x) \over |\eta|^2   } R \eta  \right)  \ d x   \ d t.
  \end{align}
Performing the change of variables $\tilde x= R^{-1} x$ in the convolution integral and using the assumption $\psi=R^{-1} \psi \circ R$, 
  \begin{align}
    \nonumber      \cR_\beta[\psi](R \eta) &= \beta^c e^{|\eta|^\gamma \left(1- {1 \over \beta^\gamma}  \right)} R \psi \left(\eta \over \beta \right)- \\
    \nonumber & \hspace{5mm} - \gamma  \int \displaylimits_{1}^{1 \over \beta}  { e^{|\eta|^\gamma \left(1-t^\gamma\right)} \over  t^{c} }   \int \displaylimits_{\RR^d}  \eta \cdot \psi  \hspace{-0.5mm} \left(\eta t - \tilde x \right)  \left( R \psi(\tilde x) - {\eta \cdot \psi(\tilde  x) \over |\eta|^2   } R \eta  \right)  \ d \tilde x   \ d t \\
    \nonumber      &= R \left(  \beta^c e^{|\eta|^\gamma \left(1- {1 \over \beta^\gamma}  \right)} \psi \left(\eta \over \beta \right)- \gamma  \int \displaylimits_{1}^{1 \over \beta}  { e^{|\eta|^\gamma \left(1-t^\gamma\right)} \over  t^{c} }   \int \displaylimits_{\RR^d}  \eta \cdot \psi  \hspace{-0.5mm} \left(\eta t - \tilde x \right) P_\eta \psi(\tilde x)  \ d \tilde x   \ d t \right) \\
    \nonumber &=  R  \cR_\beta[\psi](\eta)
 \end{align}
  \vspace{2mm}
 \noindent $iv)$  The claim follows from the fact that
  $$\eta \cdot \psi \left(\eta \over \beta \right)= \beta \left( \eta \over \beta \right) \cdot \psi \left(\eta \over \beta \right)=0$$ 
 whenever $\eta \cdot \psi(\eta)=0$ for all $\eta \in \RR^d$, together with the fact that
 $$\eta \cdot P_\eta \psi(x)=0$$
 for any $x \in \RR^d$. 

 \vspace{2mm}
 \noindent $v)$ Assume that $\eta$ is not col linear with none of the polar axes. Consider  a vector
 $$x=(\rho,\hat x_p, \hat x_{d-1})$$
 where $\hat x_p=(\hat x_1, \ldots, \hat x_{d-2})$. Let $G \approx  SO(2)$ be the subgroup of rotations that fixes the azimuthal plane $W$ of the angle $\hat x_{d-1}$. Since $e_k(x) \cdot \psi(x)=0$, $1 \le k \le d-2$,  and $e_x(x) \cdot \psi(x)=0$, we have that $\psi(x) \in W$ for any $x$.

Let $R_x \in G$ be a rotation such that
 $$ \reallywidehat{(R_x x)}_{d-1} = \hat \eta_{d-1}.$$
 We have:
 \begin{eqnarray}
   \nonumber    \psi(x)+\psi(R_x^2 x)= \psi(R_x^{-1} R_x x) + \psi(R_x R_x x) = R_x \psi(R_x x) + R_x^{-1} \psi(R_x x).
 \end{eqnarray}
By symmetry this vector field is orthogonal to the subspace spanned by the polar vectors and $\eta$: 
\begin{equation}
\label{orto1} R_x \psi(R_x x) +  R_x^{-1} \psi(R_x x) \perp {\rm span}\{\eta,e_1(\eta), \ldots , e_{d-2}(\eta)\}.
\end{equation}
Consider
\begin{eqnarray}
  \nonumber  e_k(\eta) \cdot N(\eta) &=& - \gamma  e_k(\eta) \cdot \int \displaylimits_{1}^{1 \over \beta}  { e^{|\eta|^\gamma \left(1-t^\gamma\right)} \over  t^{c} }   \int \displaylimits_{\RR^d}  \eta \cdot \psi  \hspace{-0.5mm}  \left(\eta t - x \right)  P_\eta \psi(x)   \ d x   \ d t \\
  \nonumber   &=& - \gamma   \int \displaylimits_{1}^{1 \over \beta}  { e^{|\eta|^\gamma \left(1-t^\gamma\right)} \over  t^{c} }   \int \displaylimits_{\RR^d}  \eta \cdot \psi  \hspace{-0.5mm}  \left(\eta t - x \right)  \left( e_k(\eta) \cdot \psi(x)  \right)  \ d x   \ d t,
\end{eqnarray}
$1 \le k \le d-2$. We will now consider only the integral over $S_{2}$ in the convolution part of the product $e_k \cdot N$, denoted as $\cC$. Notice that as the azimuthal angle $\hat x_{d-1}$  ranges from $ \hat \eta_{d-1}- \pi$ to $\hat \eta_{d-1}$ (counterclockwise), the angle $ \reallywidehat{\left(R_x^2 x \right)}_i$ ranges from  $\hat \eta_{d-1}+\pi$ to $\hat \eta_{d-1}$ (clockwise), therefore,
\begin{eqnarray}
  \nonumber \cC& :=&   \int \displaylimits_{\hat \eta_{d-1}- \pi}^{\hat \eta_{d-1}+\pi}  \hspace{-3mm}   \eta \cdot \psi  \hspace{-0.5mm}  \left(\eta t - x \right)  \left( e_k(\eta) \cdot \psi(x)  \right) \ d \hat{x}_{d-1}  \\
  \nonumber &=&   \int \displaylimits_{\hat \eta_{d-1}- \pi}^{\hat \eta_{d-1}}  \hspace{-3mm}   \eta \cdot \psi  \hspace{-0.5mm}  \left(\eta t \hspace{-0.5mm}- \hspace{-0.5mm} x \right)  e_k(\eta) \cdot \psi(x) -  \eta \cdot \psi  \hspace{-0.5mm}  \left(\eta t  \hspace{-0.5mm} -  \hspace{-0.5mm} R_x^2 x \right)  e_k(\eta) \cdot \psi(R_x^2 x)  \ d \hat{x}_{d-1}.
\end{eqnarray}
Again, by symmetry considerations,
$$\eta \cdot \psi  \hspace{-0.5mm}  \left(\eta t - R_x^2 x \right)=-\eta \cdot \psi  \hspace{-0.5mm}  \left(\eta t - x \right),$$
and
\begin{equation}
  \nonumber \cC:=   \int \displaylimits_{\hat \eta_{d-1}- \pi}^{\hat \eta_{d-1}}    \eta \cdot \psi  \hspace{-0.5mm}  \left(\eta t - x \right)  e_k(\eta) \cdot \left(\psi(x) + \psi(R_x^2 x) \right) \ d \hat{x}_{d-1} .
\end{equation}
By $(\ref{orto1})$, $\cC=0$.

To demonstrate $(\ref{Ronphi})$ we remark that a $SO(2)$-invariant azimuthal field $\psi$ is of the form
$$\psi(\eta)=\phi(|\eta|,\hat{\eta}_p) \  e_{d-2}(\eta).$$
Therefore, counting the azimuthal angle in the integration from the polar plane of $\eta$,
\begin{eqnarray}
  \nonumber \cC&:=&  \int \displaylimits_{- \pi}^{\pi}    \eta \cdot \psi  \hspace{-0.5mm}  \left(\eta t - x \right)  e_{d-1}(\eta) \cdot \psi(x) \ d \hat{x}_{d-1}  \\
  \nonumber &=& -|\eta|  \int \displaylimits_{ -\pi}^{\pi} \left( \prod_{i=1}^{d-2}  \sin \hat \eta_i \right) \sin \reallywidehat{(t \eta-x)}_{d-1} \  \phi  \hspace{-0.5mm}  \left(|\eta t - x|,  \reallywidehat{(t \eta-x)}_p  \right) \times \\
  \nonumber   &\phantom{=}& \hspace{ 80mm}   \times   \cos \hat{x}_{d-1} \ \phi(|x|,\hat{x}_p) \ d \hat{y}_{d-1},
\end{eqnarray}
where we have used the fact that the projection of $\eta$ onto the azimuthal plane has length $|\eta|   \prod_{i=1}^{d-2}  \sin \widehat{\eta}_i$.

Set $z:=(|x|, \hat{x}_p, -\hat{x}_{d-1})$ (note, that $x=R^2_z z$), then
\begin{eqnarray}
  \nonumber \cK&:=& \hspace{-2mm} \int \displaylimits_{- \pi}^{0} \sin  \widehat{(t \eta-x)}_{d-1} \  \phi  \hspace{-0.5mm}  \left(|\eta t - x|,  \widehat{(t \eta-x)}_p  \right) \  \cos \hat{x}_{d-1} \ \phi(|x|,\hat{x}_p) \ d \hat{x}_{d-1}  \\
  \nonumber && - \int \displaylimits_{0}^{- \pi} \sin  \reallywidehat{(t \eta-R^2_z z)}_{d-1} \  \phi  \hspace{-0.5mm}  \left(|\eta t - R^2_z z|,  \reallywidehat{(t \eta-R^2_z z)}_p  \right) \  \cos \hat{z}_{d-1} \ \phi(|z|,\hat{z}_p) \ d \hat{z}_{d-1}.
\end{eqnarray}
Since
\begin{eqnarray}
  \nonumber  \reallywidehat{(t \eta-R^2_z z)}_{d-1} &=& -\reallywidehat{(t \eta-z)}_{d-1}, \\
  \nonumber \reallywidehat{(t \eta-R^2_z z)}_p &=& \reallywidehat{(t \eta-z)}_p, \\
  \nonumber |t \eta-R^2_z z| &=&  |t \eta-z|,
\end{eqnarray}
we get
\begin{eqnarray}
  \nonumber \cK&:=& \hspace{-2mm} \int \displaylimits_{- \pi}^{0} \sin  \widehat{(t \eta-x)}_{d-1} \  \phi  \hspace{-0.5mm}  \left(|\eta t - x|,  \widehat{(t \eta-x)}_p  \right) \  \cos \hat{x}_{d-1} \ \phi(|x|,\hat{x}_p) \ d \hat{x}_{d-1}  \\
  \nonumber && -\int \displaylimits_{-\pi}^{0} \sin  \reallywidehat{(t \eta-z)}_{d-1} \  \phi  \hspace{-0.5mm}  \left(|\eta t - z|,  \reallywidehat{(t \eta-z)}_p  \right) \  \cos \hat{z}_{d-1} \ \phi(|z|,\hat{z}_p) \ d \hat{z}_{d-1} \\
  \nonumber &=&0.
\end{eqnarray}

{\flushright $\square$}

\section{Appendix B: Proof of Proposition $\ref{equicont}$} \label{appA}

\noindent {\it Proof of Proposition $\ref{equicont}$} $\phantom{aa}$ \\
\noindent \textit{\textbf{1)  Linear terms, case $\bm{|\delta|<  \beta \delta_0}$. } } Similarly to $(\ref{Lestimate2})$,
  \begin{equation} \label{Lestimate5}
  \| L-\tau_\delta L \|_{v_\nu,p}   \le   C A (1-\beta)  |\delta|^{\min\{1,\gamma\}} + \beta^{c-\theta+\nu+{d \over p}}  K  |\delta|^\theta.
  \end{equation}
\medskip

\noindent \textit{\textbf{ 2)  Nonlinear terms, case $|\eta|<3 |\delta|$ and $\bm{|\delta|<  \beta \delta_0}$. } }  We proceed to estimate the nonlinear terms. Again, all calculations in this part are made for $|\delta|< \beta \delta_0$. We  consider the case $|\eta| \le 3 |\delta|$.

Denote for brevity $I=\int \displaylimits_{|\eta| \le 3 |\delta|}  \left| N(\eta) -N(\eta-\delta) \right|^p u_\nu^p(\eta) \ d \eta$.
\begin{eqnarray}
  \nonumber I &\le & C  \int \displaylimits_{|\eta| \le 3 |\delta|}   | N(\eta)|^p u_\nu^p(\eta) \ d \eta  + C  \int \displaylimits_{|\eta| \le 3 |\delta|}  \hspace{-3mm} |N(\eta-\delta)|^p u_\nu^p(\eta) \ d \eta  \\
  \nonumber   &\le  & C \hspace{-4mm} \int \displaylimits_{|\eta| \le 3 |\delta|}  | N(\eta)|^p u_\nu^p(\eta) \ d \eta  +  C \ e^{|\delta|p}  \int \displaylimits_{|z| \le 4 |\delta|}   |N(z)|^p e^{|z|p} |\delta|^{p \nu} \ d \eta.
\end{eqnarray}

\medskip

\noindent   \textit{ \textbf{ a)}.} Consider the case  \textit{ \textbf{ $4 |\delta| < \left({1 \over \beta^\gamma}-1 \right)^{-{1 \over \gamma}}$.}}
\begin{equation}
  \nonumber I  \le  C \  \| \mt{\psi} \|_{u_\nu,p}^{2 p} \left({1 \over \beta^\gamma}-1   \right)^{p \over s} \left( e^{3 p |\delta|} \int \displaylimits_{|\eta| \le 3 |\delta|}  \hspace{-3mm} |\eta|^{p (l(c)+c+\nu)} d \eta +  e^{4 p |\delta|} |\delta|^{p \nu}  \int \displaylimits_{|\eta| \le 4 |\delta|}  \hspace{-3mm} |\eta|^{p (l(c)+c)} d \eta  \right).
\end{equation}
According to $(\ref{H5})$,
 \begin{eqnarray}
   \nonumber p (l(c)+c+\nu) +d &=&{d (r-p)\over r} +p(\nu + \kappa +1) > p>0,\\
   \nonumber p (l(c)+c) +d & >& (p-p \nu)>0,
 \end{eqnarray}
and  under hypothesis $(\ref{H5})$, both integrals converge:
\begin{eqnarray}
  \nonumber \int \displaylimits_{|\eta| \le 3 |\delta|} \hspace{-3mm}  \left| N(\eta) -N(\eta-\delta) \right|^p u_\nu^p(\eta) \ d \eta  & \le &  C \  \| \mt{\psi} \|_{u_\nu,p}^{2 p} \left({1 \over \beta^\gamma}-1   \right)^{p \over s} |\delta|^{d{r-p \over r} +p(1+\kappa +\nu)} \\
  \nonumber   & \le &  C \  \| \mt{\psi} \|_{u_\nu,p}^{2 p} \left({1 \over \beta^\gamma}-1   \right)^{p \over s}  |\delta|^{p}.
 \end{eqnarray}
\medskip

\noindent   \textit{ \textbf{ b)}.}  In the case,   \textit{ \textbf{ $4 |\delta| > \left({1 \over \beta^\gamma}-1 \right)^{-{1 \over \gamma}}$ and  $3 |\delta| < \left({1 \over \beta^\gamma}-1 \right)^{-{1 \over \gamma}}$}},
\begin{eqnarray}
  \nonumber I &\le & C \  \| \mt{\psi} \|_{u_\nu,p}^{2 p} \left({1 \over \beta^\gamma}-1   \right)^{p \over s} \left( \hspace{1mm} \int \displaylimits_{|\eta| \le 3 |\delta|}  \hspace{-3mm} |\eta|^{p (l(c)+c+\nu)} d \eta +  |\delta|^{p \nu} \hspace{-6mm} \int \displaylimits_{|\eta| \le  \left({1 \over \beta^\gamma}-1  \right)^{-{1 \over \gamma}} }  \hspace{-3mm} |\eta|^{p (l(c)+c)} d \eta \right) +\\
  \nonumber && \hspace{25mm} + \ C \  \| \mt{\psi} \|_{u_\nu,p}^{2 p} |\delta|^{p \nu}  \hspace{-8mm} \int \displaylimits_{ \left({1 \over \beta^\gamma}-1  \right)^{-{1 \over \gamma}}< |\eta| \le 4 |\delta|}  \hspace{-8mm} |\eta|^{p (l(c)+c)-{ \gamma p \over s}} d \eta \\
  \nonumber  &\le &   C \  \| \mt{\psi} \|_{u_\nu,p}^{2 p} \left( \hspace{-2mm} \left({1 \over \beta^\gamma}-1   \right)^{p \over s}  \hspace{-2mm} \left(|\delta|^p  \hspace{-0.5mm} +   \hspace{-0.5mm}  \left({1 \over \beta^\gamma}-1   \right)^{m-{p \over s} +p{\nu \over \gamma}} \hspace{-2mm} |\delta|^{p  \nu }  \right)\hspace{-0.5mm} +   \hspace{-0.5mm} |\delta|^{p \nu}  \hspace{-10mm} \int \displaylimits_{ \left({1 \over \beta^\gamma}-1  \right)^{-{1 \over \gamma}}< |\eta|}  \hspace{-9mm} |\eta|^{p (l(c)+c)-{ \gamma p \over s}} d \eta \hspace{-1mm} \right)\\
  \nonumber  &\le &   C \  \| \mt{\psi} \|_{u_\nu,p}^{2 p}  \left(  \left({1 \over \beta^\gamma}-1   \right)^{p \over s}  \left( \hspace{-1mm} |\delta|^p +  \left({1 \over \beta^\gamma}-1   \right)^{m-{p \over s}+p{\nu \over \gamma}}  \hspace{-2mm}   |\delta|^{p \nu} \right) + \left({1 \over \beta^\gamma}-1   \right)^{m+{p \nu \over \gamma}}   \hspace{-1mm}  |\delta|^{p \nu} \hspace{-1mm} \right).
\end{eqnarray}
According to hypothesis $(\ref{H3})$,
$${m+p{\nu \over \gamma}}=-{1 \over \gamma} \left( p(1+\kappa)+d\ {r-p \over r}-{\gamma p \over s} \right)>1+{p \nu \over \gamma }$$
and 
we get
\begin{eqnarray}
  \nonumber  I & \le &   C \  \| \mt{\psi} \|_{u_\nu,p}^{2 p}  \left(   \left({1 \over \beta^\gamma}-1   \right)^{p \over s}  |\delta|^p +  \left({1 \over \beta^\gamma}-1   \right)^{1+{p \nu \over \gamma}} |\delta|^{p \nu}\right).
\end{eqnarray}

\medskip

\noindent  \textit{ \textbf{ c)}.}  Finally, in the case,   \textit{ \textbf{ $3 |\delta| > \left({1 \over \beta^\gamma}-1 \right)^{-{1 \over \gamma}}$}}, going through estimates very similar to the case $b)$,
\begin{align}
  \nonumber I \hspace{-1mm}  & \le   C \  \| \mt{\psi} \|_{u_\nu,p}^{2 p} \left({1 \over \beta^\gamma}-1   \right)^{p \over s}  \hspace{-1mm}\left( \hspace{-8mm} \int \displaylimits_{ \hspace{10mm}|\eta| \le  \left({1 \over \beta^\gamma}-1  \right)^{-{1 \over \gamma}} }   \hspace{-8mm} |\eta|^{p (l(c)+c+\nu)} d \eta +  |\delta|^{p \nu} \hspace{-10mm} \int \displaylimits_{|\eta| \le  \left({1 \over \beta^\gamma}-1  \right)^{-{1 \over \gamma}} }  \hspace{-8mm} |\eta|^{p (l(c)+c)} d \eta \right) +\\
  \nonumber & \hspace{3mm} +  \ C \  \| \mt{\psi} \|_{u_\nu,p}^{2 p} |\delta|^{p \nu}  \hspace{-12mm} \int \displaylimits_{ \left({1 \over \beta^\gamma}-1  \right)^{-{1 \over \gamma}}< |\eta| \le 3 |\delta|}  \hspace{-12mm} |\eta|^{p (l(c)+c)-{ \gamma p \over s}} d \eta + C \  \| \mt{\psi} \|_{u_\nu,p}^{2 p} |\delta|^{p \nu}  \hspace{-12mm} \int \displaylimits_{ \left({1 \over \beta^\gamma}-1  \right)^{-{1 \over \gamma}}< |\eta| \le 4 |\delta|}  \hspace{-12mm} |\eta|^{p (l(c)+c)-{ \gamma p \over s}} d \eta \\
  \nonumber   & \le    C \  \| \mt{\psi} \|_{u_\nu,p}^{2 p} \left({1 \over \beta^\gamma}-1   \right)^{p \over s} \hspace{-1mm} \left( \hspace{1mm} \int \displaylimits_{|\eta| \le  3 |\delta| }   \hspace{-3mm} |\eta|^{p (l(c)+c+\nu)} d \eta +  |\delta|^{p \nu} \hspace{-8mm} \int \displaylimits_{|\eta| \le  \left({1 \over \beta^\gamma}-1  \right)^{-{1 \over \gamma}} }  \hspace{-5mm} |\eta|^{p (l(c)+c)} d \eta \right) +\\
  \nonumber  &  \hspace{3mm} +  \ C \  \| \mt{\psi} \|_{u_\nu,p}^{2 p} |\delta|^{p \nu}  \hspace{-8mm} \int \displaylimits_{ \left({1 \over \beta^\gamma}-1  \right)^{-{1 \over \gamma}}< |\eta| }  \hspace{-8mm} |\eta|^{p (l(c)+c)-{ \gamma p \over s}} d \eta  \\
     \label{J0estimate}    & \le    C \  \| \mt{\psi} \|_{u_\nu,p}^{2 p}  \left( \left({1 \over \beta^\gamma}-1   \right)^{p \over s} |\delta|^p + \left({1 \over \beta^\gamma}-1   \right)^{1+{p \nu\over \gamma}} |\delta|^{p \nu}  \right).
\end{align}

\medskip

\noindent \textit{\textbf{ 3)  Nonlinear terms, case $|\eta|>3 |\delta|$ and $\bm{|\delta|<  \beta \delta_0}$. }} We will proceed to estimate the difference  $N(\eta-\delta) -  N(\eta)$ in the case $|\eta| > 3 |\delta|$.
\begin{eqnarray}
  \nonumber \left| N(\eta)-N(\eta-\delta) \right|  \hspace{-2.0mm} &  \hspace{-2.0mm}  =   \hspace{-2.0mm} & \hspace{-1.0mm}  \gamma \left| \int \displaylimits_{1}^{1 \over \beta}  \hspace{-1.0mm}{ e^{|\eta|^\gamma \left(1-t^\gamma \right)} \over  t^{c} } \hspace{-1.5mm}  \int \displaylimits_{\RR^d} \hspace{-1.0mm} \eta \cdot \psi  \hspace{-0.5mm}  \left( \eta t-x \right) P_\eta \psi(x)   \ d x \ d t \right.\\
  \nonumber && \hspace{5mm} \left. -   \hspace{0.2mm}  \hspace{-1.0mm} \int \displaylimits_{1}^{1 \over \beta}     { e^{|\eta-\delta|^\gamma \left(1 - t^\gamma \right)} \over  t^{c} } \hspace{-1.5mm}  \int \displaylimits_{\RR^d} \hspace{-1.0mm} (\eta-\delta) \cdot \psi  \hspace{-0.5mm}  \left( (\eta-\delta) t - x \right) P_{\eta-\delta} \psi(x)   \ d x   \ d t   \ \right| \\
  \nonumber & \le & J_1(\eta,\delta) + J_2(\eta,\delta) + J_3(\eta,\delta)+ J_4(\eta,\delta),
\end{eqnarray}
where
\begin{equation}
  \nonumber     J_1(\eta,\delta) =  \hspace{0.2mm}  \gamma  \hspace{0.2mm}  \left|  \int \displaylimits_{1}^{1 \over \beta} \int \displaylimits_{\RR^d}   \hspace{-1.0mm}{ e^{|\eta|^\gamma \left( 1-t^\gamma \right)} \over  t^{c} }  \ \eta \cdot \left( \psi  \hspace{-0.5mm}  \left( \eta t - x \right) - \psi  \hspace{-0.5mm}  \left((\eta-\delta) t-x \right) \right) P_{\eta} \psi(x)   \ d x \ d t \phantom{  \int \displaylimits_{1}^{1 \over \beta}} \hspace{-3mm} \right|,
\end{equation}
\begin{eqnarray}
  \nonumber     J_2(\eta,\delta) &=&   \hspace{0.2mm}  \gamma  \hspace{0.2mm}  \left|  \int \displaylimits_{1}^{1 \over \beta} \int \displaylimits_{\RR^d}   \ { e^{|\eta|^\gamma \left( 1-t^\gamma \right)} \over  t^{c} }  \delta \cdot \psi  \hspace{-0.5mm}  \left((\eta-\delta) t-x \right) P_{\eta-\delta} \psi(x) \ d x \ d t  \right|, \\
  \nonumber     J_3(\eta,\delta) &=&  \hspace{-1.5mm}   \hspace{0.2mm} \gamma  \hspace{0.2mm}  \left| \int \displaylimits_{1}^{1 \over \beta}   \int \displaylimits_{\RR^d}  \hspace{-0.5mm} { e^{|\eta|^\gamma \left( 1 -t^\gamma \right)} - e^{ |\eta-\delta|^\gamma\left( 1 - t^\gamma \right)} \over  t^{c} }  (\eta-\delta) \cdot \psi  \hspace{-0.5mm}  \left( ( \eta-\delta ) t -x \right) P_{\eta-\delta} \psi(x)   \ d x \ d t \right|, \\
   \nonumber     J_4(\eta,\delta) &=&   \hspace{0.2mm}  \gamma  \hspace{0.2mm}  \left|  \int \displaylimits_{1}^{1 \over \beta} \int \displaylimits_{\RR^d}   \hspace{-1.0mm}{ e^{|\eta|^\gamma \left( 1-t^\gamma \right)} \over  t^{c} }  \ \eta \cdot \psi  \hspace{-0.5mm}  \left((\eta-\delta) t-x \right) \left( P_{\eta-\delta} \psi(x) -  P_{\eta} \psi(x)\right)  \ d x \ d t \phantom{  \int \displaylimits_{1}^{1 \over \beta}} \hspace{-3mm} \right|.
\end{eqnarray}
We continue with estimates on $J_i$ separately.

\vspace{2mm}

\noindent \underline{\it Calculation of $J_1$}. We first consider $J_1(\eta,\delta)$. Since $|\delta| t< \delta_0$ in the integral, we get
\begin{eqnarray}
  \nonumber  J_1(\eta,\delta) & \le & \hspace{0.3mm}  C  \hspace{0.5mm}  |\eta|  \int \displaylimits_{1}^{1 \over \beta}  { e^{|\eta|^\gamma \left( 1 -t^\gamma \right) } \over t^c }  \int \displaylimits_{\RR^d} \hspace{-0.5mm} \left| \psi  \hspace{-0.5mm}  \left( \eta t -x \right) - \psi  \hspace{-0.5mm}  \left(( \eta-\delta) t - x  \right) \right| | \psi(x) |  \ d x \ d t  \\
  \nonumber  & \le & \hspace{0.3mm}  C  \hspace{0.5mm}  |\eta|  \int \displaylimits_{1}^{1 \over \beta}  { e^{|\eta|^\gamma \left( 1 -t^\gamma \right) } \over t^c }  \ |\delta t|^\theta   \int \displaylimits_{\RR^d}   \omega_\psi(\eta t - x)  \ | \mt{\psi}(x) |  \ d x \ d t \\
  \nonumber  & \le & \hspace{0.3mm}  C  \hspace{0.5mm}  |\delta|^\theta \ |\eta|^{c-\theta} \  e^{ |\eta|^\gamma }  \ \int \displaylimits_{1}^{ 1 \over \beta }  { e^{-|\eta t|^\gamma} \over |\eta t|^{c-\theta} }    \left(   \omega_\psi  \star | \mt{\psi} | \right)(\eta t) \ d |\eta| t,
\end{eqnarray}
where $\mt{\psi} \in R_{u_\nu}^p(\RR^d,\RR^d)$ is the radial function that bounds $\psi$. We remark that the calculation of the norm $\|J_1 \|_{u_\nu,p}$ is identical to the calculation in Proposition $(\ref{inv2})$ (up to a substitution of $\| \mt{\psi}\|^2_{u_\nu,p}$ by  $\| \mt{\psi}\|_{u_\nu,p} \|\omega_\psi\|_{u_\nu,p}$ ), therefore
\begin{equation}
\label{J1estimate} \| J_1\|_{u_\nu,p} \le C_2 A K \left( {1 \over \beta^\gamma} -1\right)^{m \over p} |\delta|^\theta, 
\end{equation}
with the constant $C_2$ identical to that in $(\ref{bound8})$.

\vspace{5mm}

We proceed to bound $J_2$. We will use the shorthand notation $\zeta=\eta-\delta$. 

\vspace{2mm}

\noindent \underline{\it Calculation of $J_2$}.  We proceed with the change of variables $\zeta =\eta-\delta$:
\begin{eqnarray}
 \nonumber    J_2(\eta,\delta) & \le &   \hspace{0.2mm}  \gamma  \hspace{0.2mm} |\delta|   \int \displaylimits_{1}^{1 \over \beta} \int \displaylimits_{\RR^d}   \ { e^{|\zeta+\delta|^\gamma \left( 1-t^\gamma \right)} \over  t^{c} } | \psi  \hspace{-0.5mm}  \left(\zeta t-x \right)|  |\psi(x)| \ d x \ d t \\
   \label{J2first}       & \le &   \hspace{0.2mm}  \gamma  \hspace{0.2mm} |\delta|  e^{|\eta|^\gamma} |\zeta|^{c-1}  \int \displaylimits_{1}^{1 \over \beta} \int \displaylimits_{\RR^d}   \ { e^{ -(|\zeta t + \delta t|^\gamma-|\zeta  t|} \over  |\zeta t|^{c-\kappa} } {e^{|\zeta  t|} \over |\zeta t|^{\kappa  }} \left(|\mt{\psi}|  \star  |\mt{\psi}| \right)( \zeta t) \ d x \ d |\zeta| t.
\end{eqnarray}
We can now use Corollary $\ref{nicebound}$, the H\"older inequality with $1=1/s+1/r$, and Proposition $\ref{inv2}$.
\begin{eqnarray}
  \nonumber    J_2(\eta,\delta)  & \le &   C  \hspace{0.2mm} |\delta|  e^{|\eta|^\gamma} |\zeta|^{c-1}   \int \displaylimits_{B_{|\zeta| \over \beta} \setminus B_{|\zeta|}}  \int \displaylimits_{\RR^d}   \ { e^{ -||y|e_\zeta +\delta t|^\gamma  -|y|} \over  |y|^{c-\kappa +d -1} } {e^{|y|} \over |y|^{\kappa  }} \left(|\mt{\psi}|  \star  |\mt{\psi}| \right)(y) \ d x \ d y \\
\nonumber & \le &  C  \hspace{0.2mm} |\delta|  e^{|\eta|^\gamma-|\zeta|} |\zeta|^{c-1}  \left(  \int \displaylimits_{B_{|\zeta| \over \beta} \setminus B_{|\zeta|}}  { { e^{ -s ||y| e_\zeta+\delta t|^{\gamma}} \over  |y|^{s(c-\kappa +d -1)} }  }  d y\right)^{1 \over s} \| \mt{\psi} \|_{u_\nu}^2.
\end{eqnarray}
We notice that for $|\eta| \ge 3 |\delta|$,
\begin{equation}
\label{comens}{2 \over 3} \le {|\zeta t + \delta t| \over |\zeta t|}= {||y| e_\zeta + \delta t| \over |y|} \le {4 \over 3},
\end{equation}
therefore, after a change of variables $z=\zeta t +\delta t=\eta t$,
\begin{equation}
  \nonumber    J_2(\eta,\delta)  \le   C  \hspace{0.2mm} |\delta|  e^{|\eta|^\gamma-|\zeta|} |\zeta|^{c-1}  \left(  \int \displaylimits_{B_{|\eta| \over \beta} \setminus B_{|\eta|}}  { { e^{ -s |z|^{\gamma}} \over  |z|^{s(c-\kappa +d -1)} }  }  d z \right)^{1 \over s} \| \mt{\psi} \|_{u_\nu}^2.
\end{equation}
We can now proceed as in Lemma $\ref{Ik}$, to get
\begin{equation}
  \nonumber    J_2(\eta,\delta)  \le   C  \hspace{0.2mm} |\delta|  e^{|\eta|^\gamma}  e^{|\eta|-|\zeta|} |\eta|^{c-1}  w_c(\eta) \| \mt{\psi} \|_{u_\nu}^2,
\end{equation}
where $w_c$ has been defined in $(\ref{wk})$. We would like to prove that the function
$$o_{\cR_\beta [\psi]}^2(\eta; \delta):=  e^{|\eta|^\gamma}  e^{|\eta|-|\zeta|} |\eta|^{c-1}  w_c(\eta)$$
belongs to $R_{u_\nu}^p(\RR^d,\RR_+)$ for every fixed $\delta$, under the hypothesis $(\ref{H5})$. We have that for $|\delta| < \delta_0$, $|\eta|-|\zeta| < \delta_0$, and  
$$o_{\cR_\beta [\psi]}^2(\eta; \delta) \le   e^{|\eta|^\gamma}  e^{\delta_0} |\eta|^{c-1}  w_c(\eta) \le C  e^{|\eta|^\gamma} |\eta|^{c-1}  w_c(\eta).$$
Set
\begin{equation}
  \label{omega2}  \tilde \omega^2_{\cR_\beta[\psi]}(\eta) :=    |\eta|^{c-1} \  e^{ |\eta|^\gamma}  w_c(\eta),
\end{equation}
Then
\begin{equation}
  \nonumber    J_2(\eta,\delta)  \le   C  \hspace{0.2mm} |\delta|   \tilde \omega^2_{\cR_\beta[\psi]}(\eta) \| \mt{\psi} \|_{u_\nu}^2,
\end{equation}
where $ \omega^2_{\cR_\beta[\psi]} \in  R_{u_\nu}^p(\RR^d,\RR_+)$ under the hypothesis $(\ref{H5})$. More specifically, using $(\ref{v_c})$, similarly to the computation of $J_1$,
\begin{equation}
 \tilde \omega^2_{\cR_\beta[\psi]}(\eta) = \left({1 \over \beta^\gamma}-1 \right)^{{m \over p} +{1 \over \gamma}}  \omega^2_{\cR_\beta[\psi]}(\eta)
\end{equation}
for some  $\omega^2_{\cR_\beta[\psi]} \in  R_{u_\nu}^p(\RR^d,\RR_+)$ which can be bounded independently of $\beta$. We obtain, eventually,
\begin{equation}
  \label{J2estimate}    J_2(\eta,\delta)  \le   C  \left({1 \over \beta^\gamma}-1 \right)^{{m \over p} +{1 \over \gamma}}  \ \omega^2_{\cR_\beta[\psi]}(\eta) \| \mt{\psi} \|_{u_\nu}^2 \ |\delta|,
\end{equation}

\vspace{2mm}

\noindent \underline{\it Calculation of $J_3$}.
A before, we set $\zeta=\eta-\delta$. Suppose that $|\zeta|>|\eta|$. Then,
\begin{eqnarray}
  \nonumber     J_3(\eta,\delta)    & \le &  C |\zeta|^c  \int \displaylimits_{1}^{1 \over \beta}   \hspace{-0.5mm} \left| e^{|\eta|^\gamma -|\eta t|^\gamma}-  e^{|\zeta|^\gamma - |\zeta t|^\gamma} \right| {1 \over  |\zeta t |^c }  \left( |\mt{\psi}|\star |\mt{\psi}| \right)(\zeta t)  \ d |\zeta| t  \\
  \nonumber       & \le &  C |\zeta|^c   \left(1- e^{(|\eta|^\gamma -|\zeta|^\gamma)\left({1 \over \beta^\gamma} -1\right)}    \right)  \int \displaylimits_{1}^{1 \over \beta}  \hspace{-0.5mm} {e^{|\eta|^\gamma - |\eta t|^\gamma} \over  |\zeta t |^c }  \left( |\mt{\psi}|\star |\mt{\psi}| \right)(\zeta t)  \ d |\zeta| t \\
  \nonumber       & \le &  C |\zeta|^c   \left(1- e^{(|\eta|^\gamma -|\zeta|^\gamma)\left({1 \over \beta^\gamma}-1 \right)}    \right)  e^{|\eta|^\gamma} \int \displaylimits_{B_{|\zeta| \over \beta} \setminus B_{|\zeta|}}  \hspace{-0.5mm} {e^{-\left|y+ \delta {|y| \over |\zeta|  } \right|^\gamma} \over  |y|^{c+d-1}}  \left( |\mt{\psi}|\star |\mt{\psi}| \right)(y)  \ d y \\
  \nonumber       & \le &  C |\zeta|^c   \left(1- e^{(|\eta|^\gamma -|\zeta|^\gamma)\left({1 \over \beta^\gamma} -1 \right)}    \right)  e^{|\eta|^\gamma}  \left( \int \displaylimits_{B_{|\zeta| \over \beta} \setminus B_{|\zeta|}}  \hspace{-0.5mm} {e^{-s\left|y+ \delta {|y| \over |\zeta|  } \right|^\gamma} \over  |y|^{s(c-\kappa+d-1)}}   d y \right)^{1 \over s} \ \| \mt{\psi} \|_{u_\nu,p}^2.
\end{eqnarray}
Since $|\eta|$ is commensurable with $|\eta+\delta|$, see $(\ref{comens})$, for $|\eta|>3 |\delta|$, we get that $|y|$ is commensurable with $\left|y+ \delta {|y| \over |\zeta|  } \right|$. Therefore,
\begin{eqnarray}
  \nonumber     J_3(\eta,\delta)     & \le &  C |\zeta|^c   \left(1- e^{(|\eta|^\gamma -|\zeta|^\gamma)\left({1 \over \beta^\gamma} -1 \right)}    \right)  e^{|\eta|^\gamma}  \left( \int \displaylimits_{B_{|\zeta+\delta| \over \beta} \setminus B_{|\zeta+\delta|}}  \hspace{-0.5mm} {e^{-s|z|^\gamma} \over  |z|^{s(c-\kappa+d-1)} }   d z \right)^{1 \over s} \ \| \mt{\psi} \|_{u_\nu,p}^2 \\
  \nonumber        & \le &  C |\zeta|^c   \left(1- e^{(|\eta|^\gamma -|\zeta|^\gamma)\left({1 \over \beta^\gamma} -1 \right)}    \right)  e^{|\eta|^\gamma}  w_c(|\eta|) \ \| \mt{\psi} \|_{u_\nu,p}^2 \\
      \nonumber        & \le &  C |\eta|^c   \left(1- e^{(|\eta|^\gamma -|\zeta|^\gamma)\left({1 \over \beta^\gamma} -1 \right)}    \right)  e^{|\eta|^\gamma}  w_c(|\eta|) \ \| \mt{\psi} \|_{u_\nu,p}^2.
\end{eqnarray}
Similarly, if  $|\eta|>|\zeta|$,
\begin{eqnarray}
  \nonumber     J_3(\eta,\delta)    & \le & C |\zeta|^c   \left(1- e^{(|\zeta|^\gamma -|\eta|^\gamma)\left({1 \over \beta^\gamma} -1 \right)}    \right)  e^{|\zeta|^\gamma}  \left( \int \displaylimits_{B_{|\zeta| \over \beta} \setminus B_{|\zeta|}}  \hspace{-0.5mm} {e^{-s |y|^\gamma} \over  |y|^{s(c-\kappa+d-1)}}   d y \right)^{1 \over s} \ \| \mt{\psi} \|_{u_\nu,p}^2 \\
     \nonumber        & \le &  C |\zeta|^c   \left(1- e^{(|\zeta|^\gamma -|\eta|^\gamma)\left({1 \over \beta^\gamma} -1 \right)}    \right)  e^{|\zeta|^\gamma}  w_c(|\zeta|) \ \| \mt{\psi} \|_{u_\nu,p}^2.
\end{eqnarray}
Consider
$$\int_{|\eta|> 3 |\delta|}|J_3|^p u_\nu(\eta)^p |\eta|^{ p \nu} d \eta.$$
Since, according to $(\ref{v_c})$,
\begin{equation}
  \label{bound12} |\zeta|^{c+\nu} w_c(|\zeta|) e^{|\zeta|^\gamma+|\eta|} \le e^{|\delta|} |\zeta|^{c+\nu} w_c(|\zeta|) e^{|\zeta|^\gamma+|\zeta|}
  \end{equation}
is bounded by a power function, we get that $(\ref{bound12})$ is commensurable with $v_c(|\eta|)$. Therefore,
\begin{eqnarray}
  \nonumber  \int \displaylimits_{|\eta|> 3 |\delta|}|J_3|^p u_\nu(\eta)^p |\eta|^{ p \nu} d \eta    & \le & C   \ \| \mt{\psi} \|_{u_\nu,p}^{2p} \hspace{-3mm}  \int \displaylimits_{|\eta|>3 |\delta|}  \left( \hspace{1mm} \int \displaylimits_{S_{d-1} \cap |\zeta|>|\eta| } \left(1- e^{(|\eta|^\gamma -|\zeta|^\gamma)\left({1 \over \beta^\gamma} -1 \right)}    \right)^p  d \hat \eta  + \right. \\
  \nonumber && \left. + \int \displaylimits_{S_{d-1} \cap |\zeta| \le |\eta| } \left(1- e^{(|\zeta|^\gamma -|\zeta|^\gamma)\left({1 \over \beta^\gamma} -1 \right)}    \right)^p  d \hat \eta   \right)  \  v_c^p(|\eta|)  \  d  \ |\eta|  \\
\end{eqnarray}
We proceed to compute the integral
$$ I= \int \displaylimits_{S_{d-1} \cap |\zeta|>|\eta| } \left(1- e^{(|\eta|^\gamma -|\zeta|^\gamma)\left({1 \over \beta^\gamma} -1 \right)}    \right)^p  d \hat \eta  + \int \displaylimits_{S_{d-1} \cap |\zeta| \le |\eta| } \left(1- e^{(|\zeta|^\gamma -|\zeta|^\gamma)\left({1 \over \beta^\gamma} -1 \right)}    \right)^p  d \hat \eta.$$
To that end we use Lemma $\ref{radial}$ with $r=|\eta|$ and $\rho=|\delta|$. We also use convexity of the function $x \mapsto x^p$:
\begin{eqnarray}
  \nonumber     I  \hspace{-1mm} &  \hspace{-1mm} \le  \hspace{-1mm} &  \hspace{-1mm}  C \hspace{-1.5mm} \int \displaylimits_{-{\rho \over 2 r}}^{1} \hspace{-1.5mm} \left(1\hspace{-0.5mm}-\hspace{-0.5mm}e^{\left(r^\gamma -\left(r^2+\rho^2+2 r \rho t \right)^{\gamma \over 2 } \right) \left({1 \over \beta^\gamma}-1 \right)}  \right)^p \hspace{-2mm}d t  \hspace{-0.5mm} + \hspace{-0.5mm} C  \hspace{-1.5mm}  \int\displaylimits_{-1}^{-{\rho \over 2 r}} \hspace{-1.5mm}  \left(1\hspace{-0.5mm} - \hspace{-0.5mm} e^{\left(\left(r^2+\rho^2+2 r \rho t \right)^{\gamma \over 2 } -r^\gamma \right) \left({1 \over \beta^\gamma} -1 \right)}  \right)^p \hspace{-2mm} d t  \\
  \nonumber    \hspace{-1mm} &  \hspace{-1mm} \le  \hspace{-1mm} &  \hspace{-1mm}  C \left( \left(1+ {\rho \over 2 r} \right)  \left(1-e^{\left(r^\gamma  -\left(r+\rho \right)^{\gamma } \right) \left({1 \over \beta^\gamma}-1 \right)}  \right)^p+  \left(1-{\rho \over 2 r} \right)  \left(1-e^{\left((r-\rho)^\gamma -r^{\gamma} \right) \left({1 \over \beta^\gamma} -1 \right) }  \right)^p \right) \\
   \nonumber    \hspace{-1mm} &  \hspace{-1mm} \le  \hspace{-1mm} &  \hspace{-1mm}  C \left(1+ {\rho \over 2 r} \right)  \left(1-e^{\left(p r^\gamma  -p \left(r+\rho \right)^{\gamma } \right) \left({1 \over \beta^\gamma}-1 \right)}  \right)+  C \ \left(1-{\rho \over 2 r} \right)  \left(1-e^{p \left((r-\rho)^\gamma -p r^{\gamma} \right) \left({1 \over \beta^\gamma} -1 \right) }  \right).
\end{eqnarray}
We will now use the fact that for $|\eta|>3 |\delta|$, the arguments of both exponentials are bounded from below by
$$-D \rho \left({1 \over \beta^\gamma}-1 \right) r^{\gamma-1}$$
where $D>0$ is some constant. Therefore,
\begin{equation}
  \nonumber     I \le C \left(1-e^{-D \rho \left({1 \over \beta^\gamma}-1 \right) r^{\gamma-1}}  \right),
\end{equation}
and
\begin{eqnarray}
  \nonumber \int \displaylimits_{|\eta| > 3 |\delta|} J_3 (\eta,\delta)^p u_{\nu}(\eta)^p d \eta  &\le&   C  \ \| \mt{\psi} \|_{u_\nu,p}^{2 p} \int \displaylimits_{|\eta| > 3 |\delta|}  \left(1-e^{-D \rho \left({1 \over \beta^\gamma}-1 \right) r^{\gamma-1}}  \right) v_c(\eta)^p d \eta \\
  \nonumber &=& I_1 +I_2,
\end{eqnarray}
where
\begin{eqnarray}
  \nonumber   I_1 &\le&   C  \ \| \mt{\psi} \|_{u_\nu,p}^{2 p} \left({1 \over \beta^\gamma}-1 \right)^{p \over s}  \int \displaylimits_{r< \left({1 \over \beta^\gamma}-1 \right)^{-{1 \over \gamma}}}   \left(1-e^{-D \rho \left({1 \over \beta^\gamma}-1 \right) r^{\gamma-1}}  \right) r^{p (l(c)+c+\nu)+d-1} d r \\
  \nonumber  & \le &   C  \ \| \mt{\psi} \|_{u_\nu,p}^{2 p} \left( {1 \over \beta^\gamma}-1 \right)^{p \over s} \left({\left({1 \over \beta^\gamma}-1 \right)^{m-{p \over s}}  \over -\gamma \left(m-{p \over s} \right) }    - {\left( D \rho  \left( {1 \over \beta^\gamma}-1 \right)  \right)^{{\gamma \over \gamma-1} \left(m-{p \over s}  \right)  } \over \gamma-1 } \times \right.\\
  \nonumber && \hspace{30mm} \left. \times \gamma\left(-{\gamma \over \gamma-1} \left(m-{p \over s}    \right), D \rho  \left( {1 \over \beta^\gamma}-1  \right)^{1 \over \gamma}  \right) \right) \\
   \nonumber  & \le &   C  \ \| \mt{\psi} \|_{u_\nu,p}^{2 p} \left( {1 \over \beta^\gamma}-1 \right)^{p \over s} \left({\left({1 \over \beta^\gamma}-1 \right)^{m-{p \over s}} \over -\gamma \left(m-{p \over s} \right) }     - {\left( D \rho  \left( {1 \over \beta^\gamma}-1 \right)  \right)^{{\gamma \over \gamma-1} \left(m-{p \over s}  \right)  } \over \gamma-1 } \times \right.\\
   \nonumber && \hspace{30mm} \times \left({ \left(D \rho  \left( {1 \over \beta^\gamma}-1 \right)^{1 \over \gamma}  \right)^{-{\gamma \over \gamma -1} \left(m-{p \over s}   \right)   } \over -{\gamma \over \gamma-1} \left(m- {p \over s}\right) } + \right. \\
   \nonumber  && \hspace{55mm} \left. +O\left(\left(D \rho  \left( {1 \over \beta^\gamma}-1 \right)^{1 \over \gamma}  \right)^{-{\gamma \over \gamma -1} \left(m-{p \over s}   \right)   +1} \right) \right) \\
      \nonumber  & \le &   C  \ \| \mt{\psi} \|_{u_\nu,p}^{2 p}  \rho \left( {1 \over \beta^\gamma}-1 \right)^{m+{1 \over \gamma}},
\end{eqnarray}
and
\begin{eqnarray}
  \nonumber   I_2 &\le&   C  \ \| \mt{\psi} \|_{u_\nu,p}^{2 p}  \int \displaylimits_{|\eta| \ge  \left({1 \over \beta^\gamma}-1 \right)^{-{1 \over \gamma}}}   \left(1-e^{-D \rho \left({1 \over \beta^\gamma}-1 \right) r^{\gamma-1}}  \right) |\eta|^{p (l(c)+c+\nu)-{p \gamma \over s}+d-1} d \eta \\
  \nonumber  & \le &   C  \ \| \mt{\psi} \|_{u_\nu,p}^{2 p}  \left({\left({1 \over \beta^\gamma}-1 \right)^{m}  \over \gamma m }  - {\left( D \rho  \left( {1 \over \beta^\gamma}-1 \right)  \right)^{{\gamma \over \gamma-1} m  } \over \gamma-1 } \times \right.\\
  \nonumber && \hspace{40mm} \left. \times \Gamma\left(-{\gamma \over \gamma-1} m, D \rho  \left( {1 \over \beta^\gamma}-1  \right)^{1 \over \gamma}  \right) \right) \\
    \nonumber  & \le &   C  \ \| \mt{\psi} \|_{u_\nu,p}^{2 p}  \left({\left({1 \over \beta^\gamma}-1 \right)^{m}  \over \gamma m }  - {\left( D \rho  \left( {1 \over \beta^\gamma}-1 \right)  \right)^{{\gamma \over \gamma-1} m  } \over \gamma-1 } \times \right.\\
    \nonumber && \hspace{5mm}  \times \left( \Gamma\left(-{\gamma \over \gamma-1} m \right) - {\left(  D \rho  \left( {1 \over \beta^\gamma}-1  \right)^{1 \over \gamma}     \right)^{-{\gamma m \over \gamma-1} } \over  -{\gamma \over \gamma-1} m    }  + \right. \\
    \nonumber && \hspace{55mm} \left. + O\left( \left(  D \rho  \left( {1 \over \beta^\gamma}-1  \right)^{1 \over \gamma}     \right)^{1-{\gamma m \over \gamma-1}}  \right) \right) \\
     \nonumber  & \le &   C  \ \| \mt{\psi} \|_{u_\nu,p}^{2 p}  \left( {1 \over \beta^\gamma}-1  \right)^{\gamma m \over \gamma -1}   \rho^{\gamma m \over \gamma -1}. 
\end{eqnarray}
We obtain therefore, eventually,
\begin{equation}
  \label{J3estimate}   \| J_3(\eta,\delta) \|_{u_\nu,p} \le  C  \ \| \mt{\psi} \|_{u_\nu,p}^{2} \left(   \left( {1 \over \beta^\gamma}-1  \right)^{\vartheta}   |\delta|^{\vartheta} +   \left( {1 \over \beta^\gamma}-1  \right)^{{m \over p} +{1 \over \gamma p}}   |\delta|^{1 \over p}  \right),
 \end{equation}
where  $\omega^3_{\cR_\beta[\psi]} \in  R_{u_\nu}^p(\RR^d,\RR_+)$ and
\begin{equation}
\label{vartheta_def}  \vartheta= { \gamma m \over p (\gamma-1)}.
\end{equation}

\vspace{2mm}

\noindent \underline{\it Calculation of $J_4$}.  The final integral is $J_4$. We have
\begin{eqnarray}
  \nonumber \left| P_{\eta-\delta} \psi(x)-P_\eta \psi(x) \right| &=& \left| { \psi(x) \cdot \eta \over |\eta|^2} \eta- { \psi(x) \cdot (\eta-\delta) \over |\eta-\delta|^2} (\eta-\delta)  \right| \\
  \nonumber & \le & {4 |\psi(x)| |\eta| |\delta| + 2 |\psi(x)| |\delta|^2 \over |\eta-\delta|^2    } \\
   \nonumber & \le & {4 |\eta| + 2 |\delta| \over |\eta|-2 |\delta|    } {1 \over |\eta|} \ |\psi(x)| |\delta|.
\end{eqnarray}
We notice that $(4 |\eta| + 2 |\delta|)/(|\eta|-2 |\delta|)$ is bounded by some constant for all $|\eta| \ge 3 |\delta|$.
\begin{eqnarray}
\nonumber  J_4(\eta,\delta) & \le &  C  \ |\delta|  \ \int \displaylimits_{1}^{1 \over \beta} \int \displaylimits_{\RR^d}   \hspace{-1.0mm}{ e^{|\eta|^\gamma \left( 1-t^\gamma \right)} \over  t^{c} }  |\psi  \hspace{-0.5mm}  \left((\eta-\delta) t-x \right)| \ |\psi(x)|  \ d x \ d t.
\end{eqnarray}
This bound coincides with the bound $(\ref{J2first})$ on the integral $J_2$. Therefore,

\vspace{2mm}

We collect the estimates $(\ref{Lestimate5})$, $(\ref{J0estimate})$, $(\ref{J1estimate})$, $(\ref{J2estimate})$ and $(\ref{J3estimate})$. We have for $|\delta| < \beta \delta_0$:
\begin{align}
  \nonumber  \left| \cR_\beta[\psi](\eta)  - \cR_\beta[\psi](\eta-\delta)  \right|  &\le \lambda_{ \cR_\beta[\psi]}(\eta) |\delta|^\iota + \Omega^1_{ \cR_\beta[\psi]}(\eta) |\delta| + \Omega^2_{ \cR_\beta[\psi]}(\eta) |\delta|^\nu + \\
  \label{Requi} &  \phantom{ \le  \lambda^1_{ \cR_\beta[\psi]}(\eta) |\delta|^\iota}  \hspace{1.1mm} + \Omega^3_{ \cR_\beta[\psi]}(\eta) |\delta|^{1 \over p} +\Omega^4_{ \cR_\beta[\psi]}(\eta) |\delta|^{\vartheta} + \Omega^5_{ \cR_\beta[\psi]}(\eta) |\delta|^{\theta},
\end{align}
where $\iota=\min\{1, \gamma \}$, and
\begin{eqnarray}
    \nonumber  \| \lambda_{ \cR_\beta[\psi]} \|_{u_\nu,p}  & \le&  C A  \left(1-\beta \right),\\
  \nonumber  \|  \Omega^1_{ \cR_\beta[\psi]} \|_{u_\nu,p}  & \le&  C  \left({1 \over \beta^\gamma}-1\right)^{1 \over s} \hspace{-1.5mm} A^2 + C  \left({1 \over \beta^\gamma}-1\right)^{{m \over p} +{1 \over \gamma}}  \hspace{-2mm} A^2, \\
  \nonumber  \|  \Omega^2_{ \cR_\beta[\psi]} \|_{u_\nu,p}  & \le&  C \ \left({1 \over \beta^\gamma}-1\right)^{{1 \over p}+{\nu \over \gamma}} A^2, \\
    \nonumber  \|  \Omega^3_{ \cR_\beta[\psi]} \|_{u_\nu,p}  & \le&  C \ \left({1 \over \beta^\gamma}-1\right)^{{m \over p} +{1 \over \gamma}} A^2, \\
    \nonumber  \|  \Omega^4_{ \cR_\beta[\psi]} \|_{u_\nu,p}  & \le&  C \ \left({1 \over \beta^\gamma}-1\right)^{\vartheta} A^2, \\
    \nonumber  \|  \Omega^5_{ \cR_\beta[\psi]} \|_{u_\nu,p}  & \le& \beta^{c+\nu+{d \over p}-\theta} K + C_2 \ \left({1 \over \beta^\gamma}-1\right)^{m \over p} A K.
\end{eqnarray}
We will simplify the above bounds, by factoring $|\delta| \le |\delta|^\theta |\delta_0|^{1-\theta}$, and absorbing $|\delta_0|^{1-\theta}$ into a constant in the estimate for $\Omega^1$. Then, taking into account $(\ref{bound8})$, we get
\begin{equation}  \label{Requi1}  \left| \cR_\beta[\psi](\eta)  - \cR_\beta[\psi](\eta-\delta)  \right| \le \omega_{ \cR_\beta[\psi]}(\eta)  |\delta|^\theta,
\end{equation}
where
$$\theta=\min\left\{\iota, \vartheta, \nu, {1 \over p}  \right\},$$
and, using the notation, $n=c+\nu+d/p$,
\begin{align}
  \nonumber  \| \omega_{ \cR_\beta[\psi]}(\eta) \|_{u_\nu,p}  & \le C \ A^2 \left\{\left({1 \over \beta^\gamma}-1\right)^{{1 \over s}}  \hspace{-1mm}+ \left({1 \over \beta^\gamma}-1\right)^{{1 \over p}+{\nu \over \gamma}}  + \left({1 \over \beta^\gamma}-1\right)^{{m \over p} +{1 \over \gamma}} + \left({1 \over \beta^\gamma}-1\right)^{\vartheta}  \right\} \\
  \nonumber & \hspace{39.2mm}  + C A (1-\beta) +  \left( \beta^{n-\theta} + C_2 A \left({1 \over \beta^\gamma}-1  \right)^{m \over p} \right) K.
\end{align}
Bound $(\ref{Requi1})$ is less or equal to $K |\delta|^\theta$ if
\begin{equation}
  \label{Kineq}  K \ge   {C A   (1- \beta)^k \over 1  -  \beta^{n \shortminus \theta}  -  C_2 A \left({1 \over \beta^\gamma}  -  1  \right)^{m \over p} },
\end{equation}
for some $C$, and a $k=\min\{1, 1/s,1/p+\nu/\gamma, m/p +1/\gamma, \vartheta \}$.

\medskip

\noindent \textit{\textbf{3)  Case $\bm{\beta \delta_0 \le |\delta|< \delta_0}$.}} 

To extend the conclusion to  $|\delta| \in [\beta \delta_0,\delta_0  )$, we remark that $\|\psi - \tau_\delta \psi \|_{u_\nu} \le 2 A$ for all $\psi \in \cE_{u_\nu,p,A}$. Therefore, for $|\delta| \in [\beta \delta_0,\delta_0  )$,
  $$\|\psi - \tau_\delta \psi \|_{u_\nu} \le 2 {A \over \beta^\theta \delta_0^\theta} |\delta|^\theta,$$
and the claim of the Proposition follows by taking $K$ as the maximum of $K$ satisfying $(\ref{Kineq})$ and $2 A/ \beta^\theta \delta_0^\theta$.  

\medskip

{\flushright $\square$}

\bibliography{biblio}
\bibliographystyle{plain}

\end{document}